\documentclass[twoside,11pt]{article}

\usepackage{blindtext}

\usepackage{jmlr2e}

\RequirePackage{amsmath,amsfonts,amssymb}

\RequirePackage{graphicx}
\RequirePackage{booktabs}
\RequirePackage{colortbl}

\RequirePackage[textsize=tiny]{todonotes}

\usepackage{xcolor}
\usepackage{amsbsy}
\usepackage{subcaption}
\usepackage[shortlabels]{enumitem}

\usepackage{bbm}
\usepackage{enumitem}
\usepackage{mathtools}
\usepackage{amsmath}
\usepackage{mathrsfs}
\usepackage{physics}
\usepackage[title,page]{appendix}
\usepackage{graphicx}
\usepackage[capitalize]{cleveref} 
\crefname{assumption}{Assumption}{Assumptions}
\Crefname{condition}{Condition}{Conditions}
\Crefname{remark}{Remark}{Remarks}

\usepackage{natbib}
\usepackage{makecell}
\DeclareMathOperator*{\argmax}{arg\,max}

\allowdisplaybreaks

\numberwithin{equation}{section}

\newcommand{\postparam}{\boldsymbol{\tilde{\theta}}_n}
\newcommand{\fisherinformation}{\hat{J}_n}
\newcommand{\postfisherinformation}{\bar{J}_n}
\newcommand{\param}{\theta}
\newcommand{\logposterior}{\overline{L}_n}
\newcommand{\posterior}{\Pi_n}
\newcommand{\mle}{\hat{\theta}_n}
\newcommand{\map}{\bar{\theta}_n}
\newcommand{\prior}{\pi}
\newcommand{\loglikelihood}{L_n}

\def\kappabar{\bar{\kappa}}
\def\kappamle{\kappa}
\def\deltamle{\delta}
\def\deltabar{\bar{\delta}}

\def\jbar{\postfisherinformation(\map)}
\def\jhat{\fisherinformation(\mle)}

\def\jbarplus{\bar{J}^{p}_{n}(\map,\deltabar)}
\def\jhatplus{\hat{J}^{p}_{n}(\mle,\deltamle)}
\def\jbarminus{\bar{J}^{m}_{n}(\map,\deltabar)}
\def\jhatminus{\hat{J}^{m}_{n}(\mle,\deltamle)}

\def\trinv#1{\Tr\left[{#1}^{-1}\right]}
\def\jbarinvtr{\trinv{\postfisherinformation(\map)}}
\def\jbarminusinvtr{\trinv{\jbarminus}}
\def\jbarplusinvtr{\trinv{\jbarplus}}

\def\jhatinvtr{\trinv{\fisherinformation(\mle)}}
\def\jhatminusinvtr{\trinv{\jhatminus}}
\def\jhatplusinvtr{\trinv{\jhatplus}}

\def\det#1{\left|\mathrm{det}\left(#1\right)\right|}

\def\jbarminusdet{\det{\jbarminus}}
\def\jbarplusdet{\det{\jbarplus}}

\def\jhatminusdet{\det{\jhatminus}}
\def\jhatplusdet{\det{\jhatplus}}

\def\decaybar{\bar{\mathscr{D}}(n, \deltabar)}

\def\decaybarplus{\bar{\mathscr{D}}^{p}(n, \deltabar)}

\def\decayhat{\hat{\mathscr{D}}(n, \deltamle)}

\def\decayhatplus{\hat{\mathscr{D}}^{p}(n, \deltamle)}

\newcommand{\minevaluemle}{\lambda_{\min}(\mle)}
\newcommand{\minevaluemap}{\overline{\lambda}_{\min}(\map)}
\def\minevalmapplus{
    \overline{\lambda}^{p}_{\min}(\map, \deltabar)}
\def\minevalmapminus{
    \overline{\lambda}^{m}_{\min}(\map, \deltabar)}
\def\minevalmleplus{\hat{\lambda}^{p}_{\min}(\mle, \deltamle)}
\def\minevalmleminus{\hat{\lambda}^{m}_{\min}(\mle, \deltamle)}

\def\m#1{M_{#1}}
\def\mtil#1{\widetilde{M}_{#1}}
\def\mbar#1{\overline{M}_{#1}}
\def\mhat#1{\hat{M}_{#1}}

\def\MainMapThm{
\begin{theorem}\label{main_map}
Fix $n\in\mathbbm{N}$, suppose that \Cref{assump1,assump_prior1,assump2,assump_size_of_delta_bar,assump7,assump_kappa1} hold and retain the notation thereof.
Let TV denote the total variation distance. Then:
\begin{align*}
    \text{TV}\left(
        \mathcal{L}\left(
            \sqrt{n}\left(\postparam - \map\right)\right),
            \mathcal{N}(0,\postfisherinformation(\map)^{-1})
    \right)
\leq{}
    A_1 n^{-1/2} +
    2 \decaybar
        +
    A_2n^{d/2}e^{-n\kappabar},
\end{align*}
where
\begin{align*}
		&A_1=\frac{\sqrt{3}\, \jbarinvtr \mbar{2}}
        {4\sqrt{\left(\minevaluemap - \deltabar\, \mbar{2} \right)
                \left( 1- \decaybar \right)}}; \qquad A_2=\frac{ 2
         \det{\jhatplus} ^{1/2} \mhat{1}}
        {\left(2\pi \right)^{d/2} \left( 1-\decayhatplus\right)}.
\end{align*}
\end{theorem}
}

\def\WassertsteinMAPThm{
\begin{theorem}\label{thm:wasserstein_map}
	Fix $n\in\mathbbm{N}$, suppose that \Cref{assump1,assump_prior1,assump2,assump_size_of_delta_bar,assump7,assump_kappa1} hold and retain the notation thereof.
Then:
    \begin{align*}
    \MoveEqLeft
    W_1\left(
        \mathcal{L}\left(
                \sqrt{n}\left(\postparam-\map\right)
            \right),
            \mathcal{N}(0,\postfisherinformation(\map)^{-1})\right)
    \\& \leq
        B_1 n^{-1/2} +
        B_3 \left( B_2+\sqrt{n}\, B_4\right)n^{d/2}e^{-n\kappabar}
         + \left(
                \deltabar \sqrt{n} +
                \sqrt{\frac{2\pi}{\minevaluemap}} + B_2
            \right) \decaybar,
\end{align*}
where
\begin{align*}
    B_1 :={}& \frac{\sqrt{3} \jbarinvtr \mbar{2}}
               {2 \left(\minevaluemap - \deltabar \mbar{2} \right)
               \sqrt{1- \decaybar}};\\
    B_2 :={}&
        \frac{\jbarplusdet^{1/2} \jbarminusdet^{-1/2} \sqrt{\jbarminusinvtr}}
             {1 - \decaybarplus};
    \\
    B_3 :={}& \frac
        { \det{\jhatplus} ^{1/2} \mhat{1}}
        {\left(2\pi \right)^{d/2} \left( 1-\decayhatplus\right)};\qquad
    B_4 :={} \int_{
        \|u-\hat{\theta}_n\| > \deltabar - \|\hat{\theta}_n-\bar{\theta}_n\|}
        \|u-\map\|\pi(u)du.
    \end{align*}
\end{theorem}
}

\def\WTwoMAPThm{
\begin{theorem}\label{w2_map}
Fix $n\in\mathbbm{N}$, suppose that \Cref{assump1,assump_prior1,assump2,assump_size_of_delta_bar,assump7,assump_kappa1} hold and retain the notation thereof. Let $\boldsymbol{Z}_n\sim\mathcal{N}\left(0,\postfisherinformation(\map)^{-1}\right)$	Then:
\begin{align*}
\MoveEqLeft
\sup_{v:\|v\|\leq 1}
\left|\mathbbm{E}\left[
    \left<v,\sqrt{n}\left(\postparam-\map\right)\right>^2
\right] -
\mathbbm{E}\left[\left<v,\boldsymbol{Z}_n\right>^2\right]\right|\\
\leq & C_1  n^{-1/2} +
    \frac{\sqrt{3 \jbarinvtr } C_1 \mbar{2}}
        {4\left( \minevaluemap - \deltabar\,\mbar{2} \right)} \,
        n^{-1} 
\\&+ C_3\left( C_2 + n \,C_4\right)\, n^{d/2} e^{-n\kappabar} +  \left(
    \deltabar^2n + \frac{\sqrt{2\pi}}{\minevaluemap} + C_2
\right) \decaybar,
\end{align*}
for
\begin{align*}
C_1:={}&\frac{\sqrt{3} \left(\jbarinvtr\right)^{3/2} \mbar{2}}
             {\left(\minevaluemap -\deltabar \mbar{2}\right)
              \left(1 - \decaybar \right)}
\\
C_2 :={}&\frac{\jbarplusdet^{1/2} \jbarminusdet^{-1/2} \jbarminusinvtr}
              {1 - \decaybarplus}
\\
C_3 :={}&\frac{\jhatplusdet^{1/2} \mhat{1}}
              {(2 \pi)^{d/2}\left(1 - \decayhatplus \right)};\qquad
C_4:={} \int_{\|v-\mle\|>\deltabar-\|\mle-\map\|}
    \|u-\map\|^2\pi(u)du.
\end{align*}
\end{theorem}
}

\def\MainThm{
\begin{theorem}\label{theorem_main}
	Fix $n\in\mathbb{N}$ and suppose that \Cref{assump1,assump_prior1,assump_size_of_delta,assump_fisher,assump_kappa,assump_prior2} hold. Let TV denote the total variation distance. We have the following upper bound:
\begin{align*}
\MoveEqLeft
    \text{TV}\left(
        \mathcal{L}\left(
            \sqrt{n}\left(\postparam-\mle\right)\right),
            \mathcal{N}(0,\fisherinformation(\mle)^{-1})
        \right)
    \leq D_1 n^{-1/2} + D_2 n^{d/2} e^{-n\kappamle} + 2 \decayhat,
\end{align*}
where
\begin{align*}
D_1 :={}&
\frac{\sqrt{\mtil{1} \mhat{1}}}
     {2\sqrt{\minevaluemle - \deltamle \m{2}}}
\left(
    \frac{\sqrt{3} \jhatinvtr  }{2\sqrt{1 - \decayhat}} \m{2} +
    \m{1}
\right);
\\
D_2 :={}& \frac
    {2 \mhat{1} \jhatplusdet^{1/2}}
    {(2 \pi)^{d/2} \left(1 - \decayhatplus \right)}.
\end{align*}
\end{theorem}
}

\def\OneWassThm{
\begin{theorem}\label{theorem_1wasserstein}
	Fix $n\in\mathbbm{N}$ and suppose that \Cref{assump1,assump_prior1,assump_size_of_delta,assump_fisher,assump_kappa,assump_prior2} hold.  We have the following upper bound:
\begin{align*}
\MoveEqLeft
W_1\left(
    \mathcal{L}\left(
        \sqrt{n}\left(\postparam-\mle\right)\right),
        \mathcal{N}(0,\fisherinformation(\mle)^{-1})
        \right)
\\&\leq
E_1 n^{-1/2} +
E_3 \left( E_2 + \sqrt{n} E_4 \right) n^{d/2} e^{-n\kappamle}
+ \left(\deltamle\sqrt{n} + \sqrt{\frac{2\pi}{\minevaluemle}} +
E_2\right)\decayhat,
\end{align*}
where

\begin{align*}
E_1 :={}&
\frac{\mtil{1} \mhat{1}}{\minevaluemle -\deltamle M_2} \left(
    \frac{\sqrt{3} \jhatinvtr }{2 \sqrt{1 - \decayhat} } \m{2}
    + \m{1}  \right)
\\
E_2 :={}&\frac
    {\mhat{1} \mtil{1} \jhatplusdet^{1/2} \jhatminusdet^{-1/2}
     \sqrt{\jhatminusinvtr}}
    {1 - \decayhatplus}
\\
E_3 ={}&\frac
    {\mhat{1} \jhatplusdet^{1/2}}
    {(2 \pi)^{d/2} \left(1 - \decayhatplus \right)};\qquad E_4:={}\int_{\|u-\mle\|>\deltamle} \|u-\mle\| \pi(u)du.
\end{align*}

\end{theorem}
}

\def\TwoWassThm{
\begin{theorem}\label{theorem_2wasserstein}
		Fix $n\in\mathbbm{N}$ and suppose that \Cref{assump1,assump_prior1,assump_size_of_delta,assump_fisher,assump_kappa,assump_prior2} hold. Let $\boldsymbol{Z}_n\sim \mathcal{N}(0,\fisherinformation(\mle)^{-1})$. We have that

\begin{align*}
\MoveEqLeft
\sup_{v:\|v\|\leq 1} \left|
    \mathbbm{E}\left[
        \left<v,\sqrt{n}\left(\postparam-\mle\right)\right>^2
    \right] -
    \mathbbm{E}\left[
        \left<v,\boldsymbol{Z}_n\right>^2\right]
    \right|
\\ & \leq
    \left(F_1\right)^2 n^{-1} + F_1 F_2 n^{-1/2} +
    \frac{F_5 (F_3)^2 + F_3 F_4 \sqrt{n}} {(2\pi)^{d/2}}
         n^{d/2} e^{-n\kappamle}
\\&
     + \left(
        \deltamle^2 n + \sqrt{\frac{2\pi}{\minevaluemle}} + F_3 F_5
    \right) \decayhat
\end{align*}
where
\begin{align*}
F_1 :={}&
\frac{\mtil{1} \mhat{1}}{\minevaluemle - \deltamle \m{2}}\left(
    \frac{\sqrt{3} \jhatinvtr } {\sqrt{1 - \decayhat}}  \m{2} + \m{1}
\right);\qquad
F_2 :={} \frac{2\sqrt{\jhatinvtr}}{\sqrt{1 - \decayhat}};
\\
F_3 :={}& \frac{\mhat{1} \jhatplusdet^{1/2}}{1 - \decayhatplus};\qquad
F_4 :={} \int_{\|u\|>\deltamle} \|u\|^2 \pi(u + \hat{\theta}_n) du;
\\
F_5:={}& \mtil{1} \jhatminusdet^{-1/2} \jhatminusinvtr.
\end{align*}

\end{theorem}
}

\def\UnivariateThm{
\begin{theorem}\label{theorem_univariate}
	Assume that we study a \textbf{univariate posterior}, i.e. that ${d=1}$. Let $\sigma_n^2:=\fisherinformation(\mle)^{-1}$. Suppose that \Cref{assump1,assump_prior1,assump_fisher,assump_kappa,assump_prior2} hold. Let $\boldsymbol{Z}_n\sim\mathcal{N}(0,\sigma_n^2)$.  Then, for any function $g:\mathbbm{R}\to\mathbbm{R}$ which is integrable with respect to the posterior and with respect to $\mathcal{N}(0,\sigma_n^2)$,
	\begin{align*}
		&\left|\mathbbm{E}\left[g\left(\sqrt{n}(\postparam-\hat{\theta}_n)\right)\right]-\mathbbm{E}\left[g(\boldsymbol{Z}_n)\right]\right|\\
		\leq& (G_1+G_2) n^{-1/2}+\left|\int_{|u|>\deltamle\sqrt{n}}g(u)\frac{e^{-u^2/(2\sigma_n^2)}}{\sqrt{2\pi\sigma_n^2}}du\right|\\
		&+\left(G_3\int_{
				|u|>\deltamle}|g(u\sqrt{n})|\pi(u+\mle)du +G_4 \right) n^{1/2}e^{-n\kappamle}+G_5e^{-\deltamle^2n/(2\sigma_n^2)},
	\end{align*}
	for
	\begin{align*}
		&C_n^{(1)}:=\frac{1}{2\sigma_n^2}-\frac{\deltamle M_2}{3}>0;\qquad C_n^{(2)}:=\left(\frac{1}{2\sigma_n^2}-\frac{\deltamle M_2}{6}\right)>0;\\
		&C_n^{(3)}:=\left(\frac{1}{2\sigma_n^2}+\frac{\deltamle M_2}{3}\right)>0;\qquad
		C_n^{(4)}:=\left(\frac{1}{2\sigma_n^2}+\frac{\deltamle M_2}{6}\right)>0
	\end{align*}
	and
	\begin{align*}
	    &G_1:=\frac{2\widetilde{M}_1\hat{M}_1}{\sqrt{2\pi\sigma_n^2}}\int_{-\deltamle\sqrt{n}}^{\deltamle\sqrt{n}}|ug(u)|\left[\left(M_1+\frac{3}{\deltamle}\right)e^{-C_n^{(1)}u^2}-\frac{3}{\deltamle}e^{-C_n^{(2)}u^2}\right]du;\\
	    &G_2:=\frac{2\sqrt{C_n^{(4)}}\left(\widetilde{M}_1\hat{M}_1\right)^2\left(M_1+\frac{3}{\deltamle}\right)\int_{-\deltamle\sqrt{n}}^{\deltamle\sqrt{n}}|g(u)|e^{-C_n^{(2)}u^2}du}{C_n^{(1)}\pi\sqrt{\sigma_n^2}\left(1-2e^{-\deltamle^2nC_n^{(4)}}\right)}\left(\frac{M_1+\frac{3}{\deltamle}}{C_n^{(1)}}-\frac{3}{\deltamle C_n^{(2)}}\right);\\
	    &G_3:=\frac{ \hat{M}_1\sqrt{C_n^{(4)}}}{\sqrt{\pi} \left\{1-2\exp\left[-C_n^{(4)}\deltamle^2n\right]\right\}};\\
		&G_4:=\frac{ \hat{M}_1^2\widetilde{M}_1C_n^{(4)}\int_{|t|\leq\deltamle\sqrt{n}}|g(t)|e^{-C_n^{(2)}t^2}dt}{\pi\,  \left\{ 1-2\exp\left[-C_n^{(4)}\deltamle^2n\right]\right\}^2};\\
		&G_5:=\frac{2 \hat{M}_1\widetilde{M}_1\sqrt{C_n^{(4)}}\int_{|t|\leq\deltamle\sqrt{n}}|g(t)|e^{-C_n^{(2)}t^2}dt}{\sqrt{\pi}\,  \left\{ 1-2\exp\left[-C_n^{(4)}\deltamle^2n\right]\right\}}.
	\end{align*}
\end{theorem}
}

\newtheorem{assumption}{Assumption}

\usepackage{lastpage}
\jmlrheading{-}{-}{1-\pageref{LastPage}}{-; Revised -}{-}{00-0000}{Mikołaj J. Kasprzak, Ryan Giordano and Tamara Broderick}

\ShortHeadings{How good is your Laplace approximation?}{M.J. Kasprzak, R. Giordano and T. Broderick}
\firstpageno{1}

\usepackage{lastpage}
\jmlrheading{26}{2025}{1-\pageref{LastPage}}{4/24; Revised
2/25}{5/25}{24-0619}{Mikołaj J. Kasprzak, Ryan Giordano, Tamara Broderick}
\ShortHeadings{How good is your Laplace approximation?}{Kasprzak, Giordano, Broderick}

\begin{document}

\title{How good is your Laplace approximation of the Bayesian posterior?  Finite-sample computable error bounds for a variety of useful divergences}

\author{\name Mikołaj J. Kasprzak \email kasprzak@essec.edu \\
       \addr Department of Information Systems, Data Analytics and Operations\\
       ESSEC Business School\\
       3 Avenue Bernard Hirsch\\
       95021 Cergy-Pontoise Cedex, France
       \AND
       \name Ryan Giordano \email rgiordano@berkeley.edu \\
       \addr Department of Statistics\\
       University of California, Berkeley\\
       367 Evans Hall\\
       Berkeley, CA 94720-3860, USA
       \AND
       \name Tamara Broderick \email tamarab@mit.edu \\
       \addr Laboratory for Information and Decision Systems\\
       Massachusetts Institute of Technology\\
       32 Vassar Street, 32-D608\\
       Cambridge, MA 02139, USA}

\editor{Matthew Hoffman}

\maketitle

\begin{abstract}
The Laplace approximation is a popular method for constructing a Gaussian approximation to the Bayesian posterior and thereby approximating the posterior mean and variance. But approximation quality is a concern. One might consider using rate-of-convergence bounds from certain versions of the Bayesian Central Limit Theorem (BCLT) to provide quality guarantees. But existing bounds require assumptions that are unrealistic even for relatively simple real-life Bayesian analyses; more specifically, existing bounds either (1) require knowing the true data-generating parameter, (2) are asymptotic in the number of samples, (3) do not control the Bayesian posterior mean, or (4) require strongly log concave models to compute. In this work, we provide the first computable bounds on quality that simultaneously (1) do not require knowing the true parameter, (2) apply to finite samples, (3) control posterior means and variances, and (4) apply generally to models that satisfy the conditions of the asymptotic BCLT. Moreover, we substantially improve the dimension dependence of existing bounds; in fact, we achieve the lowest-order dimension dependence possible in the general case. We compute exact constants in our bounds for a variety of standard models, including logistic regression, and numerically demonstrate their utility. We provide a framework for analysis of more complex models.
\end{abstract}


\section{Introduction}

\subsection{Motivation}
Bayesian inference is widely used in modern statistical practice to provide both point estimates and uncertainties of unknown quantities; \citet{abbott2023population,flaxman2020estimating,freedman2021measurements,jones2021estimating} give just a few especially influential recent examples. In applications, practitioners typically report posterior means and (co)variances. These expectations under the posterior distribution are, however, often intractable
to compute, so practitioners must use approximations. The \textit{Laplace approximation} is popular in many communities (e.g., \citealt[Section 4.4]{bishop}; \citealt[Section 4.6.8.2]{murphy};
\citealt{drton,thygesen2017validation,gomez2021spatial,ritter2018online,riihimaki2014laplace,long2013fast,wang2018efficient}) in large part due to its ease of use and computational speed. It approximates the Bayesian posterior by a suitably chosen Gaussian distribution and is grounded in the celebrated \textit{Bernstein--von
Mises theorem}. The Bernstein--von
Mises theorem is often colloquially referred to as the Bayesian Central Limit Theorem, or BCLT; see e.g.\ the recent work of \citet{goehle}. Studies have shown the appealing empirical performance of the Laplace approximation, for
instance in the context of Bayesian neural networks \citep{daxberger}. 

However, approximation quality remains a concern. If users could check the quality of their approximation directly, with high confidence, they could better decide whether to trust any conclusions based on the approximation. And if the approximation could not be certified, the user might invest further computational power into a more expensive approximation. Notably, the BCLT on its own cannot be expected to provide such a check. For instance, convergence in the BCLT is normally
expressed in terms of the total variation distance, and indeed the Laplace
approximation is typically justified by the fact that the total variation
distance between the rescaled posterior and the Gaussian vanishes in the limit.
The total variation distance, however, does not control the difference of means
or the difference of covariances in general. But it is exactly posterior means and (co)variances that are most often reported by users of approximate Bayesian inference.
We therefore need finite-sample guarantees on the
quality of Laplace approximation, expressed in terms of metrics that control the error in mean
and covariance approximation.

\subsection{Our contribution}

Our work provides an important step forward toward achieving fully practical and ready-to-use quality guarantees.
Namely, our work provides the first bounds for the Laplace approximation that simultaneously
\begin{itemize}
    \item do not require knowing the true parameter in advance,
    \item apply to finite data sets (i.e., not just in the limit of infinite data),
    \item control the difference in means and the difference in covariances,
    \item do not require the posterior to be log concave.
\end{itemize}
We achieve control over difference in means by controlling the 1-Wasserstein distance, and we achieve control over the difference in covariances via another integral probability metric.
We are also able to control the total variance distance. Moreover, our bounds are closed form. 

More specifically, let $n$ be the sample size and $d$ be the dimension. Let $\map$ denote the maximum a posteriori (MAP) estimate. Let $\postparam$ denote a random variable distributed according to the posterior. Let $\postfisherinformation(\map)$ denote the \textit{posterior empirical Fisher information at the MAP}. Let $\nu^{\pi}:=\mathcal{L}\left(\sqrt{n}\left(\postparam - \map\right)\right)$ (where $\mathcal{L}$ denotes the law) and $\nu^{\mathcal{N}}:=\mathcal{N}(0,\postfisherinformation(\map)^{-1})$. Our main bounds, presented in full rigor in \Cref{s:main}, may be (very informally and in a very simplified way) described as taking the following form:
\begin{align*}
&TV\left(
        \nu^{\pi},\nu^{\mathcal{N}}
    \right)\leq{}
    H_d d n^{-1/2} +
    \exp\left(-\overline{H}_d\left(\sqrt{n}-\sqrt{d}\right)^2\right)        +
    \hat{H}_dn^{d/2}e^{-n\kappabar},\\
    &W_1\left(
        \nu^{\pi},\nu^{\mathcal{N}}
    \right)\leq{}
    H_d d n^{-1/2} +
    \sqrt{n}\exp\left(-\overline{H}_d\left(\sqrt{n}-\sqrt{d}\right)^2\right)        +
    \hat{H}_dn^{d/2+1/2}e^{-n\kappabar},\\
    &D\left(
        \nu^{\pi},\nu^{\mathcal{N}}
    \right)\leq{}
    H_d d^{3/2} n^{-1/2} +
    n\exp\left(-\overline{H}_d\left(\sqrt{n}-\sqrt{d}\right)^2\right)        +
    \hat{H}_dn^{d/2+1}e^{-n\kappabar},
\end{align*}
for appropriate model-dependent factors $H_d,\overline{H}_d,\hat{H}_d,\bar{\kappa}$. Here $TV$ is the total variation distance, $W_1$ is the 1-Wasserstein distance, and $D$ is an appropriate distance that controls the difference of covariances. All the constants and factors in our bounds are explicit and presented in full detail in \Cref{s:main}.

The sample-size and dimension dependence of our bounds on the total-variation and the $1$-Wasserstein distance are tight and such that they \textit{cannot be improved in general}, as we discuss in \Cref{dimension_dependence}. Our bounds' dimension and sample-size dependence are also \textit{strictly better} than those appearing in all the previous works on the Laplace approximation, despite our results holding under weaker assumptions than those of the previous works.

We highlight that our bounds do not require access to the true parameter or exact integrals with respect to the posterior, which would be unrealistic for practical checks. Rather, our bounds are expressed in terms of the data and work under any distribution of the data; importantly, our bounds apply when the model is misspecified, as we expect to be the case in practice. Our results are fully applicable to models involving generalized likelihoods and the resulting generalized posteriors; see \cite{bissiri, miller, chernozhukov}.
Our assumptions on the generalized likelihood and the prior are standard and no stronger than the assumptions of the classical proofs of the Bernstein--von Mises Theorem; see e.g.\ \citet[Section 1.4]{ghosh} or \citet{miller} for details. We compute our bounds explicitly for a variety of Bayesian models, including logistic regression with a Student's t prior.

Our contribution lies also in our proof techniques. In order to control the discrepancy inside a ball around the maximum likelihood estimator (MLE) or the maximum a posteriori (MAP) estimator, we use the log-Sobolev inequality or Stein's method. In order to control the discrepancy over the rest of the parameter space, we carefully bound the tail growth using standard assumptions of the Bernstein--von Mises Theorem.

A number of important directions remain for future work. Though beyond the scope of the present paper, we believe that our approach to proving computable non-asymptotic bounds could be extended so as to cover more general statistical models satisfying the conditions of the local asymptotic normality (LAN) theory \citep[Chapters 1--3]{statistical_estimation}. We also note that, while our experiments demonstrate that we can use our bounds to make non-vacuous conclusions about real-data analyses, exact computation of our bounds requires analytical derivations and a series of numerical optimization procedures. Their current instantiations can become onerous as models increase in complexity and size. We believe our current work, though, points to fruitful directions to continue building these tools.

\subsection{Related work} \label{related_work}
\small


\makeatletter
\newcommand*{\hyperlinkcite}[1]{\hyper@link{cite}{cite.#1}}
\makeatother

\begin{table}

\caption{\label{tab:related_results}Very brief summary of related results.  See section \ref{related_work} for full details.}
\centering
\begin{tabular}[t]
{c|c|c|c|c|c|c|c}
\toprule
 & \makecell{Present \\ work} & \makecell{\hyperlinkcite{spokoiny_panov}{PS} \\ \hyperlinkcite{spokoiny_panov}{15}} & \makecell{\hyperlinkcite{helin}{HK} \\ \hyperlinkcite{helin}{22}} & \makecell{\hyperlinkcite{huggins}{H+} \\ \hyperlinkcite{huggins}{18}} & \makecell{\hyperlinkcite{dehaene}{D} \\ \hyperlinkcite{dehaene}{19}} & \makecell{\hyperlinkcite{spokoiny_laplace}{S} \\ \hyperlinkcite{spokoiny_laplace}{22}} & \makecell{\hyperlinkcite{yano_kato}{YK} \\ \hyperlinkcite{yano_kato}{20}} \\
\midrule
\cellcolor{gray!6}{True parameter not needed} & \cellcolor{gray!6}{\checkmark} & \cellcolor{gray!6}{--} & \cellcolor{gray!6}{ \checkmark} & \cellcolor{gray!6}{ \checkmark} & \cellcolor{gray!6}{\checkmark} & \cellcolor{gray!6}{ \checkmark} & \cellcolor{gray!6}{--}\\
Global log-concavity not needed & \checkmark&  \checkmark & -- & -- & -- & -- & \checkmark\\
\cellcolor{gray!6}{Generic priors / likelihoods} & \cellcolor{gray!6}{\checkmark} & \cellcolor{gray!6}{--} & \cellcolor{gray!6}{--} & \cellcolor{gray!6}{ \checkmark} & \cellcolor{gray!6}{\checkmark} & \cellcolor{gray!6}{--} & \cellcolor{gray!6}{--}\\
Explicit bounds (no order notation) & \checkmark & \checkmark & \checkmark & \checkmark & -- & \checkmark & \checkmark\\
\cellcolor{gray!6}{Controls TV distance} & \cellcolor{gray!6}{\checkmark} & \cellcolor{gray!6}{\checkmark} & \cellcolor{gray!6}{\checkmark} & \cellcolor{gray!6}{--} & \cellcolor{gray!6}{\checkmark} & \cellcolor{gray!6}{\checkmark} & \cellcolor{gray!6}{\checkmark}\\
Controls means and variances & \checkmark & \checkmark & -- & \checkmark & -- & -- & --\\
\bottomrule
\end{tabular}
\end{table}
\normalsize
A number of recent papers have studied the non-asymptotic properties of the Laplace approximation and the Bernstein--von Mises theorem for Bayesian posteriors.
We next contrast the present work with these analyses in terms of the computability of the bounds,
strength of the results, and restrictiveness of the assumptions.
A brief summary of the differences can be found in \Cref{tab:related_results}.

Recently, \citet[Section 6.1]{huggins}, \citet{dehaene}, and
\citet{spokoiny_laplace} have offered ways of obtaining guarantees on the quality
of Laplace approximation under log-concavity of the posterior. In Proposition 6.1,
\citet{huggins} assume that the posterior is
strongly log-concave and obtain a computable bound on the 1- and 2-Wasserstein
distances between the posterior and the approximating Gaussian. In Proposition 6.2,
\citet{huggins} relax the assumption on the
posterior to weak log-concavity, but the bound they obtain is not computable in
practice, for finite data. \citet{dehaene} assumes weak, yet
strict, log-concavity of the posterior. In the paper's main result (Theorem 5), \citet{dehaene} obtains a bound on the Kullback-Leibler (KL) divergence between
the posterior and the Gaussian approximation. In this bound, only the leading
terms are computable, while the higher order terms are presented using the big-O
notation. 
\citet{spokoiny_laplace} assumes a Gaussian prior and a
log-concave likelihood, which yields a strongly log-concave posterior. The author
relaxes this assumption only in a very specific case of a nonlinear inverse
problem with a certain ``warm start condition.'' The bounds of \citet{spokoiny_laplace} are on the total
variation distance. The author also provides a bound on the difference of means -- yet
this bound involves generic abstract constants and it is not clear to us whether
it is computable in applications. \citet{schillings} have proved a
qualitative, asymptotic result about convergence of the Laplace approximation
for inverse problems. Motivated by this work, \citet{helin} have provided
non-asymptotic bounds for the Laplace approximation on the total variation
distance -- in the context of Bayesian inverse problems and only under the
assumption that the posterior distribution is sub-Gaussian. Moreover, in
concurrent work, \citet{fisher} have provided bounds on the Gaussian approximation
of the posterior, yet only for i.i.d.\ data coming from a regular $k$-parameter
exponential family. The dimension dependence of the bounds of \citet{helin,spokoiny_laplace,dehaene,fisher} is worse than ours, as we describe in \Cref{dimension_dependence}. Additionally, the recent paper by \citet{hasenpflug} proved
interesting \textit{asymptotic} convergence results for Bayesian posteriors in
setups in which the MAP estimator may not be not unique. We also
mention two very recent pieces of work \citep{katsevich,katsevich1}, which
appeared after the first version of the present article was put on ArXiv. \citet{katsevich} studies the leading order contribution to the total variation distance
between the posterior and the Laplace approximation and thus gives
\textit{necessary} conditions under which the Laplace approximation is valid. For the sufficient conditions, \citet{katsevich} refers the reader to the ArXiv version of the present article, in particular to our \Cref{main_map,thm:wasserstein_map}.
\citet{katsevich1} proposes a skew adjustment to the Laplace approximation so as to
improve its quality in high dimensions.

Non-asymptotic analyses of the Bernstein--von Mises (BvM) Theorem (discussed in  \Cref{sec:bvm}) have previously been performed by \citet{spokoiny_panov, yano_kato}. The aim and focus of those analyses are however significantly different from those of the present paper, as we concentrate on the quality of the Laplace approximation (of the form of \Cref{eq:bvm3}, as discussed in \Cref{sec:bvm}). \citet{spokoiny_panov} consider semiparametric inference and prove results whose purpose it is to provide insight into the \textit{critical dimension} of the parameter for which the BvM Theorem holds. 
They prove their bounds only for the non-informative and the Gaussian prior.
Computing their bounds requires knowledge of the true parameter, which is the object of inference.
Additionally, 
\citet{yano_kato} derive a Berry--Esseen-type bound (which depends on the true
parameter value) on the total variation distance between the posterior and the
approximating normal in the approximately linear regression model. 

In contrast to the references mentioned above, our bounds hold and may be computed without access to the true parameter, for general posteriors
satisfying assumptions analogous to the classical assumptions of the
Bernstein--von Mises theorem (see e.g.\ \citealt[Section 4]{miller} for a recent
reference or \citealt[Section 1.4]{ghosh} for a more classical one). In particular,
we do not require (weak or strong) log-concavity or sub-Gaussianity of the
posterior. Indeed,  we compute our bounds explicitly for examples of commonly
used non-log-concave and heavy-tailed posteriors in \Cref{applications}. The
reason we can avoid imposing the assumption of log-concavity is that all we need
to control the tail behaviour of the posterior is the assumption of the strict
optimality of the MLE or MAP (\Cref{assump_kappa} or \Cref{assump_kappa1}). This requirement
is much weaker than assuming log-concavity or strong unimodality of the
posterior. The array of priors our bounds are available for is also much wider
than just the Gaussian family. All we require is differentiability, boundedness,
and boundedness away from zero of the prior density in a small neighbourhood
around the MAP or the MLE. We do not make any assumption about the true
distribution of the data, and we cover generalized likelihoods not coming from
any particular family. Moreover, we bound a variety of divergences including
those that control means and variances.

Finally, it is worth noting that bounds similar to ours are not widely available for approximate Bayesian inference techniques in general. Indeed, they are not available for variational inference (VI) methods \citep{blei, wainwright}, except for the recently studied case of Gaussian VI \citep{katsevich_rigollet}. In the area of Markov Chain Monte Carlo methods, progress has recently been made on deriving convergence guarantees for the Unadjusted Langevin Algorithm under different sets of assumptions on the tail growth of the target distribution \citep{erdogdu1, erdogdu2, erdogdu3}. However, the popular Metropolis-adjusted Langevin Algorithm and Hamiltonian Monte Carlo are not equipped with such guarantees beyond the case of log-concave targets. Flexible and computable post hoc checks measuring a discrepancy between the empirical distribution of a sample and the target distribution are given by graph and kernel Stein discrepancies. Graph Stein discrepancies \citep{mackey1, mackey2} metrize weak convergence and control the difference of means for distantly dissipative targets. They are however not known to control the difference of variances. Graph diffusion Stein discrepancies \citep{mackey2} possess the same properties under a slightly weaker (yet technical) assumption of the underlying diffusion having a rapid Wasserstein decay rate. The fast and popular kernel Stein discrepancies \citep{chwialkowski, liu, oates, mackey3} metrize weak convergence for certain choices of kernels and for distantly-dissipative targets. In comparison, we derive control over the rate of weak convergence and the differences of means and variances under interpretable assumptions that do not require any particular tail behavior of the target posterior.

\subsection{Structure of the paper}
In \Cref{section_notation} we describe our setup, introduce the necessary notation, and present our assumptions. We also give examples of popular models satisfying the assumptions. In \Cref{s:main} we present and discuss our main bounds. In \Cref{discussion} we provide a comprehensive discussion of the main results and their dependence on the sample size, dimension, and data, as well as computability. \Cref{sed:proof_techniques} discusses our proof techniques. In \Cref{further_results} we present results analogous to those of \Cref{s:main}, yet focused on different types of approximations.
In \Cref{applications} we show how to compute our bounds for the (non-log-concave) posterior in the logistic regression model with Student's t prior. We also present plots of our bounds computed numerically in this case. Moreover, we numerically compare our control over the difference of means and the difference of variances to the ground truth for certain conjugate prior models. In \Cref{conclusions} we present conclusions of our work. The proofs of the results of \Cref{s:main,further_results} are postponed to the appendices. 

\section{Setup, assumptions and notation}\label{section_notation}
\subsection{Setup}
Our setup and assumptions are similar to those found for instance in \cite{miller}. We fix $n\in\mathbbm{N}$ and study probability measures on $\mathbbm{R}^d$ having Lebesgue densities of the form:
\begin{align}\label{post_density}
	\posterior(\param)=e^{\loglikelihood(\param)}\prior(\param)/z_n,
\end{align}
where $\pi:\mathbbm{R}^d\to\mathbbm{R}$ is a Lebesgue probability density function, $L_n:\mathbbm{R}^d\to\mathbbm{R}$ and $z_n\in\mathbbm{R}_+$ is the normalizing constant. 
Throughout the paper, we call $\prior$ \textit{the prior density} (or simply \textit{the prior}), $\loglikelihood$ \textit{the generalized log-likelihood}, and $\posterior$ \textit{the generalized posterior}.
By $$\logposterior(\param):=\log\left(\posterior(\param)\right)$$ we denote the \textit{generalized log-posterior} and let:
\begin{align*}
	\mle:=\argmax_{\param\in\mathbbm{R}^d}\loglikelihood(\param),\qquad \map:=\argmax_{\param\in\mathbbm{R}^d}\logposterior(\param),
\end{align*}
whenever those quantities exist. 
If those quantities are unique, we call $\mle$ the \textit{maximum likelihood estimator} (\textit{MLE}) and $\map$ the \textit{maximum a posteriori} (\textit{MAP}) estimator. Here and elsewhere, overbars will refer to quantities
derived from the MAP ($\jbar$, $\deltabar$, $\mbar{2}$, etc., introduced below) and hats will refer
to quantities derived from the MLE (for instance $\jhat$, introduced below).
For any twice-differentiable function $f:\mathbbm{R}^d\to\mathbbm{R}$, we let $f'$ stand for its gradient and $f''$ for its Hessian. We shall write
\begin{align*}
	\fisherinformation(\param)=-\frac{\loglikelihood''(\param)}{n},\qquad 	\postfisherinformation(\param)=-\frac{\logposterior''(\param)}{n},
\end{align*}
whenever those expressions make sense (i.e.\ when the Hessians exist). For any three times differentiable function $f:\mathbbm{R}^d\to\mathbbm{R}$, we will also write $f'''$ for its third (Fr\'echet) derivative, defined as the following multilinear $3$-form on $\mathbbm{R}^d$:
$$f'''(\param)[u,v,w]=\sum_{i,j,k=1}^d\frac{\partial f}{\partial \param_i\partial\param_j\partial\param_k}(\param)u_iv_jw_k.
$$
The norm $\|\cdot\|^*$ of this third derivative will be defined in the following way:
$$\|f'''(\param)\|^*:=\underset{\|u\|\leq 1,\|v\|\leq 1,\|w\|\leq 1}{\sup}\left|f'''(\param)[u,v,w]\right|.$$
Throughout the paper, $\|\cdot\|$ denotes the Euclidean norm.
We will also let $\minevaluemle$ be the minimal eigenvalue of $\fisherinformation(\mle)$ and $\minevaluemap$ be the minimal eigenvalue of $\postfisherinformation(\map)$.
Moreover, throughout the paper  $\|\cdot\|_{op}$ will denote the operator (i.e. spectral) norm, $\left<\cdot,\cdot\right>$ will be the Euclidean inner product and $\postparam$ will always denote a random variable distributed according to the generalized posterior measure with density given by \cref{post_density}. $\mathcal{N}(\mu,\Sigma)$ will denote the normal distribution with mean $\mu$ and covariance $\Sigma$, and the function $\mathcal{L}(\cdot)$ will return the law of its argument. $I_{d\times d}$ will always denote the $d$-dimensional identity matrix.

Our bounds will be derived for two types of approximations. The first type is what we call the \textit{MAP-centric approach}. Within this approach, $\mathcal{L}(\sqrt{n}(\postparam-\map))$  is approximated by $\mathcal{N}(0,\postfisherinformation(\map)^{-1})$. On the other hand, what we call the \textit{MLE-centric approach} is the approximation of $\mathcal{L}(\sqrt{n}(\postparam-\mle))$ by  $\mathcal{N}(0,\fisherinformation(\mle)^{-1})$. The \textit{MLE-centric} approach is similar to the classical Bernstein--von Mises Theorem (see \Cref{sec:bvm} for details). The approximation it yields is a Gaussian distribution whose parameters depend only on the likelihood while ignoring the prior completely. The \textit{MAP-centric} approach is arguably more popular and standard among practitioners using the Laplace approximation. The approximating Gaussian's parameters depend on the posterior, thus depending on the prior as well. Thus, this approximation may be expected to be more accurate for many commonly used models. We derive our results for both approaches in order to make them applicable in a variety of circumstances. Indeed, our bounds can be used by practitioners who are interested in assessing the applicability of one of those approximation approaches. At the same time, we believe they are also interesting for researchers who study theoretical aspects of Bayesian Central Limit Theorems. 

The bounds we obtain are on the following distances:
\begin{enumerate}
	\item The Total Variation (TV) distance, which, for two probability measures $\nu_1$ and $\nu_2$ on a measurable space $(\Omega,\mathcal{F})$ is defined by $$TV(\nu_1,\nu_2):=\sup_{A\in\mathcal{F}} \left|\nu_1(A)-\nu_2(A)\right|.$$
	\item The 1-Wasserstein distance, which, for probability measures $\nu_1$ and $\nu_2$ and the set $\Gamma(\nu_1,\nu_2)$ of all couplings between them, is defined by
	$$W_1(\nu_1,\nu_2):=\inf_{\gamma\in \Gamma(\nu_1,\nu_2)}\int\|x-y\|d\gamma(x,y).$$
	Kantorovich duality (see, e.g. \citealt[Theorem 5.10]{villani}) provides an equivalent definition. Let $\|\cdot\|_L$ return the Lipschitz constant of the input. Then
	$$W_1(\nu_1,\nu_2)=\underset{\|f\|_L=1}{\sup_{f\text{ Lipschitz}:}}\left|\mathbbm{E}_{\nu_1}f - \mathbbm{E}_{\nu_2}f\right|.$$
	\item The following integral probability metric, which for $Y_1\sim\nu_1$ and $Y_2\sim\nu_2$ is defined by
	$$\sup_{v:\|v\|\leq 1}\left|\mathbbm{E}\left<v,Y_1\right>^2-\mathbbm{E}\left<v,Y_2\right>^2\right|.$$
	For zero-mean $\nu_1$ and $\nu_2$, this integral probability metric controls the operator norm of the difference of the covariance matrices of $\nu_1$ and $\nu_2$. For more general $\nu_1$ and $\nu_2$, it may be combined with the 1-Wasserstein distance in order to provide control over the operator norm of the difference between the covariance matrices of $\nu_1$ and $\nu_2$. Our proof techniques allow us to upper-bound integral probability metrics. Obtaining a bound on the above integral probability metric lets us reach our goal of controlling the difference between the covariance of the posterior and that of the Laplace approximation.
\end{enumerate}

\subsection{Assumptions made throughout the paper}
Now we list the assumptions that we will need to prove our finite-sample bounds and define constants used therein. We will first present those assumptions that will stand for both approaches described above and then others, which are divided between those relevant for the MAP-centric approach (\Cref{map-centric}) and the MLE-centric approach (\Cref{mle-centric}). We reiterate that the conditions we require are similar to the classical assumptions of the Bernstein--von Mises theorem, as given in \citet[Section 1.4]{ghosh} and \citet[Theorem 5]{miller}.
The first assumption that will be used throughout the paper is the following:
\begin{assumption}\label{assump1}
	There exists a unique MLE $\mle$. There also exists a real number $\deltamle>0$ such that the generalized log-likelihood $\loglikelihood$ is three times differentiable inside $\{\param:\|\param-\mle\|\leq \deltamle\} $. For the same $\deltamle>0$ there exists a real number $M_2>0$, such that:
	\begin{align}\label{taylor_condition}
		\underset{\|\param-\mle\|\leq \deltamle}{\sup}\frac{\|L_n'''(\param)\|^*}{n}\leq M_2.
	\end{align}
\end{assumption}
\begin{remark}
	\Cref{assump1} is needed to ensure that the posterior looks Gaussian inside the $\deltamle$-neighborhood around the MLE. Indeed, \Cref{assump1}, combined with Taylor's theorem, implies that $L_n(n^{-1/2}\theta+\mle)-L_n(\mle)$ can be approximated by $-\frac{1}{2}\theta^T\fisherinformation(\mle)\theta$ for $\theta$ in a neighborhood of $\mle$. Note that $-\frac{1}{2}\theta^T\fisherinformation(\mle)\theta$ is the logarithm of the $\mathcal{N}(0,\jhat)$ density (up to an additive constant).  When proving our bounds in the MLE-centric approach, we also use \Cref{assump1} to prove that the posterior satisfies the log-Sobolev inequality inside this neighborhood (see \Cref{appendix_c} for details).
\end{remark}
Moreover, we make the following assumption on the prior:
\begin{assumption}\label{assump_prior1}
	For the same $\deltamle>0$ as in \Cref{assump1}, there exists a real number $\hat{M}_1>0$, such that
	\begin{align*}
		\underset{\param:\|\param-\mle\|\leq\deltamle}{\sup}\left|\frac{1}{\prior(\param)}\right|\leq \hat{M}_1.
	\end{align*}
\end{assumption}
\begin{remark}
	Note that for \Cref{assump_prior1} to be satisfied, it suffices to assume that $\prior$ is continuous and positive in the $\deltamle$-ball around $\mle$. \Cref{assump_prior1} essentially ensures that the prior puts a non-negligible amount of mass in the $\deltamle$-neighbourhood of the MLE.
\end{remark}

\subsection{Additional assumptions in the MAP-centric approach}\label{map-centric}
In the MAP-centric approach we keep \Cref{assump1,assump_prior1} and additionally assume the following:
\begin{assumption}\label{assump2}
	There exists a unique MAP $\map$. There also exists a real number $\deltabar>0$, such that the log-prior, $\log\prior$, is three times differentiable inside $\{\param:\|\param -\map\|\leq \deltabar\}$. Moreover, for the same $\deltabar$, there exists a real number $\overline{M}_2>0$, such that
	\begin{align}\label{post_taylor_condition}
		\sup_{\param:\|\param-\map\|\leq\deltabar}\frac{\|\logposterior'''(\param)\|^*}{n}\leq \overline{M}_2.
	\end{align}
\end{assumption}
\begin{remark}
	\Cref{assump2} is very similar to \Cref{assump1}. The difference is that we now consider a ball around the MAP rather than the MLE, and we require additional differentiability of the prior density inside this ball. We will use \Cref{assump2} in the MAP-centric approach (together with \Cref{assump7}) to show that, inside the $\deltabar$-ball around the MAP, the posterior satisfies the log-Sobolev inequality.
	
	In the MAP-centric approach, we use both \Cref{assump2,assump1}. \Cref{assump1} will play an important role in the process of controlling the posterior in the region $\{\param:\|\param -\map\|> \deltabar\}$. More specifically, we will use \Cref{assump1} to lower-bound the normalizing constant of the posterior after controlling it with the integral $\int_{\|t-\mle\|\leq\deltamle}\prior(t)e^{\loglikelihood(t)-L_n(\mle)}dt$. See \Cref{sec:control_i2map} for more detail.
\end{remark}
\begin{assumption}\label{assump_size_of_delta_bar}
	For the same $\deltamle>0$ as in \Cref{assump1} and for the same $\deltabar$ as in \Cref{assump2},
	\begin{align*}
		\max\left\{\|\hat{\theta}_n-\bar{\theta}_n\|,\sqrt{\frac{\Tr\left[\postfisherinformation(\map)^{-1}\right]}{n}}\right\}<\deltabar\quad\text{and}\quad \sqrt{\frac{\Tr\left[\left(\jhat +\frac{\deltamle M_2}{3}I_{d\times d}\right)^{-1}\right]}{n}}< \deltamle.\end{align*}
\end{assumption}
\begin{remark}
	\Cref{assump_size_of_delta_bar} is an assumption on the size of $n$ and the choice of $\deltabar$ and $\deltamle$. Indeed, as long as the MLE and MAP converge to the same limit (which is true in the majority of commonly used modelling setups), we expect $\max\left\{\|\hat{\theta}_n-\bar{\theta}_n\|,\sqrt{\frac{\Tr\left(\postfisherinformation(\map)^{-1}\right)}{n}}\right\}$ to go to zero as $n\to\infty$. We similarly expect $\sqrt{\frac{\Tr\left[\left(\jhat +\frac{\deltamle M_2}{3}I_{d\times d}\right)^{-1}\right]}{n}}$ to go to zero as $n\to\infty$. Moreover $\sqrt{\frac{\Tr\left[\left(\jhat +\frac{\deltamle M_2}{3}I_{d\times d}\right)^{-1}\right]}{n}}<\deltamle$ will be satisfied if $\sqrt{\frac{\Tr\left(\fisherinformation(\mle)^{-1}\right)}{n}}<\deltamle$, which might be an easier condition to check.
	\Cref{assump_size_of_delta_bar} allows us to use appropriate Gaussian concentration inequalities in our proofs. Moreover, the assumption $\|\mle-\map\|<\deltabar$ is necessary for \Cref{assump_kappa1} below to be satisfied.
\end{remark}
\begin{assumption}\label{assump7}
	For the same $\deltabar>0$ and $\overline{M}_2>0$ as in \Cref{assump2},
	\begin{align}\label{inv_prior_cond}
		\minevaluemap> \deltabar\, \overline{M}_2.
	\end{align}
\end{assumption}
\begin{remark}
	If $\postfisherinformation(\map)$ is positive definite and \Cref{assump1} is satisfied then one can adjust the choice of $\deltabar$ so that both \cref{post_taylor_condition,inv_prior_cond} hold. Indeed, one can adjust the value of $\deltabar$ accordingly because decreasing the value of $\deltabar$ in \Cref{assump2} does not lead to an increase in the value of $\overline{M}_2$. At the same time, decreasing the value of $\deltabar$ in \Cref{assump7}, while keeping $\overline{M}_2$ fixed,  decreases the right-hand side of \cref{inv_prior_cond}.
	
	\Cref{assump7}, combined with \Cref{assump2} and with Taylor's theorem, will allow us to prove that the posterior is strongly log-concave inside the $\deltabar$-neighborhood around the MAP. As a result, we will be able to show that the posterior satisfies the log-Sobolev inequality inside this neighborhood.
\end{remark}
\begin{assumption}\label{assump_kappa1}
	For the same $\deltabar>0$ as in \Cref{assump2}, there exists $\kappabar>0$, such that
	\begin{align*}
		\underset{\param:\|\param-\mle\|>\deltabar-\|\mle-\map\|}{\sup}\frac{L_n(\param)-L_n(\hat{\param}_n)}{n}\leq -\kappabar.
	\end{align*}
\end{assumption}
\begin{remark}
	\Cref{assump_kappa1} ensures that any local maxima of $\loglikelihood$ achieved outside of the $\left(\deltabar-\|\mle-\map\|\right)$-ball around the MLE do not get arbitrarily close to the global maximum achieved at the MLE.
	It also ensures that the posterior puts asymptotically negligible mass outside the $\left(\deltabar-\|\mle-\map\|\right)$-neighborhood around the MLE.  As a result, only the \textit{locally Gaussian} part around the MLE remains as the sample size $n$ grows.
	Note that for the vast majority of commonly used parametric models and data generating distributions,  $\|\mle-\map\|$, which appears in the expression for the radius of the ball, will tend to $0$ as $n\to\infty$, a.s. Moreover, note that we do not strictly require that $\kappabar$ not depend on the sample size $n$. Under certain conditions, our bounds will converge to zero as $n\to\infty$ even if $\bar{\kappa}\xrightarrow{n\to\infty} 0$, as long as $\kappabar$ vanishes strictly slower than $\frac{d\log n}{n}$ (see \Cref{sec:sample_size_dep,dimension_dependence} for more discussion).
	
	This kind of assumption is standard in the discussion of the Bernstein--von Mises Theorem, both in classical references \citep{ghosh} and in more recent ones  \citep{miller}. Nevertheless, we note that there are references in which this assumption is replaced with certain weaker probabilistic separation conditions (such as uniformly consistent tests), see e.g. \cite{vandervaart}. However, it is not clear how to adapt similar conditions to our setup, in which we seek to obtain computable non-asymptotic bounds. Similarly, \cite{miller} remarked that it is already not clear how to use probabilistic separation conditions for  proving the asymptotic convergence of the posterior to Gaussianity when the studied convergence is almost sure, rather than in probability.
\end{remark}

\begin{remark}\label{remark_delta}
	In \Cref{assump_prior1,assump_size_of_delta_bar,assump7,assump_kappa1} we require the stated conditions to hold for the same $\deltamle$ as in \Cref{assump1} and for the same $\deltabar$ as in \Cref{assump2}. In practice, we may, however, verify those assumptions separately and, for each of them, find the ranges for $\deltamle>0$ and $\deltabar>0$ for which it holds. We may then set the values of $\deltamle>0$ and $\deltabar>0$ equal to (one of) the values of $\deltamle>0$ and $\deltabar>0$ for which all the \Cref{assump1,assump_prior1,assump2,assump_size_of_delta_bar,assump7,assump_kappa1} are satisfied.
\end{remark}

\subsection{Additional assumptions in the MLE-centric approach}\label{mle-centric}

Besides \Cref{assump1} and \Cref{assump_prior1}, in the MLE-centric approach, we have the following assumptions:
\begin{assumption}\label{assump_size_of_delta}
	For the same $\deltamle>0$ as in \Cref{assump1},
	\begin{align*}
		\sqrt{\frac{\Tr\left[\fisherinformation(\mle)^{-1}\right]}{n}}<\deltamle.
	\end{align*}
\end{assumption}
\begin{remark}
	\Cref{assump_size_of_delta} is an assumption on the size of $n$ and the choice of $\deltamle$. Indeed, for typical applications we expect $\sqrt{\frac{\Tr\left[\fisherinformation(\mle)^{-1}\right]}{n}}$ to go to zero as $n\to\infty$. This is a technical assumption, necessary to ensure that the Gaussian concentration inequalities we use in our proofs are valid.
\end{remark}

\begin{assumption}\label{assump_fisher}
	For the same $\deltamle>0$ and $M_2>0$ as in \Cref{assump1},
	\begin{align}
		\minevaluemle> \deltamle M_2.
	\end{align}
\end{assumption}
\begin{remark}
	\Cref{assump_fisher} is the analogue of \Cref{assump7} for the MLE-centric approach.
	Combined with \Cref{assump1} and with Taylor's theorem, this assumption will allow us to prove that the likelihood is strongly log-concave inside the $\deltamle$-neighborhood around the MLE. As a result, we will be able to show that the posterior satisfies the log-Sobolev inequality inside this neighborhood.
\end{remark}
\begin{assumption}\label{assump_kappa}
	For the same $\deltamle>0$, as in \Cref{assump1}, there exists $\kappamle>0$, such that
	\begin{align*}
		\underset{\param:\|\param-\mle\|>\deltamle}{\sup}\frac{L_n(\param)-L_n(\hat{\param}_n)}{n}\leq -\kappamle.
	\end{align*}
\end{assumption}
\begin{remark}
	\Cref{assump_kappa} is similar to \Cref{assump_kappa1} and ensures that any local maxima of $\loglikelihood$ achieved outside of the $\deltamle$-ball around the MLE do not get arbitrarily close to the global maximum achieved at the MLE. In other words, this assumption ensures that the posterior puts asymptotically negligible mass outside the $\deltamle$-neighborhood around the MLE.
\end{remark}
\begin{assumption}\label{assump_prior2}
	For the same $\deltamle$ as in \Cref{assump1}, there exist real numbers $M_1>0$ and $\widetilde{M}_1>0$, such that
	\begin{align*}
		\underset{\param:\|\param-\mle\|\leq\deltamle}{\sup}\left\|\frac{\prior'(\param)}{\prior(\param)}\right\|\leq M_1 \qquad\text{and}\qquad\underset{\param:\|\param-\hat{\theta}_n\|\leq\deltamle}{\sup}\left|\prior(\param)\right|\leq \widetilde{M}_1.
	\end{align*}
\end{assumption}
\begin{remark}
	Note that for \Cref{assump_prior2} to be satisfied, it suffices that $\prior$ is continuously differentiable and positive inside the $\deltamle$-ball around $\mle$. \Cref{assump_prior2} is a technical assumption that we use in the MLE-centric approach when showing that the posterior satisfies the log-Sobolev inequality inside the $\deltamle$-ball around the MLE.
\end{remark}

\begin{remark}
	As in Remark \ref{remark_delta}, we note that one may first verify \Cref{assump1,assump_prior1,assump_kappa,assump_prior2,assump_fisher,assump_size_of_delta} separately and then set the value of $\deltamle>0$ equal to (one of) the values of $\deltamle>0$ for which all of those assumptions are satisfied.
\end{remark}

\subsection{Additional notation}\label{sec:additional_notation}
Certain quantities occur repeatedly in our bounds.  By giving these quantities
special symbols, we can express the bounds more compactly and readably.

First, we define a set of matrices closely related to $\jhat$ and $\jbar$.
\begin{align*}
\jhatplus :={}&
    \jhat + (\deltamle \m{2}/3)I_{d\times d} &
\jhatminus :={}&
    \jhat - (\deltamle \m{2}/3)I_{d\times d} \\
\jbarplus :={}&
    \jbar + (\deltabar\, \mbar{2}/3)
    I_{d\times d}  &
\jbarminus :={}&
    \jbar - (\deltabar \mbar{2}/3)
    I_{d\times d}.
\end{align*}
The superscript $p$ in those symbols refers to the \textit{plus} sign appearing in the definition of the symbol and the superscript $m$ refers to the \textit{minus} sign.
%
%
Each of these matrices is positive definite by Assumptions \ref{assump7} and
\ref{assump_fisher}.  

We analogously define the minimum eigenvalues of each matrix in the preceding
display:
\begin{align*}
\minevalmleplus :={}&
    \left[\left\| \jhatplus^{-1} \right\|_{op}\right]^{-1} ; &
\minevalmleminus :={}&
    \left[\left\| \jhatminus^{-1} \right\|_{op}\right]^{-1} ;\\
\minevalmapplus :={}&
    \left[\left\| \jbarplus^{-1} \right\|_{op}\right]^{-1} ; &
\minevalmapminus :={}&
    \left[\left\| \jbarminus^{-1} \right\|_{op}\right]^{-1}.
\end{align*}


%

%
Finally, we define a set of quantities of the following
form:
\begin{align*}
\decaybar :={}&
    \exp\left[-\frac{1}{2}
        \left(\deltabar\sqrt{n}-
        \sqrt{\jbarinvtr}\right)^2\minevaluemap
        \right] ;\\
\decaybarplus :={}&
    \exp\left[-\frac{1}{2}
        \left(\deltabar\sqrt{n}-
        \sqrt{\jbarplusinvtr}\right)^2\minevalmapplus
        \right]; \\
\decayhat :={}&
    \exp\left[-\frac{1}{2}
        \left(\deltamle\sqrt{n}-
        \sqrt{\jhatinvtr}\right)^2\minevaluemle
        \right] ;\\
\decayhatplus :={}&
    \exp\left[-\frac{1}{2}
        \left(\deltamle\sqrt{n}-
        \sqrt{\jhatplusinvtr}\right)^2\minevalmleplus
        \right].
\end{align*}
The preceding terms are labeled with a script ``D'' for ``decay,'' because they
decrease to zero exponentially in $n$ when all other quantities are held
constant.




\subsection{Examples of popular models satisfying the assumptions}
%
We now present results stating that, under standard regularity conditions, and
for large enough $n$, exponential families and generalized linear models satisfy
the assumptions listed above with constants that do not depend on $n$.

\subsubsection{Exponential families}

 We have the following result, whose proof can be found in Appendix \ref{proof_exponential}.
\begin{proposition}[cf. {\citealt[Theorem 12]{miller}}]
\label{prop_exponential}
    Consider an exponential family that is full, regular, nonempty, identifiable and in natural form. In particular, let this exponential family have density $q(y|\eta)=\exp\left(\eta^T
 s(y)-\alpha(\eta) \right)$ with respect to a sigma-finite Borel measure
 $\lambda$ on $\mathcal{Y}\subseteq\mathbbm{R}^k$, where
 $s:\mathcal{Y}\to\mathbbm{R}^d$, $\eta\in\mathbbm{R}^d$ and
 $\alpha(\eta)=\log\int_{\mathcal{Y}}\exp(\eta^Ts(y))\lambda(dy)$. Let
 $Q_{\eta}(E)=\int_Eq(y|\eta)\lambda(dy)$ and denote
 $\mathbbm{E}_{\eta}s(Y)=\int_{\mathcal{Y}}s(y)Q_{\eta}(dy)$. Let
 $\mathcal{E}:=\{\eta\in\mathbbm{R}^d\,:\,|\alpha(\eta)|<\infty\}$.
   Assume that $\mathcal{E}$ is open and nonempty, let the parameter space be
   given by $\Theta:=\mathcal{E}$ and assume that $\eta\mapsto Q_{\eta}$ is
   one-to-one. 
   
   Suppose $Y_1,Y_2,\dots\in\mathcal{Y}$ are i.i.d.\ random vectors,
   such that $\mathbbm{E}s(Y_i)=\mathbbm{E}_{\theta_0}s(Y)$ for some
   $\theta_0\in\Theta:=\mathcal{E}$. Let $L_n(\theta)=\sum_{i=1}^n \log
   q(Y_i|\theta)$. Then, almost surely, for large enough $n$, Assumptions \ref{assump1},
   \ref{assump_size_of_delta}, \ref{assump_fisher} and \ref{assump_kappa} are
   satisfied, with constants $\deltamle,M_2,\kappamle$ independent of
   $n$.

   If, in addition, the prior density $\pi$ does not depend on $n$ and is continuous and positive in a
   neighborhood around $\theta_0$ then, almost surely, for large enough $n$,  Assumption \ref{assump_prior1} is
   satisfied with for constant $\hat{M}_1$
   independent of $n$.

   If, in addition, the prior density $\pi$ is continuously differentiable in a
   neighborhood of $\theta_0$ then, almost surely, Assumption \ref{assump_prior2} is satisfied for large enough $n$, 
   with $M_1,\widetilde{M}_1$ independent of
   $n$.

   If, in addition, the prior density $\pi$ is thrice continuously
   differentiable on $\Theta$, then, almost surely, Assumptions \ref{assump2} --
   \ref{assump_kappa1} are satisfied for all large enough $n$, 
   with constants $\deltabar,\overline{M}_2,\kappabar$ independent of
   $n$.
\end{proposition}

\begin{remark}
Our Proposition \ref{prop_exponential} is very similar to \citet[Theorem 12]{miller}. The assumptions of Proposition \ref{prop_exponential} are, firstly, that the exponential family is full, regular, nonempty, identifiable and in natural form. Such conditions are standard and hold for typical commonly used exponential families \citep{miller_harrison}. Secondly, we assume that $\mathbbm{E}s(Y_i)=\mathbbm{E}_{\theta_0}s(Y)$ for some $\theta_0$, which is a standard assumption and  ensures that matching the expected sufficient statistics to the observed sufficient statistics is possible asymptotically.  Note that we do not assume that the model is correctly specified.
\end{remark}

\subsubsection{Generalized Linear Models}

We have the following result, whose proof can be found in Appendix \ref{proof_glm}.

\begin{proposition}[cf. {\citealt[Theorem 13]{miller}}]\label{prop_glm}
    Consider a regression model of the form $p(y_i|\theta,x_i)\propto_{\theta}
q(y_i|\theta^Tx_i)$ for covariates $x_i\in\mathcal{X}\subseteq\mathbbm{R}^d$
and coefficients $\theta\in \Theta\subseteq\mathbbm{R}^d$, where
$q(y|\eta)=\exp\left(\eta s(y)-\alpha(\eta)\right)$ is a one-parameter
exponential family, with respect to a sigma-finite Borel measure $\lambda$ on
$\mathcal{Y}\subseteq\mathbbm{R}^d$. Note that proportionality $\propto_{\theta}$ is with respect
to $\theta$, not $y_i$. 

Let $s:\mathcal{Y}\to\mathbbm{R}^d$,
$\eta\in\mathbbm{R}^d$ and
$\alpha(\eta)=\log\int_{\mathcal{Y}}\exp(\eta^Ts(y))\lambda(dy)$.  Moreover,
let $Q_{\eta}(E)=\int_Eq(y|\eta)\lambda(dy)$ and
$\mathcal{E}:=\{\eta\in\mathbbm{R}^d\,:\,|\alpha(\eta)|<\infty\}$.
   Assume $\Theta$ is open, $\Theta$ is convex, and $\theta^Tx\in\mathcal{E}$
   for all $\theta\in\Theta$, $x\in\mathcal{X}$. Moreover, assume $\mathcal{E}$
   is non-empty and open and $\eta\mapsto Q_{\eta}$ is one-to-one. Suppose
   $(X_1,Y_1),(X_2,Y_2),\dots\in\mathcal{X}\times\mathcal{Y}$ are i.i.d.\ such that:
   \begin{enumerate}
        \item $f'(\theta_0)=0$ for some $\theta_0\in\Theta$, where
        $f(\theta)=\mathbbm{E}\log q(Y_i\,|\,\theta^TX_i)$
        \item $\mathbbm{E}\left|X_is(Y_i)\right|<\infty$ and
        $\mathbbm{E}\left|\alpha(\theta^TX_i)\right|<\infty$ for all
        $\theta\in\Theta$,
        \item for all $a\in\mathbbm{R}^d$, if $a^TX_i\stackrel{a.s.}=0$ then
        $a=0$
       \item There is $\epsilon>0$ 
       such that for all $j,k,l\in\{1,\dots,d\}$, \sloppy
       $\mathbbm{E}\left[\underset{\theta:\|\theta-\theta_0\|\leq\epsilon}{\sup}
       \left|\alpha'''(\theta^TX_i)X_{ij}X_{ik}X_{il}\right|\right]<\infty.$
             %
\end{enumerate}
Let $L_n(\theta)=\sum_{i=1}^n \log
   p(Y_i|\theta,X_i)$. Then, almost surely, for all large enough $n$, Assumptions \ref{assump1},
\ref{assump_size_of_delta}, \ref{assump_fisher} and \ref{assump_kappa} are
satisfied with constants $\deltamle,M_2,\kappamle$ independent of
$n$.

If, in addition, the prior density $\pi$ does not depend on $n$ and is continuous and positive in a
neighborhood around $\theta_0$ then, almost surely, Assumption \ref{assump_prior1} is satisfied
for large enough $n$ and constant $\hat{M}_1$ independent of $n$.

If, in addition, the prior density $\pi$ is continuously differentiable in a
neighborhood of $\theta_0$ then, almost surely, Assumption \ref{assump_prior2} is satisfied
for large enough $n$ and $M_1,\widetilde{M}_1$ independent of $n$.

If, in addition, the prior density $\pi$ is thrice continuously
differentiable on $\Theta$, then, almost surely, Assumptions \ref{assump2} --
\ref{assump_kappa1} are satisfied for all large enough $n$ and
for constants $\deltabar,\overline{M}_2,\kappabar$ independent of
$n$.
\end{proposition}
\begin{remark}
Our Proposition \ref{prop_glm} is very similar to \citet[Theorem 13]{miller}. Condition 1 in Proposition \ref{prop_glm} says essentially that the MLE exists asymptotically. Conditions 2 and 4 are moment conditions. They will be satisfied in many situations --- for instance if the covariates are bounded and $\mathbbm{E}s(Y_i)$ exists (since $\alpha\in C^{\infty}$). Condition 3 is necessary to ensure identifiability. When $\mathbbm{E}X_iX_i^T$ exists and is finite, condition 3 is equivalent to $\mathbbm{E}X_iX_i^T$ being non-singular, which is often assumed to ensure identifiability for GLMs \cite[Example 16.8]{vandervaart}. Note that Proposition \ref{prop_glm} does not assume that the model is correctly specified.
\end{remark}

\section{Main results}\label{s:main}

In this section we present our bounds on the quality of Laplace approximation in
the \textit{MAP-centric} approach, as described in
\Cref{section_notation}. This approach is arguably the most popular one among
users of Laplace approximation (e.g. \citealt[Section 4.4]{bishop}, \citealt[Sections 4.6.8.2 and 10.5.1]{murphy}). The proofs of the results from this section are
presented in \Cref{introductory_arguments,appendix_b}. A
discussion of our bounds can be found in \Cref{discussion} below. In
particular, \Cref{discussion} discusses the dependence on the sample size
$n$ and dimension $d$ of our bounds, their dependence on the data, their
computability and the way in which they control credible sets, means and
variances. Our bounds in the  \textit{MLE-centric} approach will be presented in \Cref{further_results}.

\subsection{Control over the total variation distance}
We start with a bound over the total variation distance.
\MainMapThm{}

\subsection{Control over the 1-Wasserstein distance}
Now, we bound the 1-Wasserstein distance, which is known to control the difference of means.
\WassertsteinMAPThm{}
\begin{remark}
	Note that, in order for the bound in \Cref{thm:wasserstein_map} to be finite, we need the following integral with respect to the prior to be finite: $\int_{\|u-\mle\|>\deltabar-\|\mle-\map\|}\|u-\map\|\pi(u)du$. 
\end{remark}
\subsection{Control over the difference of covariances}
Finally, we upper bound an integral probability metric that lets us control the difference of covariances.
\WTwoMAPThm{}
\begin{remark}
	Note that, in order for the bound in \Cref{w2_map} to be finite, we need the following integral with respect to the prior to be finite: $\int_{\|u-\mle\|>\deltabar-\|\mle-\map\|}\|u-\map\|^2\pi(u)du$.
\end{remark}

\section{Discussion of the bounds}\label{discussion}
We make some remarks about the bounds presented in \Cref{s:main} and their applicability in approximate inference. 
\subsection{Dependence on the sample size $n$}\label{sec:sample_size_dep}
The quantities $A_1$, $A_2$, $A_3$, $B_1$, $B_2$, $B_3$, $B_4$, $C_1$,
$C_2$, $C_3$, and $C_4$ appearing in the above bounds depend on $n$ but, for
models and data generating distributions that satisfy the assumptions of the
Bernstein--von Mises theorem (as in \citealt[Theorem 1.4.2]{ghosh} or \citealt[Theorem 5]{miller}), they are
bounded as $n$ grows. In particular, they are bounded in $n$, as long as the
constants $M_2,\overline{M}_2,\hat{M}_1$ are bounded from above and
$\deltabar, \deltamle$ are bounded from below by a positive
number. In this scenario, keeping the dimension $d$ fixed, all our bounds will vanish as $n\to\infty$
at the rate of $\frac{1}{\sqrt{n}}$, as long as asymptotically the ratio of $\frac{d+3}{2}\cdot\frac{\log n}{n}$ to $\kappabar$ goes to zero. See \cref{map_poisson,kappa_bar_weibull,calculating_kappa} for a discussion of the choice of $\kappabar$ in our experiments. We reiterate that our bounds are fully non-asymptotic and
computable.

In the numerical examples in \Cref{applications} below, we observed that our bounds are often coarse for small and moderate sample sizes. The main reason for this looseness at small and moderate sample sizes can be understood by looking at our assumptions. Indeed, in the cases where $\minevaluemap$ is very small, \Cref{assump7} pushes $\deltabar$ to also become very small. A small $\deltabar$ means, in turn, that we need to choose a very small $\kappabar$ to make \Cref{assump_kappa1} satisfied. Subsequently, a very small $\kappabar$, combined with a small or moderate sample size $n$, results in the bound in \Cref{main_map} being dominated by the third summand $A_2n^{d/2}e^{-n\kappabar}$, the bound in \Cref{thm:wasserstein_map} being dominated by $B_3(B_2+\sqrt{n}B_4)n^{d/2}e^{-n\kappabar}$ and the bound in \Cref{w2_map} being dominated by $C_3(C_2+nC_4)n^{d/2}e^{-n\kappabar}$. As a consequence, we need a large sample size $n$ in order to start seeing the asymptotic convergence rate of our bounds. 

\citet[Section 2.2]{katsevich_bvm} notes that our bounds are not affine invariant. Indeed, passing the parameter of interest through the following transformation $\theta\mapsto \postfisherinformation(\map)^{1/2}(\theta-\map)$ and applying our analysis to the transformed parameter would change the value of our bounds. In particular, such a transformation makes the negative Hessian of the generalized log posterior, divided by $n$ and evaluated at zero (the new mode of the posterior), equal to the identity matrix. Importantly, this means that, after the transformation, the left-hand side of  \Cref{inv_prior_cond} becomes equal to $1$, irrespective of the dimension $d$ and one can thus have a much better control over the range of $\deltabar$'s that satisfy \Cref{assump7}. We conjecture that such a transformation thus leads in many cases to an improvement in the tightness of our bounds for smaller sample sizes $n$. A detailed analysis of the consequences of applying the transformation proposed in \citet[Section 2.2]{katsevich_bvm} is, however, beyond the scope of the present paper. We refer the interested reader to \citet{katsevich_bvm} and, more specifically, to Section 2.2 and Appendix A.3 therein. 

We hypothesized that the looseness of our bounds might be a function of the condition number of $\postfisherinformation(\map)$ (defined as the ratio of the largest eigenvalue of $\postfisherinformation(\map)$ and the smallest eigenvalue $\minevaluemap$). In our experiments, we investigated this hypothesis for logistic regression with a Student's t prior. We found that while it may be the case that a smaller condition number is associated with a smaller bound, the relationship is not definitive. For instance, in \Cref{fig:logistic_condition_mean,fig:logistic_condition_covariance,fig:logistic_condition_number}, we can see both very large bound values and very small bound values for the same condition number; see \cref{sec:are_bounds_smaller} for more discussion. We leave a more detailed investigation for future work.

\subsection{Dependence on the dimension $d$}\label{dimension_dependence} Several
terms in our bounds depend on $d$ and the exact dependence differs depending
on the analyzed model. This is not surprising; the fact that the dimension
dependence of the Laplace approximation is model-dependent has already been
observed by other authors, including \cite{katsevich}. Nevertheless, as
observed in \citet{katsevich}, which was written after the first version of
the current paper appeared on ArXiv, the dimension dependence in our
total-variation and 1-Wasserstein bounds is such that it cannot be improved
in general. It is also better than the dimension dependence appearing in the
previous works on this topic by \cite{helin,spokoiny_laplace,dehaene,fisher}. Below,
we elaborate more on our dimension dependence and on the comparison to the
dependence other authors have achieved in the past.

Let us take a closer look at our bound on the total variation distance
presented in \Cref{main_map}. Our bound consists of three summands. The
first summand is of the same order as
$\frac{\Tr\left[\postfisherinformation(\map)^{-1}\right]\overline{M}_2}{\sqrt{n}\cdot\sqrt{\minevaluemap-\overline{M}_2\deltabar}}$
as $d$ and $n$ grow. The second summand is of a lower order, provided, for
instance, that $\deltabar\gg\frac{\sqrt{\log n}}{\sqrt{n\,\minevaluemap}}$.
This condition is reasonable to expect as for models satisfying the
assumptions of the Bernstein--von Mises theorem \citep[Theorem 1.4.2]{ghosh},
$\deltabar$ is required to be positive and constant in $n$. The third
summand is of the same order as
$\hat{M}_1\left|\text{det}\left(\jhatplus\right)\right|^{1/2}n^{d/2}e^{-n\kappabar}$.
Therefore, if $\kappabar\gg \frac{\log
	n}{n}\cdot\frac{d+1}{2}+\frac{1}{n}\log\left(\hat{M}_1\left|\text{det}\left(\jhatplus\right)\right|^{1/2}\right)$
and $\deltabar\gg\frac{\sqrt{\log n}}{\sqrt{n\,\minevaluemap}}$, then the
first summand is of the leading order as $n$ and $d$ grow. 

Let us then look at the dimension dependence of the first summand in
\Cref{main_map} through the lens of the recent analysis presented by
\citet{katsevich} and thus make a comparison to the previous work. Our first
summand in \Cref{main_map} is of the same order as $C_d\sqrt{d^2/n}$, where
\begin{align*}
	C_d = 
	\frac{\Tr\left[\postfisherinformation(\map)^{-1}\right]}{d}\cdot
	\frac{\overline{M}_2}
	{\sqrt{\minevaluemap-\overline{M}_2\deltabar}} \leq
	\frac{\overline{M}_2}{\minevaluemap\sqrt{\minevaluemap-\overline{M}_2\deltabar}},
\end{align*}
where we recall from \Cref{assump2} that $\overline{M}_2$ controls the
third derivative of the log posterior.
As noted by \cite{katsevich}, all the recent works by
\cite{helin,spokoiny_laplace,dehaene} have obtained finite-sample bounds on
the total variation distance of order $c_d\sqrt{d^3/n}$, where $c_d$ is a
ratio of the third derivative of the log posterior and the $(3/2)$-power of
the second derivative of the log-posterior, with the definition varying
slightly from paper to paper\footnote{To be precise, we note that the bound
	appearing in \cite{spokoiny_laplace} is actually expressed in terms of the
	\textit{effective dimension} $d_{\text{eff}}$, which may sometimes be
	smaller than the actual dimension $d$ and depends on the strength of
	regularization by a Gaussian prior. The effective dimension $d_{\text{eff}}$
	however approaches the true dimension $d$ as the sample size $n$ goes to
	infinity, as long as the prior does not depend on $n$. For such a prior, the
	asymptotic order of the total-variation-distance bound in
	\cite{spokoiny_laplace} is still $c_d\sqrt{d^3/n}$.} (while the bounds provided by \citealt{fisher} are of a higher order).  As described in
\Cref{related_work}, all the works of \citet{helin,spokoiny_laplace,dehaene,fisher}
obtain their bounds under assumptions significantly stronger than ours.
Still, our bound offers a tighter dimension dependence as it is of order
$C_d\sqrt{d^2/n}$, where $C_d$ again represents the ratio of a bound on the
third derivative of the log posterior and the $(3/2)$-power of the second
derivative of the log-posterior.

As noted above, \cite{katsevich} refers specifically to our
paper and states that the order of our bound cannot be improved in general.
Indeed, \citet[Theorem V1]{katsevich} states that the condition
$\bar{C}_dd\ll \sqrt{n}$ is necessary for accurate Laplace approximation,
where $\bar{C}_d$ is a model-specific term involving ratios of derivatives
of the log-posterior. We refer the interested reader to \citet{katsevich} for
a more precise statement of this result and a more thorough discussion.

In \Cref{example_1}, we discuss a popular model which satisfies the
assumptions of \cite{spokoiny_laplace}  and for which our bound on the
total variation distance is of smaller order than the one of
\cite{spokoiny_laplace}. 
\begin{example}\label{example_1}
	Let us consider the following example, inspired by \citet[Section 3]{katsevich} and discussed in more detail in \Cref{app_logistic_example}. Suppose that $X_i\stackrel{\text{i.i.d.}}\sim\mathcal{N}(0,I_{d\times d})$, $\tilde{Y}_i\,|\,X_i\sim \text{Bernoulli}\left(s(\param_0^TX_i)\right)$ and 
	\begin{align*}
		Y_i=\begin{cases}
			1,&\text{if }\tilde{Y}_i=1\\
			-1, &\text{if }\tilde{Y}_i=0,
		\end{cases}
	\end{align*}
	\sloppy where $s$ is the sigmoid $s(t)=(1+e^{-t})^{-1}$ and $\param_0$ is the ground truth value of the parameter. For simplicity, take $\param_0=(1,0,\dots,0)$. Now, let $\rho(t)=-\log s(t)$ and take $\loglikelihood(\param)=-\sum_{i=1}^n\rho\left(Y_iX_i^T\param\right)$, which corresponds to logistic regression. Consider a standard Gaussian prior on $\param$ given by $\pi(\param)\propto \exp(-\|\param\|^2/2)$ and assume $d^{3/2}\leq n\leq e^{\sqrt{d}}.$ This model satisfies the assumptions of \Cref{main_map} with high probability for sufficiently large $n$ and the first summand in the bound in \Cref{main_map} is of the leading order. See \Cref{app_logistic_example_assumptions} for a discussion of those facts.
	
	In \citet{spokoiny_laplace}, bounds are proved on the quality of the Laplace approximation for models in which the likelihood is log-concave and the prior is Gaussian. The model we consider here in this example (logistic regression with a Gaussian prior) satisfies these conditions as well as the additional local smoothness conditions imposed in \cite{spokoiny_laplace}. \cite{spokoiny_laplace} notes in their paper that their bounds on the total-variation distance from the Laplace approximation converge to zero only if $\sqrt{\frac{(d_{eff})^3}{n}}\to 0$. For logistic regression with a Gaussian prior, as we consider it here in this example, we have that:
	\begin{align*}
		d_{eff}:=&\Tr{\left(\jbar+\frac{(\log\pi)''(\map)}{n}\right)\jbar^{-1}}\geq d\left(1- \frac{1}{n \,\minevaluemap}\right),
	\end{align*}
	and $\minevaluemap$ is lower-bounded by a positive constant not depending on $n$ or $d$ with high probability (see \Cref{lower_bound_eff_dim} for more detail). Therefore $\frac{d_{eff}}{d}\xrightarrow{n\to\infty}1$ in probability. As we show in \Cref{example_bound}, our bound on the
	total variation distance is therefore of a smaller order than the one of
	\cite{spokoiny_laplace}; our bound is of order $\sqrt{\frac{d^2}{n}}$ with high
	probability. Moreover, we note that \citet[Section 3]{katsevich} has shown
	numerically that this rate is optimal for logistic regression. 
\end{example}

    Furthermore, we note that in \Cref{thm:wasserstein_map}, the first summand (which will be the leading-order term in the majority of standard modeling setups) in our $1$-Wasserstein bound is of the order of $\bar{C}_d\sqrt{d^2/n}$, where 
    \begin{align*}
\bar{C}_d=\frac{\Tr\left[\postfisherinformation(\map)^{-1}\right]\overline{M}_2}{d\cdot\left(\minevaluemap-\overline{M}_2\deltabar\right)}\leq \frac{\overline{M}_2}{\minevaluemap\left(\minevaluemap-\overline{M}_2\deltabar\right)}.
    \end{align*}
    Moreover, in the bound of \Cref{w2_map}, the first summand is of the order of $\tilde{C}_d\sqrt{d^3/n}$, where
    \begin{equation*}
    \tilde{C}_d\leq \frac{\overline{M}_2}{\minevaluemap^{3/2}\left(\minevaluemap-\overline{M}_2\deltabar\right)},
    \end{equation*}
and the second summand is of the order $\hat{C}_d\,d^2/n$, where
\begin{align*}
    \hat{C}_d\leq \frac{\overline{M}_2}{\minevaluemap^{2}\left(\minevaluemap-\overline{M}_2\deltabar\right)^2}
\end{align*}
and those two summands form together the term that will be of the leading order in the majority of standard modelling setups. 

In \Cref{sed:proof_techniques} we explain how our proof techniques allowed us to obtain the bounds of order $C_d\sqrt{d^2/n}$ and  $\bar{C}_d\sqrt{d^2/n}$ in \Cref{main_map,thm:wasserstein_map}, thus achieving a sharper dimension dependence as compared to all earlier works. In \Cref{sed:proof_techniques} we also explain why our covariance error bound in \Cref{w2_map} is of the order of $\tilde{C}_d\sqrt{d^3/n}$ and why we could not achieve the order of $\tilde{C}_d\sqrt{d^2/n}$ in \Cref{w2_map} using our techniques.

\subsection{Dependence on the data}
In the typical applications, the bounds in \Cref{main_map,thm:wasserstein_map,w2_map} depend on the data. For a practitioner, such dependence on data is very natural because the bound is computable using the data and therefore usable in practice.
However, for a theoretical statistician, interested in Bernstein--von Mises phenomena, it might be interesting to assume the data come from a certain distribution. For instance, researchers investigating frequentist properties of Bayesian estimators may make such an assumption in order to compare the coverage of Bayesian credible sets with that of frequentist confidence sets.  It might be interesting to make such an assumption for instance in order to investigate how fast the coverage of Bayesian credible sets approaches the coverage of frequentist confidence sets as $n\to\infty$.
Having made such an assumption, one can use our bounds in order to quantify the speed of almost sure convergence or convergence in probability of the distances between the prior and the Gaussian. To this end, one only needs to control the speed of the relevant mode of convergence of our bounds, which, in most cases, should be achievable using standard results, similar to those quantifying the rate of convergence in the law of large numbers.


\subsection{Computability}

Our bounds are computable from the constants and quantities
introduced in \Cref{section_notation}, such as $\hat{M}_1$, $M_2$,
$\overline{M}_2$, $\deltabar$, $\kappabar$, $\mle$, $\map$,
$\postfisherinformation(\map)$, $\fisherinformation(\mle)$. Computing some of
those constants and quantities, including the bounds on the third derivatives of the log-likelihood and
log-posterior given by $M_2$ and $\overline{M}_2$, may require additional work. It
is often possible to obtain such bounds analytically. Sharper bounds can
typically be obtained numerically. A robust approach to doing so is to obtain an
analytical bound on the fourth derivative and then run a grid-search optimizer
and apply the mean value theorem. Running a grid-search optimizer can, however, be very expensive
computationally, especially in high dimensions. Another option is to run a
simpler global optimizer. This is much faster but can
occasionally return results that are inaccurate or incorrect. Indeed, a global numerical optimizer might output a value smaller than the maximum, rather than producing the desired upper bound. There is,
therefore, a statistical--computational trade-off that needs to be taken
into consideration when calculating our bounds in practice, on real data sets. Moreover, there is typically a range of values for $\deltamle$ and $\deltabar$ for which the assumptions of \Cref{thm:wasserstein_map,w2_map,main_map} are satisfied. The user may find the optimal choice of $\deltamle$ and $\deltabar$ within the appropriate ranges by running a numerical optimizer on the bounds. We describe how we compute our bounds in practice in \cref{sec:logistic_setup} and \cref{calculations_poisson,calculations_weibull}.
We present our
bounds computed for several example Bayesian models in
\Cref{applications}.

\subsection{Our bounds control the quality of credible-set approximations}
Our bound on the total variation distance provides quality guarantees on the approximate computation of posterior credible sets. Indeed, suppose, for instance, that one is interested in finding a value $b_{\alpha}$, such that
$\mathbbm{P}\left(\|\postparam-\map\|\leq \frac{b_{\alpha}}{\sqrt{n}}\right)\geq1-\alpha$,
for a fixed value $\alpha$, where the probability $\mathbbm{P}$ is \textit{conditional on the observed data}. Let $A(n)$ denote the value of our upper bound in \Cref{main_map}.  If $n$ is sufficiently large and $A(n)$ is smaller that $\alpha$, then one could choose $b_{\alpha}=\tilde{b}_{\alpha}$, such that, for $\boldsymbol{Z}_n\sim\mathcal{N}(0,\postfisherinformation(\map)^{-1})$ ,
$\mathbbm{P}\left(\|\boldsymbol{Z}_n\|\leq \tilde{b}_{\alpha}\right)= 1-\alpha+A(n).$
Our bound implies that:
\begin{align*}
	\left|\mathbbm{P}\left(\|\postparam-\map\|\leq \frac{\tilde{b}_{\alpha}}{\sqrt{n}}\right) - \mathbbm{P}\left(\|\boldsymbol{Z}_n\|\leq \tilde{b}_{\alpha}\right)\right|\leq A(n)
\end{align*}
and so
$\mathbbm{P}\left(\|\postparam-\map\|\leq \frac{\tilde{b}_{\alpha}}{\sqrt{n}}\right)\geq1-\alpha.$

\subsection{Our bounds control the difference of means}
Our bound on the 1-Wasserstein distance controls the difference of means in the Laplace approximation, in the following way. The upper bound in \Cref{thm:wasserstein_map} controls $\sqrt{n}\|\mathbbm{E}[\postparam] - \map\|$. In order to obtain an upper bound on $\|\mathbbm{E}[\postparam] - \map\|$, one needs to divide our bound from \Cref{thm:wasserstein_map} by $\sqrt{n}$.

\subsection{Our bounds control the difference of covariances}
\Cref{w2_map} together with \Cref{thm:wasserstein_map} let us control the difference of covariances. Suppose, for instance, that we are interested in the operator norm of the difference of the posterior covariance matrix and the covariance matrix of the Laplace approximation. Let $B(n)$ denote the value of our bound from \Cref{thm:wasserstein_map} and  let $C(n)$ be the value of our bound from \Cref{w2_map}. Then,
a straightforward calculation reveals that:
\begin{align*}
	\left\|\text{Cov}(\postparam)-\frac{\postfisherinformation(\map)^{-1}}{n}\right\|_{op}\leq \frac{1}{n}\left(B(n)^2+C(n)\right).
\end{align*}

\section{Our proof techniques}\label{sed:proof_techniques}

\def\postmeas{\mu_{n}^{\pi}}
\def\normmeas{\mu_{n}^{\mathcal{N}}}
\def\postdens{f_{n}^{\pi}}
\def\normdens{f_{n}^{\mathcal{N}}}
Now, we briefly describe our proof strategies. We discuss how we prove our theory for the MAP-centric approach (in \cref{main_map,thm:wasserstein_map,w2_map} above), the MLE-centric approach (in \cref{theorem_main,theorem_1wasserstein,theorem_2wasserstein} below), and additional theory for general test functions in one dimension (in \cref{theorem_univariate} in the appendix).

All the proofs of our results are provided in full detail in \cref{introductory_arguments,appendix_b,appendix_c,appendix_d}.

\subsection{The MAP-Centric Approach}
In our proofs for the MAP-centric approach, we compare the distribution of
$\sqrt{n}\left(\postparam-\map\right)$ to
$\mathcal{N}\left(0,\postfisherinformation(\map)^{-1}\right)$ by concentrating
separately on the region $\{\param:\|\param\|\leq\deltabar\sqrt{n}\}$ and the
region $\{\param:\|\param\|>\deltabar\sqrt{n}\}$.

In order to compare the rescaled posterior with the Gaussian in the outer region
$\{\param:\|\param\|>\sqrt{n}\deltabar\}$, we simply upper bound tail
integrals with respect to the posterior and with respect to the Gaussian.  We
use \cref{assump_kappa1}, together with Gaussian concentration
inequalities, which let us upper-bound integrals with respect to the posterior
with integrals with respect to the prior, multiplied by $n^{d/2}e^{-n\bar{\kappa}}$
and suitable constants. Integrals with respect to the Gaussian are upper-bounded
using Gaussian concentration inequalities.

With the tail integrals controlled, we can focus on the two distributions
truncated to the region $\{\param:\|\param\| \le \sqrt{n}\deltabar\}$.  As we
will describe in more detail shortly, we prove our main results, 
\cref{main_map,thm:wasserstein_map,w2_map} by (1) controlling the Fisher
divergence using a Taylor series expansion, (2) controlling the Kullback-Leibler (KL) divergence
using strong log-concavity in the inner region together with the log-Sobolev
inequality, and (3) controlling the total variation and Wasserstein distances
in terms of the KL divergence using Pinsker's inequality and the
transportation-entropy (Talagrand) inequality.

We use the following notation to represent the truncated versions of the two
probability measures. First, for any probability measure $\mu$ and measurable set $A$, let $[\mu]_A$ denote $\mu$ truncated to set $A$. $[\mu]_A$ is constructed by restricting $\mu$ to set $A$ and then renormalizing it so that $[\mu]_A$ is a well-defined probability measure.  Now, let $B_0\left(\deltabar\sqrt{n}\right) :=
\{\param:\|\param\|\leq\deltabar\sqrt{n}\}$ denote the inner region,
and define the truncated posterior and normal measures respectively as
$\postmeas:=\left[\mathcal{L}\left(\sqrt{n}\left(\postparam-\map\right)\right)\right]_{B_0\left(\deltabar\sqrt{n}\right)} $
and
$\normmeas=\left[\mathcal{N}\left(0,\postfisherinformation(\map)^{-1}\right)\right]_{B_0\left(\deltabar\sqrt{n}\right)}$,
with densities with respect to the Lebesgue measure given respectively by
$\postdens$ and $\normdens$.
By Taylor expanding $(\log \postdens)'$ around $\map$ and applying 
\cref{assump2} we can
successfully control the following \textit{Fisher divergence} between the two truncated
distributions:
\begin{align}\label{fisher_control}
\text{Fisher}\left(\normmeas\|\postmeas\right):=\int
\left\|\left( \log \postdens\right)'(t)-\left( \log
\normdens\right)'(t)\right\|^2 \normmeas(dt)
\leq\frac{3\left(\Tr\left[\jbar^{-1}\right]\overline{M}_2\right)^2}{4n}.
\end{align}
%

%
\sloppy Now, by \cref{assump2,assump7}, combined with
Taylor's theorem, we show that the density of $\postmeas$ is
$\left(\minevaluemap-\deltabar\overline{M}_2\right)$-strongly log-concave on $B_0\left(\deltabar\sqrt{n}\right)$
(i.e. the logarithm of its density is a strongly concave function). The
celebrated \textit{Bakry-\'Emery} criterion  \citep{bakry-emery} (see also
\citealt[Theorem A1]{Schlichting:2019}) says that any strongly log-concave measure
satisfies the \textit{log-Sobolev inequality}, which is an inequality between
the KL divergence (or relative entropy) and the Fisher divergence (see
\citealt[Chapter 5]{bakry_gentil_ledoux} for a comprehensive treatment of the topic
or \citealt[Section 2]{vempala_wibisono} for a brief and intuitive explanation of
the main ideas). Specifically, in our context, the log-Sobolev inequality for
$\postmeas$ on the set $B_0\left(\deltabar\sqrt{n}\right)$
implies that:
%
%
%
\begin{align}\label{kl_control}
\text{KL}[\normmeas||\postmeas] :=
\int \log\left(\frac{\normdens(x)}{\postdens(x)}\right) \normmeas(dx)
\leq \frac{\text{Fisher}(\normmeas\|\postmeas)}{2\left(\minevaluemap-\deltabar\overline{M}_2\right)}.
\end{align}
 Note that we could also use the log-Sobolev inequality for $\normmeas$ and obtain an bound on $\text{KL}(\postmeas||\normmeas)$ expressed in terms of $\text{Fisher}(\postmeas\|\normmeas)$. However, such a bound would not be as useful as it would involve integrating with respect to $\postmeas$, which is often intractable. This is why we intentionally use the log-Sobolev inequality for $\postmeas$, which yields the upper bound in \cref{kl_control} expressed in terms of $\text{Fisher}(\normmeas\|\postmeas)$, thus involving only Gaussian integration. Because of that, we derive a bound expressed in terms of the right-hand side of \cref{fisher_control}.

Finally, we can then control the total variation and Wasserstein distances in
terms of the KL divergence. By Pinsker's inequality \cite[Theorem 2.16]{massart}
we have that
\begin{align*} \text{TV}\left(\postmeas,\normmeas\right)\leq
\sqrt{\frac{1}{2}\text{KL}[\normmeas||\postmeas]},
\end{align*}
which gives us
the desired control over the total variation distance. In order to control the
1-Wasserstein distance and the integral probability metric introduced in
\cref{w2_map}, we bound the 2-Wasserstein distance given by:
\begin{align*}
\text{W}_2\left(\postmeas,\normmeas\right) :=
    \inf_{\Gamma}\sqrt{\mathbbm{E}_\Gamma\left[\|X-Y\|^2\right]},
\end{align*}
where the infimum is taken over all distributions $\Gamma$ of
$(X,Y)$ with the correct marginals $X\sim\postmeas$ and $Y\sim \normmeas$.
Indeed, the 2-Wasserstein distance is known to upper bound the 1-Wasserstein
distance and a transformation of it upper-bounds the integral probability metric
appearing in \cref{w2_map}. We upper-bound the 2-Wasserstein distance by
the KL divergence, which we have previously controlled in
\cref{fisher_control}. We use the \textit{transportation-entropy (Talagrand)
inequality}, which is implied by the log-Sobolev inequality (see
\citealt[Theorem 4.1]{gozlan} for the specific result we use or \citealt[Section
2.2.1]{vempala_wibisono} for an intuitive explanation). In our context, this
inequality says that
\begin{align*} &\text{W}_2\left(\postmeas,\normmeas\right)
\leq
\sqrt{\frac{2\text{KL}\left[\normmeas||\postmeas\right]}{\minevaluemap-\deltabar\overline{M}_2}}.
\end{align*}
Combined with \cref{kl_control} it yields the desired control over the
1-Wasserstein distance and the metric considered in \cref{w2_map}, inside
the inner region $B_0\left(\deltabar\sqrt{n}\right)$.

It is precisely the bounds on the total-variational distance and the 1-Wasserstein distance between the truncated measures $\postmeas$ and $\normmeas$ that yield the leading summands of orders $C_d\sqrt{d^2/n}$ and $\bar{C}_d\sqrt{d^2/n}$  in the bounds in \Cref{main_map,thm:wasserstein_map} respectively, described in detail in \Cref{dimension_dependence}. The leading summand in \Cref{w2_map} is of order $\tilde{C}_d\sqrt{d^3/n}$ (using the notation of \Cref{dimension_dependence}) because of the following part of our proof, which can be found in full detail in \Cref{sec:control_i1map}. Suppose that $X\sim\postmeas$ and $Y\sim\normmeas$. Consider any vector $v\in\mathbbm{R}^d$ with $\|v\|\leq 1$ and any coupling $(\tilde{X},\tilde{Y})$ between $\postmeas$ and $\normmeas$. Then,
\begin{align*}
    \left|\mathbbm{E}\left[\left<v,X\right>^2\right]-\mathbbm{E}\left[\left<v,Y\right>^2\right]\right|=&\left|\mathbbm{E}\left[\left<v,\tilde{X}-\tilde{Y}\right>\left<v,\tilde{X}+\tilde{Y}\right>\right]\right|\\
    =&\left|\mathbbm{E}\left[\left<v,\tilde{X}-\tilde{Y}\right>^2\right]+2\mathbbm{E}\left[\left<v,\tilde{X}-\tilde{Y}\right>\left<v,\tilde{Y}\right>\right]\right|\\
	\leq&\mathbbm{E}\left[\left\|\tilde{X}-\tilde{Y}\right\|^2\right]+2\sqrt{\mathbbm{E}\left[\left\|\tilde{X}-\tilde{Y}\right\|^2\right]}\sqrt{\mathbbm{E}\left[\|Y\|^2\right]},
\end{align*}
from which it follows that
\begin{align*}
\sup_{\|v\|\leq1}\left|\mathbbm{E}\left[\left<v,X\right>^2\right]-\mathbbm{E}\left[\left<v,Y\right>^2\right]\right|\leq \left(\text{W}_2\left(\postmeas,\normmeas\right)\right)^2+2\text{W}_2\left(\postmeas,\normmeas\right)\sqrt{\mathbbm{E}\left[\|Y\|^2\right]}.
\end{align*}
We obtain a bound on $\text{W}_2\left(\postmeas,\normmeas\right)$ that is of order $\bar{C}_d\sqrt{d^2/n}$ (and appears as the leading summand in the final bound in \Cref{thm:wasserstein_map}). Yet, at the same time, $\sqrt{\mathbbm{E}\left[\|Y\|^2\right]}$ is of order $\frac{\sqrt{d}}{\minevaluemap}$. As a result, the leading term in the bound of \Cref{w2_map} is of order $\tilde{C}_d\sqrt{d^3/n}$.

\subsection{The MLE-Centric Approach}
\cref{theorem_main,theorem_1wasserstein,theorem_2wasserstein} establish our bounds in the MLE-centric approach. Our proof technique is very similar to the strategy we used in the MAP-centric approach. We compare the distribution of
$\sqrt{n}\left(\postparam-\mle\right)$ to
$\mathcal{N}\left(0,\fisherinformation(\mle)^{-1}\right)$ by concentrating
separately on the region $\{\param:\|\param\|\leq\deltamle\sqrt{n}\}$ and the
region $\{\param:\|\param\|>\deltamle\sqrt{n}\}$. In the outer region, we use \cref{assump_kappa} and Gaussian concentration inequalities. In the inner region, we use the log-Sobolev inequality.

\subsection{Our theory for general test functions in one dimension}
\cref{theorem_univariate} in the appendix establishes our bounds for general test functions in a single dimension. Its proof also divides the domain into an inner and outer
region, but proceeds using Stein's method instead of via the Fisher
divergence, using known properties of solutions to the Stein equation in one
dimension.  As with our control over the Fisher divergence, the key to our proof is using Taylor's expansions and
expressing the bound as an integral over the known measure $\normmeas$.

\section{Further results: MLE-centered approach}\label{further_results}
In this section we present some further results on the quality of Laplace approximation in the \textit{MLE-centric} approach. The proofs of \Cref{theorem_main,theorem_1wasserstein,theorem_2wasserstein} presented below can be found in \Cref{introductory_arguments,appendix_c}.

\subsection{Control over the TV distance in the MLE-centric approach}
We start by controlling the total variation distance.
\MainThm{}

\subsection{Control over the $1$-Wasserstein distance in the MLE-centric approach}
\OneWassThm{}
\begin{remark}
	Note that, in order for the bound in \Cref{theorem_1wasserstein} to be finite, we need the following integral with respect to the prior: $\int_{\|u-\mle\|>\deltamle}\|u-\mle\|\pi(u)du$ to be finite.
\end{remark}

\subsection{Control over the difference of covariances in the MLE-centric approach}
\TwoWassThm{}
\begin{remark}
	Note that, in order for the bound in \Cref{theorem_2wasserstein} to be finite, we need the following integral with respect to the prior to be finite: $\int_{\|u\|>\deltamle}\|u\|^2\pi(u+\hat{\theta}_n)du$. 
\end{remark}

\section{Example applications}\label{applications}

Now, we present examples of our bounds computed for different models.\footnote{The code for our experiments is available at \url{https://github.com/mikkasprzak/laplace_approximation.git}} None of
these examples correspond to strongly log-concave posteriors, making them challenging for the Laplace approximation. Indeed, the first
one yields a weakly log-concave posterior, and the second and third one produce
non-log-concave posteriors. They all show that our bounds are explicitly
computable for a variety of commonly used models, including heavy-tailed ones, to which the methods currently available in the literature do not apply.
Our examples also show that the computations may be executed in a practical amount of
time. Our bounds also go well below the true values of the mean and the norm of the covariance, for reasonable sample sizes. This comparison indicates that they are applicable for practitioners who wish to assess how confident they should be in their mean and variance estimates.

\subsection{Our bounds work under misspecification: Poisson likelihood with gamma prior and exponential data}\label{poisson_example}

First, we look at a one-dimensional conjugate model, for which we can compare
our bounds on the difference of means and the difference of variances to the
ground truth. We consider a Poisson likelihood and a gamma prior with shape
equal to $0.1$ and rate equal to $3$, which yields a weakly log-concave posterior. Our data are generated from the
exponential distribution with mean $10$. \Cref{fig:mean_gamma,fig:variance_gamma} show that our bounds (in the MAP-centric approach) on the difference of means and the
difference of variances get close to the true difference of means and the true difference of variances  for
sample sizes above around  $250$. More detail on how one can compute the
constants appearing in the bounds can be found in
\Cref{calculations_poisson}.

\begin{figure}
	\begin{center}
		\begin{subfigure}{0.45\textwidth}
			\includegraphics[width=\textwidth]{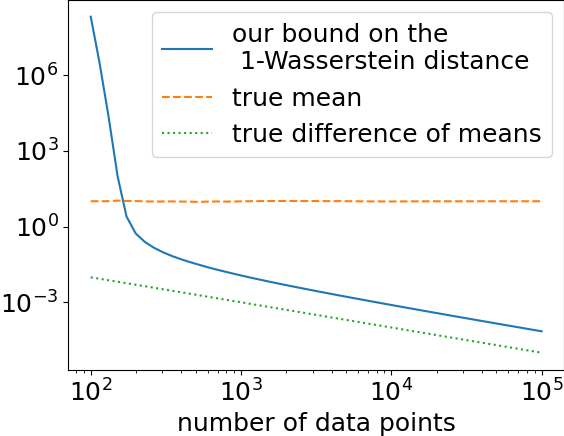}
			\caption{Posterior mean}
			\label{fig:mean_gamma}
		\end{subfigure}
		\hspace{1cm}
		\begin{subfigure}{0.45\textwidth}
			\includegraphics[width=\textwidth]{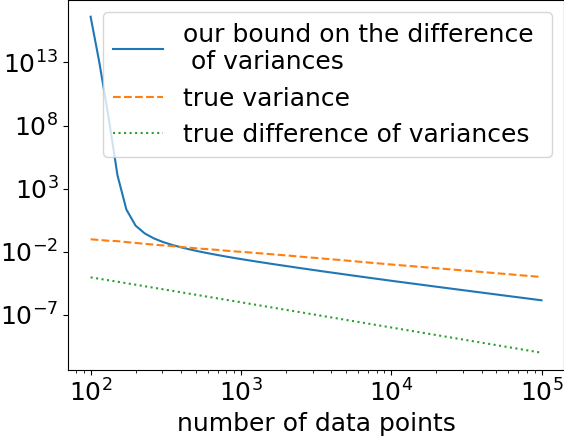}
			\caption{Posterior variance}
			\label{fig:variance_gamma}
		\end{subfigure}
	\end{center}
	\caption{Poisson likelihood with gamma prior and exponential data (MAP-centric approach)}
\end{figure}

\subsection{Our bounds work for heavy-tailed posteriors: Weibull likelihood with inverse-gamma prior}\label{weibull_example}

Now, we consider another conjugate model and compare our bounds to the ground
truth. In this case, the posterior is heavy-tailed, which represents one common way that non-log-concavity arises in practice. In our experiment, we set
the shape of the Weibull to $\frac{1}{2}$, and we make inference about the scale.
The prior is inverse-gamma with shape equal to $3$ and scale equal to $10$. The
data are Weibull with shape $1/2$ and scale $1$. \Cref{fig:mean_weibull,fig:variance_weibull} demonstrate that our bounds on the difference of means
and the difference of variances (in the MAP-centric approach) get close to the
true difference of means and the true difference of variances for sample sizes
in low thousands. Our bounds take large values for small sample sizes --- a phenomenon explained in detail in \cref{sec:sample_size_dep} above. More detail on how one can compute the constants appearing in
the bounds can be found in \Cref{calculations_weibull}.

\begin{figure}
	\begin{center}
		\begin{subfigure}{0.45\textwidth}
			\includegraphics[width=\textwidth]{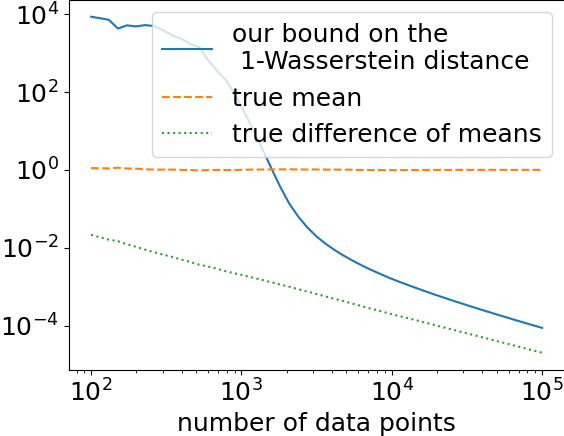}
			\caption{Posterior mean}
			\label{fig:mean_weibull}
		\end{subfigure}
		\hspace{1cm}
		\begin{subfigure}{0.45\textwidth}
			\includegraphics[width=\textwidth]{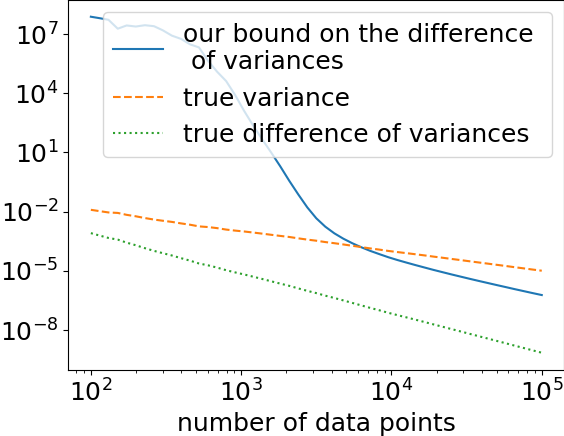}
			\caption{Posterior variance}
			\label{fig:variance_weibull}
		\end{subfigure}
	\end{center}
	\caption{Weibull likelihood with inverse-gamma prior (MAP-centric approach)}
\end{figure}

\subsection{Our bounds work for multivariate non-log-concave posteriors: logistic regression with Student's t prior}\label{logistic_example}

While the examples above have a one-dimensional parameter, we next show that our bounds are computable and non-vacuous for a multivariate posterior. Like the Weibull example above, it is not log-concave. But in  this case, the non-log-concavity arises closer to the central mass of the distribution rather than in the tails.

\subsubsection{Setup} \label{sec:logistic_setup}
Suppose $X_1,\dots,X_n\in\mathbbm{R}^d$ and $Y_1,\dots,Y_n\in\{-1,1\}$.  We will study the following log-likelihood:
\begin{align}\label{logistic_likelihood}
	\loglikelihood(\param):=\loglikelihood\left(\param|\left(Y_i\right)_{i=1}^n,\left(X_i\right)_{i=1}^n\right)=-\sum_{i=1}^n\log(1+e^{-X_i^T\param Y_i}).
\end{align}
For a covariance matrix $\Sigma$, a vector $\mu\in\mathbbm{R}^d$ and the number of degrees of freedom $\nu>0$, we consider a $d$-dimensional Student's t prior on $\param$, given by
\begin{align}\label{t_prior}
	\prior(\param)=\frac{\Gamma((\nu+d)/2)}{\Gamma(\nu/2)\nu^{d/2}\pi^{d/2}|\Sigma|^{1/2}}\left[1+\frac{1}{\nu}(\theta-\mu)^T\Sigma^{-1}(\theta-\mu)\right]^{-(\nu+d)/2}.
\end{align}
Combining the log-likelihood given by \cref{logistic_likelihood} with the prior given by \cref{t_prior} yields a posterior that is known not to be log-concave. In \Cref{example_detail_logistic} we compute analytically the constants arising in our bounds in the case of the posterior coming from the likelihood given in \cref{logistic_likelihood} and prior given in \cref{t_prior}.  In our present experiments, we choose the Student's t prior to have mean zero, identity covariance, and four degrees of freedom. 

We performed our experiments for the MAP-centric approach. We derived the values for the MAP and MLE numerically; recall that we use both the MAP and MLE in our MAP-centric bounds. Moreover, in order to sharpen our bounds, we ran a  built-in global optimizer scipy.optimize.shgo to derive the maximum of the third derivative inside a ball around the MLE and the MAP, i.e. to derive $M_2$ and $\overline{M}_2$. As there was a certain degree of choice for $\deltamle$ and $\deltabar$, we also optimized the choice thereof numerically. The data $(Y_i)_{i=1}^n$ we used came from logistic regression with parameter $(1,1,1,1,1)$, where we simulated the $(X_i)_{i=1}^n$ i.i.d.\ from the $5$-dimensional normal distribution with mean $(0,0,0,0,0)$ and covariance $\frac{0.15}{\sqrt{5}}I_{5\times 5}$, where $I_{5\times 5}$ is the $5\times 5$ identity matrix. We compared our bound from \Cref{thm:wasserstein_map} with the Euclidean norm of the difference between the MAP and the simulated true posterior mean. We obtained our simulated true posterior mean through MCMC, using the standard \textit{Stan} implementation (with 4 parallel chains). Similarly, we compared our bound from \Cref{w2_map} with the operator norm of the difference between $\frac{\postfisherinformation(\map)^{-1}}{n}$ and the simulated true posterior covariance, obtained again through the standard \textit{Stan} MCMC implementation.

\subsubsection{Results}\label{subsec_experim}

\Cref{fig:logistic_mean_true_comparison,fig:logistic_covariance_true_comparison} demonstrate that, in this case, our bounds on the 1-Wasserstein distance and the 2-norm of the difference of covariances tightly control the true difference of means and the true difference of covariances, respectively, for sample sizes of about $10{,}000$ or higher. Our bounds are still computable and accurate at lower sample sizes, but they can become loose enough as to lose practicality.

\subsection{Our bounds become looser and more difficult to compute as parameter dimension increases: logistic regression with Student's t prior}\label{logistic_dimension}

We use a similar setup as in \cref{sec:logistic_setup}, but now we vary the dimension $d$ between 1 and 9. We find that increasing $d$ loosens our bounds and makes them more difficult to compute. But we still find that they are computable and tight for reasonable data set sizes.

Specifically, we follow the setup in \Cref{sec:logistic_setup} except that now we simulate $(X_i)_{i=1}^n$ i.i.d.\ from the $d$-dimensional normal distribution with mean $(0,\dots,0)$ and covariance $\frac{0.15}{\sqrt{d}}I_{d\times d}$. We simulate the corresponding $(Y_i)_{i=1}^n$ from $d$-dimensional logistic regression with parameter $(0,\dots,0)$. We compare our bounds across different dimensions in \Cref{fig:logistic_dimension_comparison}. While we see that the bounds generally increase with dimension, we also see that they remain practically small for some portion of the plotted data range for every dimension we consider.

There are two ways in which our bounds become more difficult to compute with increasing dimension. First, we cover the ability to compute the bounds. Second, we cover computational speed in cases where the bounds are possible to compute.

On the first point, we are able to compute our bounds when \cref{assump7} is satisfied.
The sample sizes where \cref{assump7} is satisfied vary by dimension. In \Cref{fig:logistic_dimension_comparison}, we plot our bounds for every number of data points where \cref{assump7} is satisfied. 
For the same data sets and priors, we calculated the minimum number of data points required to satisfy \cref{assump7}, for the same range of dimensions: $d$ between $1$ and $9$. See \Cref{fig:logistic_minimal_data}. We find that the minimum number of data points remains practical in this range.

\begin{figure}
	\begin{center}
		\begin{subfigure}{0.45\textwidth}
			\includegraphics[width=\textwidth]{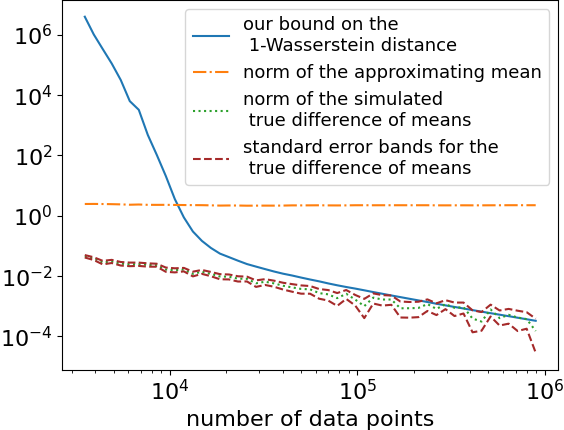}
			\caption{Posterior mean}
			\label{fig:logistic_mean_true_comparison}
		\end{subfigure}
		\hspace{1cm}
		\begin{subfigure}{0.45\textwidth}
			\includegraphics[width=\textwidth]{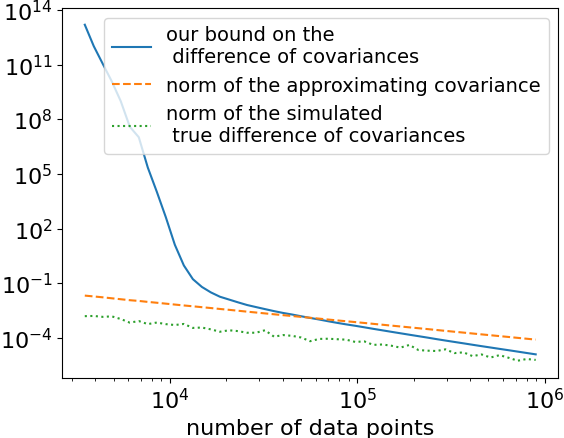}
			\caption{Posterior covariance}
			\label{fig:logistic_covariance_true_comparison}
		\end{subfigure}
	\end{center}
	\caption{Our bounds for 5-dimensional logistic regression with t prior (MAP-centric approach) compared to the simulated true difference of means and covariances (with Monte Carlo standard error bands included for the posterior mean). The data $(Y_i)_{i=1}^n$ come from logistic regression with parameter $(1,1,1,1,1)$, where the $(X_i)_{i=1}^n$ are simulated i.i.d.\ from the $5$-dimensional normal distribution with mean $(0,0,0,0,0)$ and covariance $\frac{0.15}{\sqrt{5}}I_{5\times 5}$.}
	\label{figure3}
\end{figure}

\begin{figure}
	\begin{center}
		\begin{subfigure}{0.45\textwidth}
			\includegraphics[width=\textwidth]{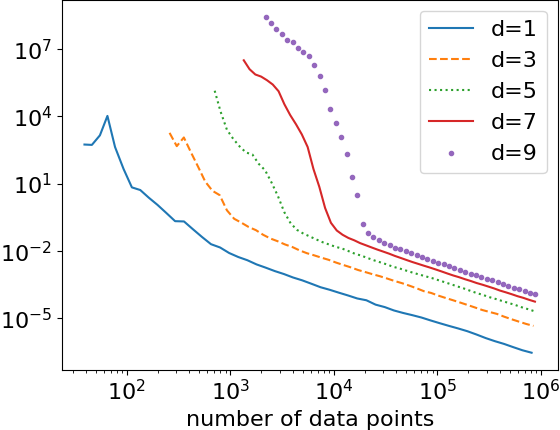}
			\caption{Our bound on 1-Wasserstein distance}
		\end{subfigure}
		\hspace{1cm}
		\begin{subfigure}{0.45\textwidth}
			\includegraphics[width=\textwidth]{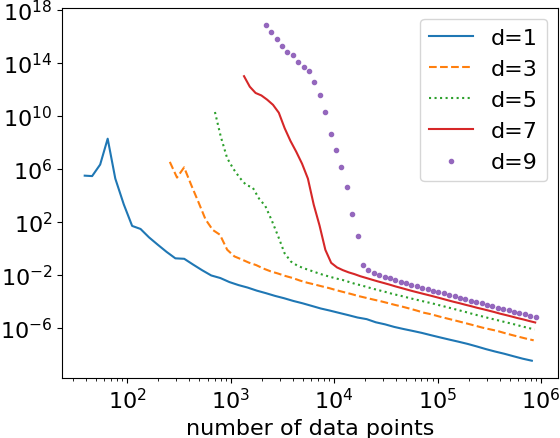}
			\caption{Our bound on difference of covariances}
		\end{subfigure}
	\end{center}
	\caption{Comparison of our bounds for the logistic regression with t prior (MAP-centric approach) for different dimensions d. The data $(Y_i)_{i=1}^n$ come from logistic regression with parameter $(0,\dots,0)$, where the $(X_i)_{i=1}^n$ are simulated i.i.d.\ from the normal distribution with mean $(0,\dots,0)$ and covariance $\frac{0.15}{\sqrt{d}}I_{d\times d}$.}\label{fig:logistic_dimension_comparison}
\end{figure}

\begin{figure}
	\begin{center}
		\includegraphics[width=0.5\textwidth]{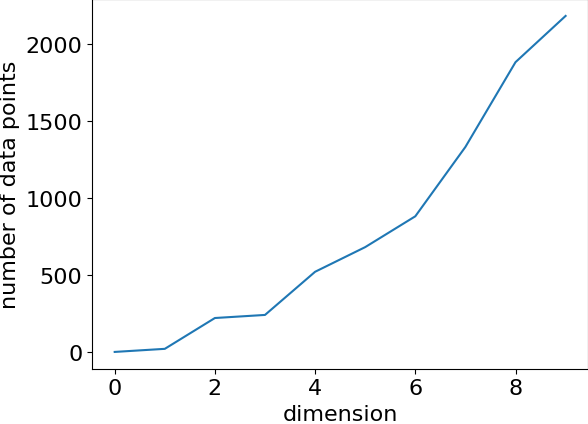}
		\caption{Minimum number of data points required to compute our bounds for logistic regression with t prior (MAP-centric approach). The data $(Y_i)_{i=1}^n$ come from logistic regression with parameter $(0,\dots,0)$, where the $(X_i)_{i=1}^n$ are simulated i.i.d.\ from the normal distribution with mean $(0,\dots,0)$ and covariance $\frac{0.15}{\sqrt{d}}I_{d\times d}$.}\label{fig:logistic_minimal_data}
	\end{center}
\end{figure}

\begin{figure}
	\begin{center}
		\begin{subfigure}{0.45\textwidth}
			\includegraphics[width=\textwidth]{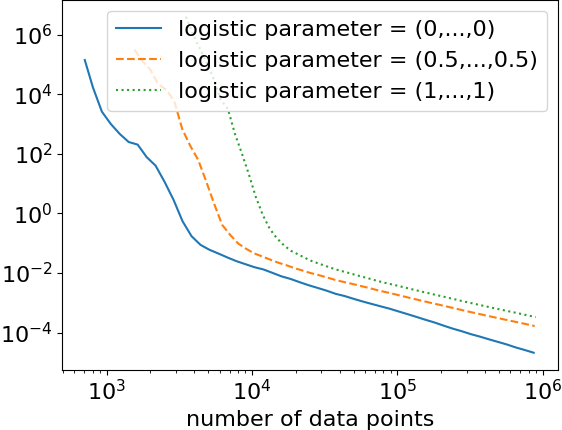}
			\caption{Our bound on 1-Wasserstein distance}
                \label{fig:logistic_condition_mean}
		\end{subfigure}
		\hspace{1cm}\vspace{1cm}
		\begin{subfigure}{0.45\textwidth}
			\includegraphics[width=\textwidth]{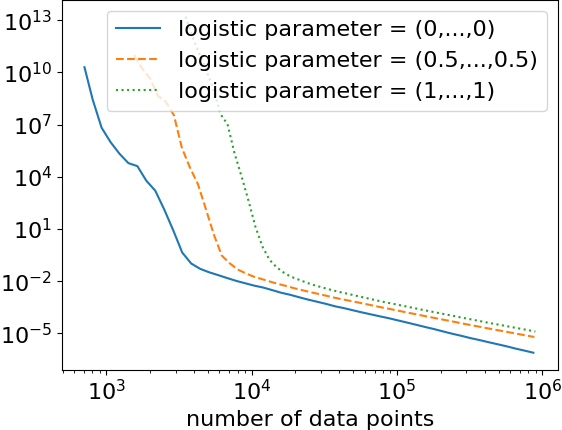}
			\caption{Our bound on difference of covariances}
                \label{fig:logistic_condition_covariance}
		\end{subfigure}
        		
        \begin{subfigure}{0.45\textwidth}
			\includegraphics[width=\textwidth]{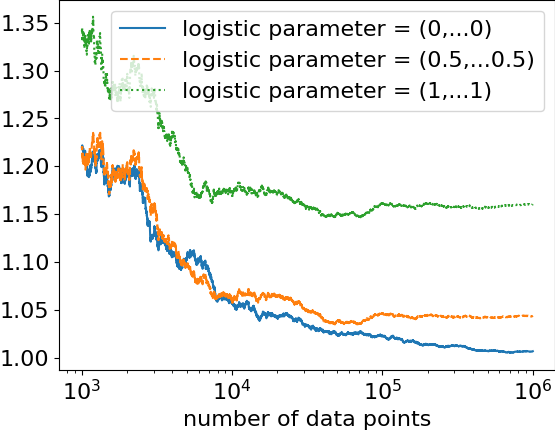}
			\caption{Condition number}
                \label{fig:logistic_condition_number}
		\end{subfigure}
	\end{center}
	\caption{Comparison of our bounds for the 5-dimensional logistic regression with t prior (MAP-centric approach) for different data sets and the comparison of the condition number calculated for those data sets. The data  $(X_i)_{i=1}^n$ were in each case simulated i.i.d.\ from the $5$-dimensional normal distribution with mean $(0,0,0,0,0)$ and covariance $\frac{0.15}{\sqrt{5}}I_{5\times 5}$. The corresponding $(Y_i)_{i=1}^n$ were simulated from logistic regression with parameter $(0,0,0,0,0)$ in the first case, $(0.5,0.5,0.5,0.5,0.5)$ in the second case and $(1,1,1,1,1)$ in the third case.}
\end{figure}

When we are able to compute our bounds, computation still slows down as dimension increases due to
the numerical optimization involved in calculating $\overline{M}_2$. The slowdown is mainly caused by the global optimizer scipy.optimize.shgo becoming slow in higher dimensions. To create \Cref{fig:logistic_dimension_comparison}, we computed our bounds on the 1-Wasserstein distance and the 2-norm of the difference of covariances, simultaneously, on 51 points spread uniformly on the logarithmic scale between the minimum sample size given by \Cref{fig:logistic_minimal_data} and the sample size of $10^6$. Performing this experiment on a MacBook Pro with the Apple M3 Max Chip and 16 cores in total for $d=1$ took us about $1.5$ total minutes, for $d=3$ about $3.35$ minutes, for $d=5$ about $10$ minutes, for $d=7$ about $36$ minutes and for $d=9$ about $147$ minutes.

\subsection{Are our bounds smaller for small condition numbers? Logistic regression with Student's t prior}\label{sec:are_bounds_smaller}

In all of our experiments above, we see that our bounds are tighter for larger sample sizes and become loose at small numbers of data points; see \Cref{sec:sample_size_dep}. We hypothesized that the looseness of our bounds might be a function of condition number (defined as the ratio of the largest and the smallest eigenvalue of $\postfisherinformation(\map)$). To test this hypothesis, we calculated and compared our bounds from \Cref{thm:wasserstein_map,w2_map} for $5$-dimensional logistic regression on three different datasets. As before, the prior was Student's t with $4$ degrees of freedom, mean zero, and identity covariance. We simulated the data  $(X_i)_{i=1}^n$ i.i.d.\ from the $5$-dimensional normal distribution with mean $(0,0,0,0,0)$ and covariance $\frac{0.15}{\sqrt{5}}I_{5\times 5}$. We simulated the corresponding $(Y_i)_{i=1}^n$ from logistic regression but now with three different parameter values: $(0,0,0,0,0)$ in the first case, $(0.5,0.5,0.5,0.5,0.5)$ in the second case, and $(1,1,1,1,1)$ in the third case. In \Cref{fig:logistic_condition_mean,fig:logistic_condition_covariance}, we plot how our bound varies both by number of data points and by the logistic parameter. In \Cref{fig:logistic_condition_number} we plot how the condition number varies by number of data points and logistic parameter. As expected, the condition number is higher for smaller numbers of data points in \Cref{fig:logistic_condition_number}, and that range of data points roughly corresponds to the range where our bound is looser in \Cref{fig:logistic_condition_mean,fig:logistic_condition_covariance}. But we do not see that condition number is decisive. For instance, we see that our bound is loose for condition numbers above 1.15 on the orange curve (parameter $(0.5,0.5,0.5,0.5,0.5)$); compare \Cref{fig:logistic_condition_mean,fig:logistic_condition_covariance} and  \Cref{fig:logistic_condition_number}. But we also see that the condition numbers for the green curve (parameter $(1,1,1,1,1)$) are  generally at or above $1.15$ (\Cref{fig:logistic_condition_number}), but our bounds are still tight for high sample sizes (\Cref{fig:logistic_condition_mean,fig:logistic_condition_covariance}).
See \Cref{sec:sample_size_dep} for an additional discussion.

\section{Conclusions and future work}\label{conclusions}
We provide bounds on the quality of the Laplace approximation that are computable and hold under the standard assumptions of the Bernstein--von Mises Theorem. We control the total-variation distance, the 1-Wasserstein distance, and another useful integral probability metric. Our bounds on the total-variation and 1-Wasserstein distance are such that their sample-size and dimension dependence cannot be improved in general. However, the constants in our bounds are evidently sub-optimal, which makes our bounds unnecessarily large for small and moderate sample sizes. We hope that the values of those constants may be reduced in the future.

An interesting question is whether, for multivariate posteriors, we could derive bounds on more general integral probability metrics, in a way similar to our univariate \Cref{theorem_univariate}.  We proved \Cref{theorem_univariate} using Stein's method and we present it as part of the appendix. Whether or not this result could be extended to the multivariate context is to a certain extent a question about the applicability of techniques like Stein's method in dimension greater than one for measures truncated to a bounded convex set. So far we have struggled to find enough theory that would let us compute useful bounds using this approach.

\acks{The authors would like to thank Krishna Balasubramanian, Xiao Fang, Max Fathi, Jonathan Huggins, Anya Katsevich, Tin D. Nguyen, Jeff Miller, Daniel Paulin, Gesine Reinert, H\r{a}vard Rue, Adil Salim, Vladimir Spokoiny and Yvik Swan for interesting discussions. We are particularly grateful to Anya Katsevich for a discussion concerning the dimension dependence.
	This work is part of project \textbf{Stein-ML} that has received funding from the European Union’s Horizon 2020 research and innovation program under the Marie Skłodowska-Curie grant agreement \textbf{No 101024264}. Additionally, this work was supported by France 2030. This work was also supported in part by an NSF Career
	Award and an ONR Early Career Grant.}


\appendix
\begin{appendix}

\section{An additional general univariate bound}\label{sec:univariate}
Now we present the following bound, which works for pretty arbitrary test functions, but only in dimension one. The proof of this bound is different from the proofs of the results presented in the main body of the paper as it relies on Stein's method instead of the log-Sobolev inequality. Stein's method allows us to achieve control over the difference of expectations with respect to the rescaled posterior and with respect to the approximating Gaussian of very general test functions. The particular properties of Stein's method we use, however, only yield bounds in dimension one. We believe it would be interesting and useful to derive such general bounds in higher dimensions in the future. The proof of \Cref{theorem_univariate} presented below can be found in \Cref{introductory_arguments,appendix_d}.
\UnivariateThm{}
\begin{remark}
    The bound in \Cref{theorem_univariate} is for the MLE-centric approach. However, its proof can easily be modified in order to yield analogous bounds on the quality of approximation in the MAP-centric approach, as described in \Cref{section_notation} and presented in \Cref{s:main}.

Constants $G_1,G_2,G_4,G_5$ appearing above may be straightforwardly controlled by controlling the involved Gaussian integrals. Moreover, the fact that $C_n^{(1)}, C_n^{(2)}, C_n^{(3)}, C_n^{(4)}$ are positive follows from \Cref{assump_fisher}.
\end{remark}
\section{Introduction to the Laplace approximation and the Bernstein--von Mises theorem}\label{sec:bvm}
The foundations of Laplace approximation date back to the work of \cite{laplace} (see \citealt{laplace_english} for an English translation and \citealt{bach} for an intuitive discussion). It was originally introduced as a method of approximating integrals of the form
\begin{align*}
	\text{Int}(n):=\int_K e^{-nf(x)}dx,\quad n\in\mathbbm{N},
\end{align*}
where $K$ is a subset of $\mathbbm{R}^d$ and $f$ is a real-valued function on $\mathbbm{R}^d$. Suppose that $x^*\in K$ is a strict global maximizer of $f$ on $K$. Heuristically, under appropriate smoothness assumptions on $f$, we can use Taylor's expansion to obtain:
\begin{align*}
f(x)\approx &f(x^*)+f'(x^*)(x-x^*)+\frac{1}{2}(x-x^*)^Tf''(x^*)(x-x^*)\\
=& f(x^*)+\frac{1}{2}(x-x^*)^Tf''(x^*)(x-x^*).
\end{align*}
We therefore have:
\begin{align*}
	\text{Int}(n)\approx&\int_K\exp\left[-nf(x^*)-\frac{n}{2}(x-x^*)^Tf''(x^*)(x-x^*)\right]dx
\end{align*}
 As a result, heuristically, up to a constant not depending on $K$, $	\text{Int}(n)$ can be approximated by the integral of the density of the Gaussian measure with mean $x^*$ and covariance matrix given by $\frac{1}{n}f''(x^*)^{-1}$. Now, suppose that $e^{-nf(x)}$ is an unnormalized posterior density, i.e. that $\Pi_n(\cdot)\propto e^{-nf(\cdot)}$ for some posterior $\Pi_n$.
Writing $\boldsymbol{\Pi}_n$ for the posterior probability measure and $\map$ for the posterior mode (i.e. the maximum a posteriori or MAP),  the above statement suggests that
\begin{align}\label{laplace_approximation}
	\boldsymbol{\Pi}_n\approx\mathcal{N}(\map,\left(\log(\Pi_n)''(\map)\right)^{-1}),
\end{align}
where $\mathcal{N}(\mu,\Sigma)$ denotes the normal law with mean $\mu$ and covariance $\Sigma$. The computation of the mean and covariance of the above Gaussian is in the majority of cases easy numerically. It can be achieved using standard optimization schemes and does not require access to integrals with respect to the posterior, the normalizing constant of the posterior density or the true parameter. This is why Laplace approximation is a popular tool in approximate Bayesian inference.

While the above heuristic considerations may be turned into rigorous statements under certain conditions, a proper probabilistic grounding for the Laplace approximation is provided by the Bernstein--von Mises (BvM) theorem. As described in numerous classical references, including \citet[Section 1.4]{ghosh} or \citet[Section 10.2]{vandervaart}, the BvM theorem says that under mild assumptions on the likelihood and the prior, the posterior distribution converges to a Gaussian law in the following sense. Suppose that $\postparam$ is distributed according to the posterior, obtained after observing $n$ data points. Let $\param_0$ be the true parameter, $\mle$ be the MLE and $I(\cdot)$ be the Fisher information matrix. Let $\text{TV}$ denote the total variation distance and let the function $\mathcal{L}(\cdot)$ return the law of its argument. Then, if the model is well-specified and certain regularity conditions are satisfied,
\begin{align}\label{eq:bvm1}
	\text{TV}\left(\mathcal{L}\left(\sqrt{n}(\postparam-\mle)\right),\mathcal{N}(0,I(\param_0)^{-1})\right)\xrightarrow{\mathcal{P}}0,\qquad \text{as }n\to\infty,
\end{align}
where $n$ is the number of data and the convergence occurs in probability with respect to the law of the data.

While the model being well-specified is a crucial assumption in the above statement, \cite{misspecified_bvm} proved its modified version, under model misspecification. The main difference is in the limiting covariance matrix. Specifically, in this context, the authors assume the model is of the form $\param\mapsto p_{\param}$ and the observations are sampled from a density $p_0$ that is not necessarily of the form $p_{\param_0}$ for some $\theta_0$. They show that under certain regularity conditions,
\begin{align}\label{eq:bvm2}
	\text{TV}\left(\mathcal{L}\left(\sqrt{n}(\postparam-\mle)\right),\mathcal{N}(0,V(\param^*)^{-1})\right)\xrightarrow{\mathcal{P}}0,\qquad \text{as }n\to\infty,
\end{align}
where $\param^*$ minimizes the Kullback-Leibler divergence  $\param\mapsto \int \log(p_0(x)/p_{\param}(x))p_0(x)dx$ and $V(\param^*)$ is minus the second derivative of this map, evaluated at $\param^*$.

 Let us now denote by $L_n$ the generalized log-likelihood. A closer look at the classical proofs, including Le Cam's one (see e.g. \citealt[Section 1.4]{ghosh}), or more recent ones, including that of \citet[Appendix B]{miller}, reveals that, under standard regularity conditions:
\begin{align}\label{eq:bvm3}
	\text{TV}\left(\mathcal{L}\left(\sqrt{n}(\postparam-\mle)\right),\mathcal{N}\left(0,\left[-\frac{L_n''(\mle)}{n}\right]^{-1}\right)\right)\xrightarrow{\mathcal{P}}0,\qquad \text{as }n\to\infty,
\end{align}
no matter if the model is well-specified or not. It is known that under mild assumptions the MLE $\mle$ and the maximum a posteriori (MAP) $\map$ get arbitrarily close to each other as the number of data $n$ goes to infinity. Similarly, denoting by $\overline{L}_n$ the logarithm of the posterior density, $\overline{L}_n$ and $L_n$ get arbitrarily close  as $n$ goes to infinity. It can be shown, in a similar fashion to \cref{eq:bvm3}, that under standard regularity assumptions,
\begin{align}\label{eq:bvm4}
	\text{TV}\left(\mathcal{L}\left(\sqrt{n}(\postparam-\map)\right),\mathcal{N}\left(0,\left[-\frac{\logposterior''(\map)}{n}\right]^{-1}\right)\right)\xrightarrow{\mathcal{P}}0,\qquad \text{as }n\to\infty.
\end{align}
\Cref{eq:bvm4} gives a rigorous justification and meaning to the Laplace approximation given by \cref{laplace_approximation} and \cref{eq:bvm3} provides its alternative version. The approximating covariance in both \cref{eq:bvm4,eq:bvm3} is computable without access to the posterior normalizing constant or the true parameter.

The recent paper \citet{miller} proves almost sure versions of the statements given in \cref{eq:bvm1,eq:bvm2} for a large collection of commonly used models. Naturally, similar almost sure convergence statements can be obtained for the approximations appearing in \cref{eq:bvm3,eq:bvm4}.

\section{Proofs of Propositions \ref{prop_exponential} and \ref{prop_glm}}\label{sec:proofs_of_propositions}
\subsection{Proof of Proposition \ref{prop_exponential}}\label{proof_exponential}
    The proof is inspired by the proof of \citet[Theorem 12]{miller}. Note that $L_n(\theta)=-n\alpha(\theta)+\theta^TS_n,$ where $S_n=\sum_{i=1}^n s(Y_i)$. By standard exponential family theory \cite[Proposition 19]{miller_harrison}, $\alpha$ is $C^{\infty}$, strictly convex on $\Theta$, $\alpha'(\theta)=\mathbbm{E}_{\theta}s(Y)$ and $\alpha''(\theta)$ is symmetric positive definite for all $\theta\in\Theta$. Let $s_0:=\mathbbm{E}_{\theta_0}s(Y)$. By the strong law of large numbers, for all $\theta\in\Theta$, 
    \begin{align*}
    \frac{L_n(\theta)}{n}\xrightarrow{n\to\infty}-\alpha(\theta)+\theta^Ts_0=:f(\theta), \quad\text{almost surely}. 
    \end{align*}
    Note that due to the almost sure convergence of the sufficient statistics, we actually have a stronger statement. Indeed, it holds that, almost surely, for all $\theta\in\Theta$, $\frac{L_n(\theta)}{n}\xrightarrow{n\to\infty}f(\theta)$.
    
    Now, note that $f'(\theta_0)=0$. Note also that the MLE $\mle$ satisfying $L_n'(\mle)=0$ is unique (if it exists) because $L_n$ is strictly concave. By \citet[Theorems 12 and 5]{miller}, we have that the MLE $\mle$ almost surely exists and 
    \begin{align}\label{conv_of_mle}
    \mle\to\theta_0\,\,\text{ almost surely}\quad\text{and}\quad \frac{L_n(\mle)}{n}\xrightarrow{n\to\infty}f(\theta_0)\,\,\text{ almost surely.}
    \end{align}

    Let $\deltamle>0$ and $\bar{E}:=\{\theta:\|\theta-\theta_0\|\leq 2\delta\}$ be such that $\bar{E}\subseteq\Theta$. Then $\alpha'''$ is bounded on $\bar{E}$, since $\alpha'''$ is continuous and $\bar{E}$ is compact. Hence, $\frac{L_n'''}{n}$ is uniformly bounded on $\bar{E}$ because $L_n'''(\theta)=-n\alpha'''(\theta)$. Therefore, almost surely, \Cref{assump1} is satisfied for large enough $n$ and small enough $\delta$ with $M_2$ not depending on $n$. This is because $\mle\to\theta_0$ almost surely and so, almost surely, for large enough $n$, $\{\theta:\|\theta-\mle\|\leq \delta\}\subseteq \bar{E}$.  Now,  let $\xi_{\theta}$ be the point on the line connecting $\theta$ and $\theta_0$ that lies on $\{t:\|t-\theta_0\|=\deltamle/2\}$.
    The strict concavity of $L_n$ implies that, almost surely,
    \begin{align}
        &\limsup_{n\to\infty}\sup_{\theta:\|\theta-\mle\|>\deltamle}\frac{L_n(\theta) - L_n(\mle)}{n}\notag\\
        \leq& \limsup_{n\to\infty}\sup_{\theta:\|\theta-\theta_0\|>\deltamle/2}\frac{L_n(\theta)-L_n(\theta_0)}{n}\notag \\
        \leq&\limsup_{n\to\infty}\sup_{\theta:\|\theta-\theta_0\|>\deltamle/2}\frac{\|\theta_0-\theta\|}{\deltamle/2}\cdot\frac{L_n(\xi_{\theta})-L_n(\theta_0)}{n}\notag\\
        \leq&\limsup_{n\to\infty} \sup_{t:\|t-\theta_0\|=\deltamle/2}\frac{L_n(t)-L_n(\theta_0)}{n}\notag\\
        =&\limsup_{n\to\infty} \sup_{t:\|t-\theta_0\|=\deltamle/2}\left(-\alpha(t)+\frac{t^TS_n}{n}\right)-f(\theta_0)\notag\\
        \stackrel{(\ast)}=&\sup_{t:\|t-\theta_0\|=\deltamle/2}f(t)-f(\theta_0)\stackrel{(\ast\ast)}<0.\label{kappa_calculation1}
        \end{align}
    Equality $(\ast)$ follows from the fact that the sequence of functions $\left(t\mapsto \frac{t^TS_n}{n}\right)_n$ almost surely converges to $\left(t\mapsto t^Ts_0\right)$ uniformly on the set $\{t:\|t-\theta_0\|=\deltamle/2\}$.
    Now, we note that the function $t\mapsto f(t)-f(\theta_0)$ is continuous and negative on the set $\{t:\|t-\theta_0\|=\deltamle/2\}$, by \citet[Lemma 27 (1)]{miller}. Therefore, its supremum on the compact set $\{t:\|t-\theta_0\|=\deltamle/2\}$ is negative, which justifies the last inequality $(\ast\ast)$. Therefore almost surely \Cref{assump_kappa} is satisfied for large enough $n$, for a $\kappamle$ not depending on $n$.

    Note that the above considerations are valid for any fixed $\deltamle>0$, such that $\bar{E}\subseteq\Theta$. Since $\alpha''$ is continuous on $\Theta$ and $\alpha''(\theta)$ is symmetric, positive definite for all $\theta$ and $\mle\to\theta_0$, almost surely, \Cref{assump_fisher} is also satisfied  for $n$ large enough and for $\deltamle>0$ small enough, almost surely. Moreover, note that $\Tr\left(\fisherinformation(\mle)^{-1}\right)=\Tr\left(\alpha''(\mle)^{-1}\right)\leq \frac{d}{\minevaluemle}$. Since $\alpha''$ is continuous, $\alpha''(\theta)$ is positive definite for all $\theta\in\Theta$ and $\mle\to\theta_0$ almost surely, we have that almost surely for large $n$, $\frac{d}{\minevaluemle}$ is uniformly bounded from above. Therefore, \Cref{assump_size_of_delta} is satisfied for large enough $n$, almost surely.

    Furthermore, \Cref{assump_prior1} is satisfied immediately for large $n$ (with $\hat{M}_1$ independent of $n$) if the prior density $\pi$ is continuous and positive in a neighborhood around $\theta_0$ and if $\mle\to\theta_0$, which holds almost surely (\cref{conv_of_mle}). Similarly \Cref{assump_prior2} is immediately satisfied for large $n$ (with $M_1$ and $\widetilde{M}_1$ independent of $n$) if, additionally, $\pi$ is continuously differentiable in a neighborhood of $\theta_0$.

    Now, assume that $\pi$ is thrice continuously differentiable on $\Theta$. Then $\logposterior$ is upper-bounded and has at least one global maximum $\theta_n^*\in\Theta$. Using \cref{kappa_calculation1}, we note that $\|\theta_n^*-\mle\|\xrightarrow{n\to\infty}0$ almost surely (as $\pi$ is uniformly upper bounded on $\Theta$). We also note that, almost surely, for any fixed neighborhood of $\theta_0$ and for sufficiently large $n$, $-\frac{\logposterior''(\theta)}{n}=\alpha''(\theta)-\frac{(\log\pi)''(\theta)}{n}$ is strictly positive definite for $\theta$ in that neighborhood. This means that, almost surely, for sufficiently large $n$, the MAP $\map$ is unique. Moreover, $\|\map-\mle\|\xrightarrow{n\to\infty}0$ almost surely,which implies that $\map\to\theta_0$ almost surely, by \cref{conv_of_mle}.
    Moreover, note that $$\frac{\logposterior'''(\theta)}{n}=-\alpha'''(\theta)+\frac{\left(\log \pi\right)'''(\theta)}{n}.$$
    Note that $(\log\pi)'''$ is continuous in a neighbourhood of $\theta_0$. Then, almost surely, for large enough $n$ and small enough $\deltabar$, $\frac{\logposterior'''(\theta)}{n}$ is uniformly bounded inside $\{\theta: \|\theta-\map\|\leq \deltabar\}$ (which follows from the fact that \Cref{assump1} is satisfied). Thus, \Cref{assump2} is satisfied for large enough $n$ and small enough $\deltabar$, almost surely.  Using the same assumption that $(\log\pi)'''$ is continuous in a neighbourhood of $\theta_0$, we have that $\alpha''-\frac{(\log \pi)''}{n}$ is continuous in a neighbourhood of $\theta_0$ and for large enough $n$, $\alpha''(\theta)-\frac{(\log \pi)''(\theta)}{n}$ is symmetric, positive definite for all $\theta$ in a (potentially smaller) neighborhood of $\theta_0$. Moreover, as we showed above, $\map\to\theta_0$, almost surely. Therefore, \Cref{assump7} is satisfied for $n$ large enough and for $\deltabar>0$ small enough, almost surely.  
    
    Now, keeping the same assumptions, note that $\max\left\{\|\hat{\theta}_n-\bar{\theta}_n\|,\sqrt{\frac{\Tr\left(\postfisherinformation(\map)^{-1}\right)}{n}}\right\}<\deltabar$ for large enough $n$ and small enough $\deltabar$, almost surely. This is because if $\minevaluemap>0$ then $\Tr\left(\postfisherinformation(\map)^{-1}\right)\leq \frac{d}{\minevaluemap}$, which is uniformly bounded from above for large enough $n$. Note also that $\sqrt{\frac{\Tr\left[\left(\jhat +\frac{\deltamle M_2}{3}I_{d\times d}\right)^{-1}\right]}{n}}<\deltamle$ follows from \Cref{assump_size_of_delta}, which holds for large $n$ and small $\delta>0$, almost surely, as we showed above. Therefore,  \Cref{assump_size_of_delta_bar} is satisfied for large enough $n$ and small enough $\deltamle$ and $\deltabar$, almost surely. Moreover, as $\map\to\theta_0$ almost surely, \Cref{assump_kappa1} is satisfied for large enough $n$ and small enough $\deltabar$, almost surely, by an argument analogous to \cref{kappa_calculation1}.$\blacksquare$
    
\subsection{Proof of Proposition \ref{prop_glm}}\label{proof_glm}
    The proof is inspired by the proof of \citet[Theorem 13]{miller}.
    Note that, for all $\theta\in\Theta,$ $L_n(\theta)=-\sum_{i=1}^n\alpha(\theta^TX_i)+\theta^TS_n$, where $S_n=\sum_{i=1}^nX_is(Y_i)$. Thus, $L_n$ is $C^{\infty}$ on $\Theta$ by the chain rule, since $\alpha$ is $C^{\infty}$ on $\mathcal{E}$ by \citet[Proposition 19]{miller_harrison}. Also, $L_n$ is strictly concave since $\alpha$ is strictly convex \citep[Proposition 19]{miller_harrison}. Now, note that, for all $\theta\in\Theta$, $\frac{L_n(\theta)}{n}\xrightarrow{n\to\infty}f(\theta)$ almost surely (by our assumed condition 2). This implies that, almost surely, $\frac{L_n(\theta)}{n}\xrightarrow{n\to\infty}f(\theta)$, for all $\theta\in\Theta$. In order to show this implication, let us fix a countable dense subset $C$ of $\Theta$. Then, almost surely, $\frac{L_n(\theta)}{n}\xrightarrow{n\to\infty}f(\theta)$ for all $\theta\in C$. Since $\frac{L_n}{n}$ is concave, it follows from \citet[Theorem 10.8]{rockafellar} that, almost surely, the limit $\tilde{f}(\theta):=\lim_n\frac{L_n(\theta)}{n}$ exists and is finite for all $\theta\in\Theta$ and $\tilde{f}$ is concave. As $f$ is also concave, then $\tilde{f}$ and $f$ are continuous functions \cite[Theorem 10.1]{rockafellar} that agree on a dense subset of $\Theta$ so they are equal on $\Theta$.

    Now, note that, by \citet[Theorem 5 and Theorem 13]{miller}, almost surely, there exists MLE $\mle$, such that $L_n'(\mle)=0$ and
    \begin{align}\label{mle_conv}
    \mle\xrightarrow{n\to\infty}\theta_0,\quad\text{almost surely}. 
    \end{align}
    As $L_n$ is strictly concave, this MLE is almost surely unique.

    We will show that, with probability $1$, $\frac{L_n'''}{n}$ is uniformly bounded on $\bar{E}=\{\theta:\|\theta-\theta_0\|\leq\epsilon\}$, where $\epsilon>0$ satisfies condition 4 of Proposition \ref{prop_glm}. Fix $j,k,l\in\{1,\dots,d\}$ and define
   $T(\theta,x)=\alpha'''(\theta^Tx)x_jx_kx_l$
  for $\theta\in\Theta$, $x\in\mathcal{X}$. For all $x\in\mathcal{X}$, $\theta\mapsto T(\theta,x)$ is continuous and for all $\theta\in\Theta$, $x\mapsto T(\theta,x)$ is measurable. Since $L_n'''(\theta)_{j,k,l}=-\sum_{i=1}^n T(\theta,X_i)$, condition 4 above implies that with probability $1$, $\frac{L_n'''(\theta)_{j,k,l}}{n}$ is uniformly bounded on $\bar{E}$, by the uniform law of large numbers \cite[Theorem 1.3.3]{ghosh}. Letting $C_{j,k,l}(X_1,X_2,\dots)$ be such a uniform bound for each $j,k,l$, we have that with probability $1$, for all $n\in\mathbbm{N}$, $\theta\in \bar{E}$, $\left\|\frac{L_n'''(\theta)}{n}\right\|^2=\sum_{j,k,l} \frac{L_n'''(\theta)^2_{j,k,l}}{n^2}\leq \sum_{j,k,l}C_{j,k,l}(X_1,X_2,\dots)^2<\infty$. Hence $\frac{L_n'''}{n}$ is uniformly bounded on $\bar{E}$. Therefore, almost surely, for large enough $n$ and small enough $\deltamle$, \Cref{assump1} is satisfied with a constant $M_2$ not depending on $n$. This is because $\mle\to\theta_0$ almost surely and so, almost surely, for large enough $n$ and $\delta=\epsilon/2$, $\{\theta:\|\theta-\mle\|\leq\delta\}\subseteq \bar{E}$.

    Now, using \citet[Theorem 7]{miller},
    \begin{align}\label{f''_conv}
    f''(\theta_0)\stackrel{a.s.}=\lim_{n\to\infty}\frac{L_n''(\theta_0)}{n}=\lim_{n\to\infty}\left(-\frac{1}{n}\sum_{i=1}^n\alpha''(\theta_0^TX_i)X_iX_i^T\right).
    \end{align}
    This limit exists and is finite almost surely. Therefore, by the strong law of large numbers,
    \begin{align*}
        f''(\theta_0)\stackrel{a.s.}=-\mathbbm{E}\left(\alpha''(\theta_0^TX_i)X_iX_i^T\right).
    \end{align*}
    Therefore, $-f''(\theta_0)$ is positive definite almost surely because $\alpha''(\eta)$ is positive definite for any $\eta$ and so for any nonzero $a\in\mathbbm{R}^d$, we have
    \begin{align*}
        \mathbbm{E}\left[\alpha''(\theta_0^TX_i)a^TX_iX_i^Ta\right]>0
    \end{align*}
    by our assumption that $a^TX_i$ is distinct from zero with a positive probability. Now, by \citet[Theorem 7]{miller}, $\frac{L_n''}{n}$ is $L$-equi-Lipschitz for some $L<\infty$. Therefore, we have that:
    \begin{align}\label{equilip}
        \left\|\frac{L_n''(\mle)}{n}-f''(\theta_0)\right\|\leq L\left\|\mle-\theta_0\right\|+\left\|\frac{L_n''(\theta_0)}{n}-f''(\theta_0)\right\|.
    \end{align}
    It follows that $\frac{L_n''(\mle)}{n}\xrightarrow{n\to\infty}f''(\theta_0)$ almost surely by \cref{mle_conv,f''_conv}.
    This means that, almost surely, for large enough $n$, the minimum eigenvalue of $-\frac{L_n''(\mle)}{n}$ stays positive and lower bounded by a positive number, not depending on $n$, as $f''(\theta_0)$ is positive definite. Therefore, almost surely, \Cref{assump_fisher} is satisfied for large enough $n$ and small enough $\deltamle$. Moreover, for the same reason, \Cref{assump_size_of_delta} is satisfied, for large enough $n$ and small enough $\deltamle>0$, almost surely. Indeed, note that $\Tr\left(\fisherinformation(\mle)^{-1}\right)=\Tr\left(\alpha''(\mle)^{-1}\right)\leq \frac{d}{\minevaluemle}$. Since $\frac{L_n''(\mle)}{n}\xrightarrow{n\to\infty}f''(\theta_0)$ almost surely, we have that, almost surely, for large $n$, $\frac{d}{\minevaluemle}$ is uniformly bounded from above.

    Now,  let $\xi_{\theta}$ be the point on the line connecting $\theta$ and $\theta_0$ that lies on $\{t:\|t-\theta_0\|=\deltamle/2\}$.
    The strict concavity of $L_n$ implies that, almost surely,
    \begin{align}
        &\limsup_{n\to\infty}\sup_{\theta:\|\theta-\mle\|>\deltamle}\frac{L_n(\theta) - L_n(\mle)}{n}\notag\\
        \leq& \limsup_{n\to\infty}\sup_{\theta:\|\theta-\theta_0\|>\deltamle/2}\frac{L_n(\theta)-L_n(\theta_0)}{n}\notag \\
        \leq&\limsup_{n\to\infty}\sup_{\theta:\|\theta-\theta_0\|>\deltamle/2}\frac{\|\theta_0-\theta\|}{\deltamle/2}\cdot\frac{L_n(\xi_{\theta})-L_n(\theta_0)}{n}\notag\\
        \leq&\limsup_{n\to\infty} \sup_{t:\|t-\theta_0\|=\deltamle/2}\frac{L_n(t)-L_n(\theta_0)}{n}\notag\\
        =&\limsup_{n\to\infty} \sup_{t:\|t-\theta_0\|=\deltamle/2}\frac{L_n(t)}{n}-f(\theta_0)\notag\\
        \stackrel{(\ast)}=&\sup_{t:\|t-\theta_0\|=\deltamle/2}f(t)-f(\theta_0)<0.\label{kappa_calculation}
        \end{align}
    Equality $(\ast)$ follows from the fact that, by \citet[Theorem 7]{miller}, the sequence of functions $\left(\frac{L_n}{n}\right)_n$ almost surely converges to $f$ uniformly on the set $\{t:\|t-\theta_0\|=\deltamle/2\}$.
    Now, we note that the function $t\mapsto f(t)-f(\theta_0)$ is continuous and negative on the set $\{t:\|t-\theta_0\|=\deltamle/2\}$, by \citet[Lemma 27 (1)]{miller}. Therefore, its supremum on the compact set $\{t:\|t-\theta_0\|=\deltamle/2\}$ is negative, which justifies the last inequality. Therefore, almost surely, \Cref{assump_kappa} is satisfied, for large enough $n$, small enough $\delta$ and for some $\kappamle>0$ not depending on $n$.

     Furthermore, \Cref{assump_prior1} is satisfied immediately, almost surely for large $n$ and small enough $\delta$, if the prior density $\pi$ is continuous and positive in a neighborhood around $\theta_0$ since $\mle\to\theta_0$ almost surely (\cref{mle_conv}). Similarly \Cref{assump_prior2} is immediately satisfied almost surely if, additionally, $\pi$ is continuously differentiable in a neighborhood of $\theta_0$.
     
 Now, assume that $\pi$ is thrice continuously differentiable on $\Theta$. Then $\logposterior$ is upper-bounded and has at least one global maximum $\theta_n^*\in\Theta$. Using \cref{kappa_calculation}, we note that $\|\theta_n^*-\mle\|\xrightarrow{n\to\infty}0$ almost surely (as $\pi$ is uniformly upper bounded on $\Theta$). We also note that, almost surely, for any fixed neighborhood of $\theta_0$ and for sufficiently large $n$, $-\frac{\logposterior''(\theta)}{n}$ is strictly positive definite for $\theta$ in that neighborhood. This means that, almost surely, for sufficiently large $n$, the MAP $\map$ is unique. Moreover, $\|\map-\mle\|\xrightarrow{n\to\infty}0$ almost surely,which implies that $\map\to\theta_0$ almost surely. Moreover, note that $$\frac{\logposterior'''(\theta)}{n}=\frac{L_n'''(\theta)}{n}+\frac{\left(\log \pi\right)'''(\theta)}{n}.$$
    Hence, if $(\log\pi)'''$ is continuous in a neighbourhood of $\theta_0$ then, almost surely, for large enough $n$ and small enough $\deltabar$, $\frac{\logposterior'''(\theta)}{n}$ is uniformly bounded inside $\{\theta: \|\theta-\map\|\leq \deltabar\}$ (which follows from the fact that \Cref{assump1} holds almost surely and that $\|\mle-\map\|\to 0$ almost surely). Thus, almost surely, \Cref{assump2} is satisfied for large enough $n$ and small anough $\deltabar$ with $\overline{M}_2$ independent of $n$.  

    Now, as in \cref{equilip}, we have that
    \begin{align*}
        \left\|\frac{L_n''(\map)}{n}-f''(\theta_0)\right\|\leq L\left\|\map-\theta_0\right\|+\left\|\frac{L_n''(\theta_0)}{n}-f''(\theta_0)\right\|
    \end{align*}
    and so $\frac{L_n''(\map)}{n}\xrightarrow{n\to\infty} f''(\theta_0)$ almost surely. As a result, almost surely, for large enough $n$, the minimum eigenvalue of $-\frac{L_n''(\map)}{n}$ stays lower bounded by a positive number not depending on $n$. It follows that, if $(\log\pi)'''$ is continuous in a neighbourhood of $\theta_0$, then, almost surely, for large $n$, the minimum eigenvalue of $-\frac{L_n''(\map)}{n}-\frac{(\log \pi)''(\map)}{n}$ stays lower bounded by a positive number not depending on $n$. Therefore, almost surely, \Cref{assump7} is also satisfied, for $n$ large enough and for $\deltabar>0$ small enough.  
    
    Still assuming that $(\log\pi)'''$ is continuous in a neighborhood of $\theta_0$, note that, almost surely, $\max\left\{\|\hat{\theta}_n-\bar{\theta}_n\|,\sqrt{\frac{\Tr\left(\postfisherinformation(\map)^{-1}\right)}{n}}\right\}<\deltabar$ for large enough $n$ and small enough $\deltabar$. This is because $\Tr\left(\postfisherinformation(\map)^{-1}\right)\leq \frac{d}{\minevaluemap}$ and $\|\hat{\theta}_n-\bar{\theta}_n\|\to 0$ almost surely. Moreover, note that $\sqrt{\frac{\Tr\left[\left(\jhat+\frac{\deltamle M_2}{3}I_{d\times d}\right)^{-1}\right]}{n}}< \deltamle$ if \Cref{assump_size_of_delta} is satisfied. Therefore, almost surely, \Cref{assump_size_of_delta_bar} is satisfied for large enough $n$ and small enough $\deltamle$ and $\deltabar$. Furthermore, as $\map\to\theta_0$ almost surely then also almost surely \Cref{assump_kappa1} is satisfied for large enough $n$ and small enough $\deltabar$, by an argument analogous to \cref{kappa_calculation}.$\blacksquare$

	\section{Introductory arguments for the proofs of Theorems \ref{main_map} -- \ref{w2_map}, \ref{theorem_main} -- \ref{theorem_2wasserstein} and \ref{theorem_univariate}}\label{introductory_arguments}
	\subsection{Introduction}
	In order to prove the main results of this paper, we let $g:\mathbbm{R}^d\to\mathbbm{R}$ and consider:
	\begin{align}
		D_g^{MLE}:=&\left|\mathbbm{E}\left[g\left(\sqrt{n}(\postparam-\mle)\right)\right]-\mathbbm{E}_{Z_n\sim\mathcal{N}(0,\fisherinformation(\mle)^{-1})}\left[g(Z_n)\right]\right|;\label{dgmle}\\
		D_g^{MAP}:=&\left|\mathbbm{E}\left[g\left(\sqrt{n}(\postparam-\map)\right)\right]-\mathbbm{E}_{\bar{Z}_n\sim\mathcal{N}(0,\postfisherinformation(\map)^{-1})}\left[g(\bar{Z}_n)\right]\right|.\label{dgmap}
	\end{align}
	The following lemma will be useful in the sequel:
	\begin{lemma}\label{lemma1}
		Let $Z\sim\mathcal{N}(0,\Sigma)$. Then, for any $t>0$,
		\begin{align}
			&\mathbbm{P}\left[\left\|Z\right\|-\sqrt{\Tr(\Sigma)}\geq \sqrt{2\left\|\Sigma\right\|_{op}t}\right]\leq e^{-t};\label{concentration1}\\
			&\mathbbm{P}\left[\left\|Z\right\|^2-\Tr(\Sigma)\geq 2\sqrt{\Tr\left(\Sigma^2\right)t}+2\left\|\Sigma\right\|_{op}t\right]\leq e^{-t}.\label{concentration2}
		\end{align}
	\end{lemma}
	\begin{proof}
		The proof follows an argument similar to the one of \citet[page 135]{vershynin}. Specifically, \cref{concentration2} comes from \citet[Proposition 1]{hsu}. In order to prove \cref{concentration1}, we note that
		\begin{align*}
			\mathbbm{P}\left[\left\|Z\right\|-\sqrt{\Tr(\Sigma)}\geq \sqrt{2\left\|\Sigma\right\|_{op}t}\right]=&\mathbbm{P}\left[\left\|Z\right\|^2-\Tr(\Sigma)\geq 2\left\|\Sigma\right\|_{op} t + 2\sqrt{2\Tr\left(\Sigma\right)\left\|\Sigma\right\|_{op}t}\right]\\
			\leq & \mathbbm{P}\left[\left\|Z\right\|^2-\Tr(\Sigma)\geq 2\sqrt{\Tr\left(\Sigma^2\right)t}+2\left\|\Sigma\right\|_{op}t\right]\leq e^{-t}.
		\end{align*}
	\end{proof}
	\subsection{Initial decomposition of the distances $D_g^{MLE}$ and $D_g^{MAP}$}
	Now, let
	\begin{align*}
	h_g^{MLE}(u):=g(u)-\frac{n^{-d/2}}{C^{MLE}_n} \int_{\|t\|\leq \deltamle\sqrt{n}}g(t)\posterior(n^{-1/2} t+\mle) dt,\\
	h_g^{MAP}(u):=g(u)-\frac{n^{-d/2}}{C^{MAP}_n} \int_{\|t\|\leq \deltabar\sqrt{n}}g(t)\posterior(n^{-1/2} t+\map) dt,
	\end{align*}
	for
	\begin{align}
	 C_n^{MLE}:=n^{-d/2}\int_{\|u\|\leq\deltamle\sqrt{n}}\posterior(n^{-1/2} u+\mle)du\label{c_n_mle}\\
	 C_n^{MAP}:=n^{-d/2}\int_{\|u\|\leq\deltabar\sqrt{n}}\posterior(n^{-1/2} u+\map)du.\notag
	\end{align}
	Note that, for
	\begin{align}\label{fn_mle}
	F_n^{MLE}:=\int_{\|u\|\leq\deltamle\sqrt{n}}\frac{\sqrt{\det \jhat}e^{-u^T\fisherinformation(\mle)u/2}}{(2\pi)^{d/2}}du,
	\end{align}
	we have
	\begin{align}
		&D_g^{MLE}\notag\\
		=& \left|\int_{\mathbbm{R}^d}h_g^{MLE}(u)\frac{\sqrt{\det \jhat}e^{-u^T\fisherinformation(\hat{\theta}_n)u/2}}{(2\pi)^{d/2}}du\right.\notag\\
		&\left.\hspace{5cm}-n^{-d/2} \int_{\|u\|>\deltamle\sqrt{n}} h_g^{MLE}(u) \posterior(n^{-1/2} u+\mle)du\right|\notag\\
		\leq &\frac{1}{F_n^{MLE}}\left|\int_{\|u\|\leq\deltamle\sqrt{n}}h_g^{MLE}(u)\frac{\sqrt{\det \jhat}e^{-u^T\fisherinformation(\mle)u/2}}{(2\pi)^{d/2}}du\right|\notag\\
		&+\left|\int_{\|u\|>\deltamle\sqrt{n}}h_g^{MLE}(u)\left[\frac{\sqrt{\det \jhat}e^{-u^T\fisherinformation(\mle)u/2}}{(2\pi)^{d/2}}-n^{-d/2}\,\posterior(n^{-1/2} u+\mle)\right] du\right|\notag\\
		=&\left|\int_{\|u\|\leq\deltamle\sqrt{n}}g(u)\frac{\sqrt{\det \jhat}e^{-u^T\fisherinformation(\mle)u/2}}{F_n^{MLE}(2\pi)^{d/2}}du\right.\notag\\
		&\left.\hspace{6cm}-\frac{n^{-d/2}}{C_n^{MLE}} \int_{\|u\|\leq\deltamle\sqrt{n}}g(u)\posterior(n^{-1/2} u+\mle) du\right|\notag\\
		&+\left|\int_{\|u\|>\deltamle\sqrt{n}}h^{MLE}_g(u)\left[\frac{\sqrt{\det \jhat}e^{-u^T\fisherinformation(\mle)u/2}}{(2\pi)^{d/2}}-n^{-d/2}\,\posterior(n^{-1/2} u+\mle)\right] du\right|\notag\\
		=:&I_1^{MLE}+I_2^{MLE}.\label{i1i2mle}
	\end{align}

In a similar manner, for
\begin{align}
	F_n^{MAP}:=\int_{\|u\|\leq\deltabar\sqrt{n}}\frac{\sqrt{\det \jbar}e^{-u^T\postfisherinformation(\map)u/2}}{(2\pi)^{d/2}}du,\label{f_n_map}
\end{align}
we have
\begin{align}
	&D_g^{MAP}\notag\\
	\leq&\left|\int_{\|u\|\leq\deltabar\sqrt{n}}g(u)\frac{\sqrt{\det \jbar}e^{-u^T\postfisherinformation(\map)u/2}}{F_n^{MAP}(2\pi)^{d/2}}du\right.\notag\\
	&\left.\hspace{6cm}-\frac{n^{-d/2}}{C_n^{MAP}} \int_{\|u\|\leq\deltabar\sqrt{n}}g(u)\posterior(n^{-1/2} u+\map) du\right|\notag\\
	&+\left|\int_{\|u\|>\deltabar\sqrt{n}}h^{MAP}_g(u)\left[\frac{\sqrt{\det \jbar}e^{-u^T\postfisherinformation(\map)u/2}}{(2\pi)^{d/2}}-n^{-d/2}\,\posterior(n^{-1/2} u+\map)\right] du\right|\notag\\
	=:&I_1^{MAP}+I_2^{MAP}. \label{i1i2map}
\end{align}

\subsection{Strategies for controlling terms $I_1^{MLE}$ and $I_1^{MAP}$}
In order to control $I_1^{MAP}$ in \Cref{main_map,thm:wasserstein_map,w2_map} and $I_1^{MLE}$ in \Cref{theorem_main,theorem_1wasserstein,theorem_2wasserstein}, we apply the log-Sobolev inequality and the associated transportation inequalities. Before we introduce those concepts, we define the following Kullback-Leibler divergence (or relative entropy, see e.g. \citealt[Section 1.6.1]{bishop}) between two measures $\nu$ and $\mu$ such that $\nu\ll\mu$:
\begin{align}\label{def:kl_divergence}
\text{KL}(\nu\|\mu)=\int\log\left(\frac{d\nu}{d\mu}(x)\right)\nu(dx).
\end{align}
We also define the following Fisher divergence between two measures $\nu$ and $\mu$ such that $\nu\ll\mu$
\begin{align}\label{def:fisher_divergence}
	\text{Fisher}(\nu\|\mu)=\int\left\|\left(\log\left(\frac{d\nu}{d\mu}\right)\right)'(x)\right\|^2\nu(dx).
\end{align}
The following definition, which can be found in  \citet[Chapter 5]{bakry_gentil_ledoux} and \citet[Section 2.2]{vempala_wibisono}, will play a central role in our proofs:
\begin{definition}[Log-Sobolev inequality (LSI)]\label{def:log-sobolev}
Let $\mu$ be a probability measure on $\Omega\subset\mathbbm{R}^d$ and let $\alpha>0$. We say that $\mu$ satisfies the log-Sobolev inequality with constant $\alpha$, or LSI($\alpha$) for short, if 
\begin{align}\label{general_lsi}
	\int f^2\log f^2 d\mu - \int f^2d\mu \log\left(\int f^2 d\mu\right)\leq \frac{2}{\alpha}\int\left\| f'\right\|^2d\mu,\quad\text{for all }f:\Omega\to\mathbbm{R}^+.
\end{align}
Equivalently, we say that $\mu$ satisfies LSI($\alpha$) if for any probability measure $\nu$ on $\Omega$, such that $\mu\ll\nu$ and $\nu\ll\mu$,
\begin{align}\label{measure_lsi}
	\text{KL}(\nu\|\mu)\leq \frac{1}{2\alpha}\text{Fisher}(\nu\|\mu).
\end{align}
\end{definition}
\begin{remark}
	The equivalence between the two definitions of the log-Sobolev inequality LSI($\alpha$), given by \cref{general_lsi,measure_lsi}, follows by the following argument. If $\nu$ is a probability measure on $\Omega$ such that $\mu$ and $\nu$ are equivalent then we obtain \cref{measure_lsi} by setting $f^2=\frac{d\nu}{d\mu}$ in \cref{general_lsi}. On the other hand, for a function $f:\Omega\to\mathbbm{R}^+$, let $\nu$ be a probability measure on $\Omega$ with density $\nu(dx)\propto f^2(x)\mu(dx)$. Plugging this $\nu$ into \cref{measure_lsi} yields \cref{general_lsi}.
\end{remark}

One of the fundamental results we will use is the \textit{Bakry-\'Emery} criterion. Its first general version was proved in \citet{bakry-emery}. Below we state its version for measures on $\mathbbm{R}^d$ truncated to convex sets (see \citealt[Theorem 2.1]{kolesnikov} and \citealt[Theorem A1]{Schlichting:2019}):
\begin{proposition}[Bakry-\'Emery criterion]\label{prop_lsi}
Let $\Omega\subset\mathbbm{R}^d$ be convex, let $H:\Omega\to\mathbbm{R}$ and consider a probability measure $\mu(dx)\propto e^{-H(x)}\mathbbm{1}_\Omega(x)dx$. Assume that $H''(x)\succeq \alpha I_{d\times d}$, for some $\alpha>0$. Then $\mu$ satisfies LSI($\alpha$) (see Definition \cref{def:log-sobolev}).
\end{proposition}

The following Holley-Stroock perturbation principle (see \citealt[page 1184]{holley-stroock}, \citealt[Theorem A2]{Schlichting:2019}) will be crucial in our proofs concerning the MLE-centric approach: 
\begin{proposition}[Holley-Stroock perturbation principle]
	\label{holley-stroock}
Let $\Omega\subset\mathbb{R}^d$ and $H:\Omega\to\mathbbm{R}$. Let $\psi:\Omega\to\mathbbm{R}^d$ be a bounded function. Let $\mu$ and $\tilde{\mu}$ be probability measures with densities of the form $\mu(dx)\propto e^{-H(x)}\mathbbm{1}_{\Omega}(x)dx$ and $\tilde{\mu}(dx)\propto e^{-H(x)-\psi(x)}\mathbbm{1}_{\Omega}(x)dx$. Suppose that $\mu$ satisfies the LSI($\alpha$) (see Definition \cref{def:log-sobolev}). Then $\tilde{\mu}$ satisfies  LSI($e^{-\text{osc}\psi}\alpha$), where $\text{osc}\psi:=\sup_{\Omega}\psi-\inf_{\Omega}\psi$.
\end{proposition}

In order to control $I_1^{MAP}$ and $I_1^{MLE}$ in our proofs, we control the Fisher divergence between the rescaled posterior and the Gaussian inside the ball around the MAP or the MLE, establish the log-Sobolev inequality inside that ball for the rescaled posterior and thus control the KL divergence. In order to obtain our bounds on the total variation distance in \Cref{main_map,theorem_main}, we will use the following result:
\begin{proposition}[Pinsker's inequality, e.g. {\citealt[Theorem 2.16]{massart}}]\label{pinsker}
For any two probability measures $\mu$ and $\nu$ such that $\nu\ll\mu$, we have that
\begin{align*}
	\text{TV}(\mu,\nu)\leq \sqrt{\frac{1}{2}\text{KL}\left(\nu\|\mu\right)}.
\end{align*}
\end{proposition}

In order to control the 1-Wasserstein distance in \Cref{thm:wasserstein_map,theorem_1wasserstein} and the integral probability metric appearing in \Cref{w2_map,theorem_2wasserstein}, we first control $I_1^{MAP}$ or $I_1^{MLE}$ by controlling the 2-Wasserstein distance inside the appropriate neighborhood of the MAP or the MLE. The 2-Wasserstein distance  for two measures $\mu$ and $\nu$ on the same measurable space is defined in the following way:
\begin{align}\label{def:w2}
	W_2(\nu,\mu)=\inf_{\Gamma}\sqrt{\mathbbm{E}_{\Gamma}\left[\|X-Y\|^2\right]}
\end{align}
where the infimum is over all distributions $\Gamma$ of $(X,Y)$ with the correct marginals $X\sim\nu$, $Y\sim\mu$. It is an easy consequence of Jensen's inequality that for any probability measures $\mu$, $\nu$,
\begin{align}\label{wasserstein_ineq}
	W_1(\mu,\nu)\leq W_2(\mu,\nu).
\end{align}
We also have the following Talagrand inequality:
\begin{proposition}\label{prop_talagrand}
Let $\mu$ be a probability measure on a closed ball $\{u:\|u\|\leq \eta\}$ for some $\eta>0$ such that $\mu$ is absolutely continuous with respect to the Lebesgue measure. Suppose that $\mu$ satisfies LSI($\alpha$) on $\{u:\|u\|\leq \eta\}$. Then, for any probability measure $\nu\ll\mu$ on $\{u:\|u\|\leq \eta\}$, we have that
\begin{align*}
	W_2(\mu,\nu)^2\leq \frac{2}{\alpha} \text{KL}(\nu\|\mu).
\end{align*}
\end{proposition}
\begin{proof}
	This result follows from \citet[Theorem 4.1]{gozlan} by an argument analogous to the one that let the authors prove the result in \citet[Corollary 4.2]{gozlan}. Indeed, for a natural number $n\geq 1$ and  $x=(x_1,\dots,x_n)\in\left(\mathbbm{R}^d\right)^n$, consider $L_n^x=n^{-1}\sum_{i=1}^n\delta_{x_i}$, where $\delta_{x_i}$'s are Dirac deltas. In the proof of \citet[Corollary 4.2]{gozlan} it is shown that the function $x\mapsto W_2(L_n^x,\mu)$ is $\frac{1}{\sqrt{n}}$-Lipschitz for the Euclidean distance. When restricted to the set $\{u:\|u\|\leq \eta\}^n\subset\left(\mathbbm{R}^d\right)^n$, this function is still $\frac{1}{\sqrt{n}}$-Lipschitz. As in the proof of \citet[Corollary 4.2]{gozlan}, we note that, by Rademacher's theorem the function $F_n(x)=W_2(L_n^x,\mu)$ is differentiable Lebesgue-almost everywhere on $\left(\mathbbm{R}^{d}\right)^n$. When restricted to $\{u:\|u\|\leq \eta\}^n\subset\left(\mathbbm{R}^d\right)^n$ this function stays differentiable Lebesgue-almost everywhere. Therefore, the condition 
	\begin{align*}
		\sum_{i=1}^n\left\|\nabla_iF_n\right\|^2(x)\leq\frac{1}{n}\quad\text{for }\mu^n\text{-almost every }x\in\{u:\|u\|\leq \eta\}^n\subset\left(\mathbbm{R}^d\right)^n
	\end{align*}
is fulfilled, as required by \citet[Theorem 4.1]{gozlan}.
\end{proof}
Proposition \ref{prop_talagrand} together with \cref{wasserstein_ineq} and the log-Sobolev inequality will let us control $I_1^{MAP}$ and $I_1^{MLE}$ in the proofs of \Cref{thm:wasserstein_map,w2_map,theorem_1wasserstein,theorem_2wasserstein}.
	\subsection{Controlling term $I_2^{MLE}$ }

Note that
	\begin{align*}
		I_2^{MLE}\leq &\left|\int_{\|u\|>\deltamle\sqrt{n}}g(u)\left(n^{-d/2} \posterior(n^{-1/2} u+\mle)-\frac{\sqrt{\det\jhat}e^{-u^T\fisherinformation(\mle)u/2}}{(2\pi)^{d/2}}\right)du\right|\\
		&+ \frac{n^{-d/2}}{C_n^{MLE}}\int_{\|t\|\leq \deltamle\sqrt{n}}|g(t)|\posterior(n^{-1/2}t+\hat{\theta}_n)dt\\
		&\hspace{1cm}\cdot\left|\int_{\|u\|>\deltamle\sqrt{n}}\left(n^{-d/2} \posterior(n^{-1/2} u+\mle)-\frac{\sqrt{\det\jhat}e^{-u^T\fisherinformation(\hat{\theta}_n)u/2}}{(2\pi)^{d/2}}\right)du\right|\\
		=:&I_{2,1}^{MLE}+I_{2,2}^{MLE}.
	\end{align*}
	Now, note that $L_n'(\mle)=0$. Therefore, for $\|t\|\leq\deltamle\sqrt{n}$, \Cref{assump1} implies that
	\begin{align}\label{taylor1}
	\left|L_n(n^{-1/2}t+\hat{\theta}_n)-L_n(\hat{\theta}_n)+\frac{1}{2}t^T\jhat t\right|\leq \frac{1}{6}n^{-1/2}M_2\|t\|^3\leq \frac{\deltamle M_2}{6}\|t\|^2.
	\end{align}
	Therefore, under \Cref{assump1,assump_prior1,assump_kappa}, and using the notation of \Cref{sec:additional_notation},
	\begin{align}
		&\int_{\|u\|>\deltamle\sqrt{n}}|g(u)|n^{-d/2}\posterior(n^{-1/2}u+\mle)du\notag\\
		=&\frac{\int_{\|u\|>\deltamle\sqrt{n}}|g(u)|\prior(n^{-1/2}u+\mle)e^{\loglikelihood(n^{-1/2}u+\mle)}du}{\int_{\mathbbm{R}^d}\prior(n^{-1/2}t+\mle)e^{L_n(n^{-1/2}t+\mle)}dt}\notag\\
		\leq &\frac{\int_{\|u\|>\deltamle\sqrt{n}}|g(u)|\pi(n^{-1/2}u+\mle)e^{L_n(n^{-1/2}u+\mle)}du}{\int_{\|t\|\leq\deltamle\sqrt{n}}\prior(n^{-1/2}t+\mle)e^{\loglikelihood(n^{-1/2}t+\mle)}dt}\notag\\
		= &\frac{\int_{\|u\|>\deltamle\sqrt{n}}|g(u)|\prior(n^{-1/2}u+\mle)e^{\loglikelihood(n^{-1/2}u+\mle)-\loglikelihood(\mle)}du}{\int_{\|t\|\leq\deltamle\sqrt{n}}\pi(n^{-1/2}t+\mle)e^{\loglikelihood(n^{-1/2}t+\mle)-\loglikelihood(\mle)}dt}\notag\\
		\stackrel{\cref{taylor1}}\leq & \frac{e^{-n\kappamle}\int_{\|u\|>\deltamle\sqrt{n}}|g(u)|\prior(n^{-1/2}u+\mle)du}{\int_{\|t\|\leq\deltamle\sqrt{n}}\prior(n^{-1/2}t+\mle)e^{-t^T(\fisherinformation(\mle)+(\deltamle M_2/3)I_{d\times d})t/2}dt}\notag\\
		\leq &\frac{ n^{d/2}e^{-n\kappamle}\hat{M}_1\left|\text{det}\left(\jhatplus\right)\right|^{1/2}\int_{\|u\|>\deltamle}|g(u\sqrt{n})|\pi(u+\mle)du}{\left(2\pi \right)^{d/2}\hspace{-1mm}\left( \hspace{-1mm}1-\decayhatplus\right)},\label{i211}
	\end{align}
	if $n>\frac{\Tr\left[\jhatplus^{-1}\right]}{\deltamle^2}$, where the last inequality follows from Lemma \ref{lemma1}.
	Therefore, if $n>\frac{\Tr\left[\jhatplus^{-1}\right]}{\deltamle^2}$ and if \Cref{assump1,assump_prior1,assump_kappa} are satisfied,
	\begin{align}
		I_{2,1}^{MLE}\leq & \left|\int_{\|u\|>\deltamle\sqrt{n}}g(u)\frac{\sqrt{\left|\text{det}\,\fisherinformation(\mle)\right|}e^{-u^T\fisherinformation(\mle)u/2}}{(2\pi)^{d/2}}du\right|\notag\\
		&+\frac{ n^{d/2}e^{-n\kappamle}\hat{M}_1\left|\text{det}\left(\jbarplus\right)\right|^{1/2}\int_{\|u\|>\deltamle}|g(u\sqrt{n})|\pi(u+\mle)du}{\left(2\pi \right)^{d/2}\left( 1-\decayhatplus\right)}\label{i21}.
	\end{align}
	Now, under \Cref{assump1,assump_prior1,assump_size_of_delta,assump_fisher,assump_kappa,assump_prior2},
	\begin{align}
		&\frac{n^{-d/2}}{C_n^{MLE}}\int_{\|t\|\leq \deltamle\sqrt{n}}|g(t)|\posterior(n^{-1/2}t+\mle)dt\notag\\
		\leq&\frac{\hat{M}_1\widetilde{M}_1\int_{\|t\|\leq\deltamle\sqrt{n}}|g(t)|e^{\loglikelihood(n^{-1/2}t+\loglikelihood)-\loglikelihood(\mle)}dt}{\int_{\|u\|\leq\deltamle\sqrt{n}}e^{\loglikelihood(n^{-1/2}u+\hat{\theta}_n)-\loglikelihood(\mle)}du}\notag\\
		\leq &\frac{\hat{M}_1\widetilde{M}_1\int_{\|t\|\leq\deltamle\sqrt{n}}|g(t)|e^{-\frac{1}{2}t^T\left(\fisherinformation(\mle)-\frac{M_2\deltamle}{3}I_{d\times d}\right)t}dt}{\int_{\|u\|\leq\deltamle\sqrt{n}}e^{-u^T(\fisherinformation(\mle)+(\deltamle M_2/3)I_{d\times d})u/2}du}\notag\\
		\leq &\frac{ \hat{M}_1\widetilde{M}_1\left|\text{det}\left(\jhatplus\right)\right|^{1/2}\int_{\|t\|\leq\deltamle\sqrt{n}}|g(t)|e^{-\frac{1}{2}t^T\jhatminus t}dt}{\left(2\pi \right)^{d/2}\hspace{-1mm}\left( 1-\decayhatplus\right)},\label{i22}
	\end{align}
	where the last inequality follows from Lemma \ref{lemma1}. A bound on $I_{2,2}$ can be obtained by combining \cref{i22} with \cref{i21} applied to $g=1$. Indeed, we thus obtain:
	\begin{align}
I_{2,2}^{MLE}\leq &\frac{ \hat{M}_1\widetilde{M}_1\left|\text{det}\left(\jhatplus\right)\right|^{1/2}\int_{\|u\|\leq\deltamle\sqrt{n}}|g(u)|e^{-\frac{1}{2}u^T\jhatminus u}du}{\left(2\pi \right)^{d/2}\left( 1-\decayhatplus\right)}\notag\\
&\cdot\left\{\decayhat+\frac{ n^{d/2}e^{-n\kappamle}\hat{M}_1\left|\text{det}\left(\jhatplus\right)\right|^{1/2}}{\left(2\pi \right)^{d/2}\hspace{-1mm}\left( 1-\decayhatplus\right)}\right\}
\label{i221},
	\end{align}
where we used Lemma \ref{lemma1}.   A bound on $I_2^{MLE}$ is obtained by adding together the bounds on $I_{2,1}^{MLE}$ (\cref{i21}) and $I_{2,2}^{MLE}$ (\cref{i221}).
\begin{remark}\label{remark_proofs_a}
	Note that, for $g$, such that $|g|\leq U$, for some $U>0$, we have that
		\begin{align*}
		&\frac{n^{-d/2}}{C_n^{MLE}}\int_{\|t\|\leq \deltamle\sqrt{n}}|g(t)|\posterior(n^{-1/2}t+\mle)dt\leq U.
		\end{align*}
	Therefore, for $|g|\leq U$, the same argument as above yields a simpler bound:
	\begin{align*}
		I_{2,2}^{MLE}\leq& U\left\{\decayhat+\frac{ n^{d/2}e^{-n\kappamle}\hat{M}_1\left|\text{det}\left(\jhatplus\right)\right|^{1/2}}{\left(2\pi \right)^{d/2}\hspace{-1mm}\left( 1-\decayhatplus\right)}\right\}.
	\end{align*}

\end{remark}

	\subsection{Controlling term $I_{2}^{MAP}$}\label{sec:control_i2map}
	Note that
	\begin{align*}
		I_2^{MAP}\leq &\left|\int_{\|u\|>\deltabar\sqrt{n}}g(u)\left(n^{-d/2} \posterior(n^{-1/2} u+\map)-\frac{\sqrt{\left|\text{det}\,\postfisherinformation(\map)\right|}e^{-u^T\postfisherinformation(\map)u/2}}{(2\pi)^{d/2}}\right)du\right|\\
		&+ \frac{n^{-d/2}}{C_n^{MAP}}\int_{\|t\|\leq \deltabar\sqrt{n}}|g(t)|\posterior(n^{-1/2}t+\map)dt\\
		&\hspace{1cm}\cdot\left|\int_{\|u\|>\deltabar\sqrt{n}}\left(n^{-d/2} \posterior(n^{-1/2} u+\map)-\frac{\sqrt{\left|\text{det}\,\postfisherinformation(\map)\right|}e^{-u^T\postfisherinformation(\map)u/2}}{(2\pi)^{d/2}}\right)du\right|\\
		=:&I_{2,1}^{MAP}+I_{2,2}^{MAP}.
	\end{align*}

%
	Note that, using \Cref{assump1,assump_prior1,assump_kappa1}, a calculation similar to \cref{i211}, yields

	\begin{align*}
		&\int_{\|u\|>\deltabar\sqrt{n}}|g(u)|n^{-d/2}\posterior(n^{-1/2}u+\bar{\theta}_n)du\\
		=&\int_{\|v-\bar{\theta}_n\|>\deltabar}\left|g\left(\sqrt{n} (v-\bar{\theta}_n)\right)\right|\,\posterior(v)dv\\
		\leq&\frac{\int_{\|v-\mle\|>\deltabar-\|\mle-\map\|}\left|g\left(\sqrt{n} (v-\map)\right)\right|\,\prior(v)e^{\loglikelihood(v)}dv}{\int_{\mathbbm{R}^d}\prior(t)e^{\loglikelihood(t)}dt}\\
	\leq&\frac{\int_{\|v-\hat{\theta}_n\|>\deltabar-\|\hat{\theta}_n-\bar{\theta}_n\|}\left|g\left(\sqrt{n} (v-\map)\right)\right|\,\prior(v)e^{\loglikelihood(v)-\loglikelihood(\hat{\theta}_n)}dv}{\int_{\|t-\mle\|\leq\deltamle}\prior(t)e^{\loglikelihood(t)-L_n(\mle)}dt}\\
		\leq & \frac{n^{d/2}e^{-n\kappabar}\int_{\|v-\hat{\theta}_n\|>\deltabar-\|\hat{\theta}_n-\bar{\theta}_n\|}\left|g\left(\sqrt{n} (v-\bar{\theta}_n)\right)\right|\pi(v)dv}{\int_{\|t\|\leq\deltamle\sqrt{n}}\prior(n^{-1/2}t+\hat{\theta}_n)e^{-t^T(\fisherinformation(\mle)+(\deltamle M_2/3)I_{d\times d})t/2}dt}\\
		\leq &\frac{ n^{d/2}e^{-n\kappabar}\hat{M}_1\left|\text{det}\left(\jhatplus\right)\right|^{1/2}\int_{\|v-\hat{\theta}_n\|>\deltabar-\|\hat{\theta}_n-\bar{\theta}_n\|}\left|g\left(\sqrt{n} (v-\bar{\theta}_n)\right)\right|\pi(v)du}{\left(2\pi \right)^{d/2}\hspace{-1mm}\left( 1-\decayhatplus\right)},
	\end{align*}
	if $n>\frac{\Tr\left[\jhatplus^{-1}\right]}{\deltamle^2}$, where the last inequality follows from Lemma \ref{lemma1}.
	Therefore, under \Cref{assump1,assump_prior1,assump_kappa1} and if $n>\frac{\Tr\left[\jhatplus^{-1}\right]}{\deltamle^2}$,
	\begin{align}
		I_{2,1}^{MAP}\leq & \left|\int_{\|u\|>\deltabar\sqrt{n}}g(u)\frac{\sqrt{\left|\text{det}\,\postfisherinformation(\map)\right|}e^{-u^T\postfisherinformation(\map)u/2}}{(2\pi)^{d/2}}du\right|\notag\\
		&+\frac{ n^{d/2}e^{-n\kappabar}\hat{M}_1\left|\text{det}\left(\jhatplus\right)\right|^{1/2}\int_{\|v-\hat{\theta}_n\|>\deltabar-\|\hat{\theta}_n-\bar{\theta}_n\|}\left|g\left(\sqrt{n} (v-\bar{\theta}_n)\right)\right|\pi(v)du}{\left(2\pi \right)^{d/2}\hspace{-1mm}\left( 1-\decayhatplus\right)}.\label{i21_map}
	\end{align}
	Now, under \Cref{assump1,assump_prior1,assump2,assump_size_of_delta_bar,assump7,assump_kappa1},
	\begin{align}
		&\frac{n^{-d/2}}{C_n^{MAP}}\int_{\|t\|\leq \deltabar\sqrt{n}}|g(t)|\posterior(n^{-1/2}t+\map)dt\notag\\
		\leq&\frac{\int_{\|t\|\leq\deltabar\sqrt{n}}|g(t)|e^{\l\logposterior(n^{-1/2}t+\map)-\logposterior(\map)}dt}{\int_{\|u\|\leq\deltabar\sqrt{n}}e^{\logposterior(n^{-1/2}u+\map)-\logposterior(\map)}du}\notag\\
		\leq &\frac{\int_{\|t\|\leq\deltabar\sqrt{n}}|g(t)|e^{-\frac{1}{2}t^T\left(\postfisherinformation(\map)-\frac{\overline{M}_2\deltabar}{3}I_{d\times d}\right)t}dt}{\int_{\|u\|\leq\deltabar\sqrt{n}}e^{-u^T(\postfisherinformation(\map)+(\deltabar \overline{M}_2/3)I_{d\times d})u/2}du}\notag\\
		\leq &\frac{ \left|\text{det}\left(\jbarplus\right)\right|^{1/2}\int_{\|u\|\leq\deltabar\sqrt{n}}|g(u)|e^{-\frac{1}{2}u^T\jbarminus u}du}{\left(2\pi \right)^{d/2}\hspace{-1mm}\left( 1-\decaybarplus\right)},\label{i222}
	\end{align}
	where the last inequality follows from Lemma \ref{lemma1}. A bound on $I_{2,2}^{MAP}$ can be obtained by combining \cref{i222} with \cref{i21_map} applied to $g=1$. Indeed, we obtain:
	\begin{align}
		I_{2,2}^{MAP}\leq &\frac{ \left|\text{det}\left(\jbarplus\right)\right|^{1/2}\int_{\|u\|\leq\deltabar\sqrt{n}}|g(u)|e^{-\frac{1}{2}u^T\jbarminus u}du}{\left(2\pi \right)^{d/2}\hspace{-1mm}\left( 1-\decaybarplus\right)}\notag\\
		&\cdot \left\{\decaybar +\frac{ n^{d/2}e^{-n\kappabar}\hat{M}_1\left|\text{det}\left(\jhatplus\right)\right|^{1/2}}{\left(2\pi \right)^{d/2}\hspace{-1mm}\left( 1-\decayhatplus\right)}\right\}\label{i22_map},
	\end{align}
where we applied Lemma \ref{lemma1}.
	\begin{remark}\label{remark_i2_map}
		As in Remark \ref{remark_proofs_a}, our bound gets simpler if $|g|\leq U$, for some $U>0$. In that case, instead of \cref{i22_map}, we can write:
		\begin{align*}
				I_{2,2}^{MAP}\leq &U \left\{\decaybar+\frac{ n^{d/2}e^{-n\kappabar}\hat{M}_1\left|\text{det}\left(\jhatplus\right)\right|^{1/2}}{\left(2\pi \right)^{d/2}\hspace{-1mm}\left( 1-\decayhatplus\right)}\right\}.
		\end{align*}
	\end{remark}
	\section{Proofs of Theorems \ref{main_map}, \ref{thm:wasserstein_map} and \ref{w2_map}}\label{appendix_b}
	Throughout this section we adopt the notation of \Cref{introductory_arguments}. Additionally, for any probability measure $\mu$, we let $[\mu]_{B_0(\deltabar\sqrt{n})}$ denote its restriction (truncation) to the ball of radius $\deltabar\sqrt{n}$ around $0$.
	
	In all the proofs below, we wish to control the quantity $D_g^{MAP}$ of \cref{dgmap} for all functions $g$ that satisfy certain prescribed criteria. In the proof of \Cref{main_map}, we look at functions $g$ that are indicators of measurable sets, in the proof of \Cref{thm:wasserstein_map} we look at 1-Lipschitz functions $g$ and in the proof of \Cref{w2_map} at those that are of the form $g(x)=\left<v,x\right>^2$ for some $v\in\mathbbm{R}^d$ with $\|v\|=1$. In order to prove \Cref{main_map,thm:wasserstein_map,w2_map}, we will bound terms $I_2^{MAP}$ and $I_1^{MAP}$ of \cref{i1i2map} separately.
\subsection{Proof of \Cref{main_map}}

\subsubsection{Controlling term $I_2^{MAP}$}We wish to obtain a uniform bound on $I_2^{MAP}$ for all functions $g$ that are indicators of measurable sets.
Every indicator function is upper-bounded by one, so we can use Remark \ref{remark_i2_map} to obtain:
\begin{align}
	&I_{2,2}^{MAP}\leq \decaybar+\frac{ n^{d/2}e^{-n\kappabar}\hat{M}_1\left|\text{det}\left(\jhatplus\right)\right|^{1/2}}{\left(2\pi \right)^{d/2}\hspace{-1mm}\left( 1-\decayhatplus\right)}.\label{bi22}
\end{align}
Similarly, since $|g|\leq 1$, we can use \cref{i21_map}  and Lemma \ref{lemma1} to obtain:
\begin{align}
	&I_{2,1}^{MAP}\leq \decaybar+\frac{ n^{d/2}e^{-n\kappabar}\hat{M}_1\left|\text{det}\left(\jhatplus\right)\right|^{1/2}}{\left(2\pi \right)^{d/2}\hspace{-1mm}\left( 1-\decayhatplus\right)}.\label{bi21}
\end{align}
From \cref{bi22,bi21}, it follows that
\begin{align}
	I_2^{MAP}\leq& 2\,\decaybar+\frac{2 n^{d/2}e^{-n\kappabar}\hat{M}_1\left|\text{det}\left(\jhatplus\right)\right|^{1/2}}{\left(2\pi \right)^{d/2}\hspace{-1mm}\left( 1-\decayhatplus\right)}.\label{i2_map_tv}
\end{align}
\subsubsection{Controlling term $I_1^{MAP}$ using the log-Sobolev inequality}\label{subs:controlling_i1}
 Let $\text{KL}\left(\cdot\|\cdot\right)$ denote the Kullback-Leibler divergence (\cref{def:kl_divergence}) and $\text{Fisher}\left(\cdot\|\cdot\right)$ denote the Fisher divergence (\cref{def:fisher_divergence}). Let $F_n^{MAP}$ be given by \cref{f_n_map}.
 
 Note that, by \Cref{assump2}, for $t$ such that $\|t\|<\sqrt{n}\deltabar$, we have
\begin{align*}
	-n^{-1}\logposterior''(\map+n^{-1/2}t)\succeq\postfisherinformation(\map)-\deltabar \overline{M}_2I_{d\times d}\succeq \left(\minevaluemap- \deltabar\overline{M}_2\right)I_{d\times d}.
\end{align*}
This means that, inside the convex set $\{t\in\mathbbm{R}^d\,:\,\|t\|<\sqrt{n}\deltabar\}$, the density of $\sqrt{n}(\postparam-\map)$ is $ \left(\minevaluemap- \deltabar\overline{M}_2\right)$-strongly log-concave (see e.g. \citealt{saumard}). 
 Using  Proposition \ref{prop_lsi} (the \textbf{Bakry-\'Emery criterion}), we have that $\left[\mathcal{L}\left(\sqrt{n}\left(\postparam-\map\right)\right)\right]_{B_{0}(\deltabar\sqrt{n})}$ satisfies the \textbf{log-Sobolev inequality} LSI$ \left(\minevaluemap- \deltabar\overline{M}_2\right)$ (see Definition \ref{def:log-sobolev}).  By combining the log-Sobolev inequality with Pinsker's inequality (Proposition \ref{pinsker}) we obtain that, for all functions $g$, which are indicators of measurable sets,
\begin{align}
	I_1^{MAP}\leq&\text{TV}\left(\left[\mathcal{L}\left(\sqrt{n}\left(\postparam-\map\right)\right)\right]_{B_{0}(\deltabar\sqrt{n})},\left[\mathcal{N}(0,\postfisherinformation(\map)^{-1})\right]_{B_{0}(\deltabar\sqrt{n})}\right)\notag\\
	\underset{\text{inequality}}{\stackrel{\text{Pinsker's}}\leq}&\sqrt{\frac{1}{2}\text{KL}\left(\left.\left[\mathcal{N}(0,\postfisherinformation(\map)^{-1})\right]_{B_{0}(\deltabar\sqrt{n})}\,\right\|\,\left[\mathcal{L}\left(\sqrt{n}\left(\postparam-\map\right)\right)\right]_{B_{0}(\deltabar\sqrt{n})}\right)}\notag\\
	\underset{\text{inequality}}{\stackrel{\text{log-Sobolev}}\leq} &\frac{\sqrt{\text{Fisher}\left(\left.\left[\mathcal{N}(0,\postfisherinformation(\map)^{-1})\right]_{B_{0}(\deltabar\sqrt{n})}\,\right\|\,\left[\mathcal{L}\left(\sqrt{n}\left(\postparam-\map\right)\right)\right]_{B_{0}(\deltabar\sqrt{n})}\right)}}{2\sqrt{\minevaluemap-\deltabar\, \overline{M}_2}}\notag\\
	= &\frac{1}{2\sqrt{\minevaluemap-\deltabar\, \overline{M}_2}}\notag\\
	&\cdot\sqrt{\int_{\|u\|\leq\deltabar\sqrt{n}}\frac{\sqrt{\det \jbar}e^{-u^T\postfisherinformation(\map)u/2}}{F_n^{MAP}(2\pi)^{d/2}}\left\|\postfisherinformation(\map)u+\frac{\logposterior'(n^{-1/2}u+\map)}{\sqrt{n}}\right\|^2du}\notag\\
	\underset{\text{theorem}}{\stackrel{\text{Taylor's}}\leq} &\frac{\overline{M}_2}{4\sqrt{n\left(\minevaluemap-\deltabar\, \overline{M}_2\right)}}\sqrt{\int_{\|u\|\leq\deltabar\sqrt{n}}\frac{\det \jbar^{1/2}e^{-u^T\postfisherinformation(\map)u/2}}{F_n^{MAP}(2\pi)^{d/2}}\left\|u\right\|^4du}\notag\\
	\leq &\frac{\sqrt{3}\,\Tr\left[\postfisherinformation(\map)^{-1}\right] \overline{M}_2}{4\sqrt{n\left(\minevaluemap-\deltabar\overline{M}_2\right)\left( 1- \decaybar\right)}},\label{i1_map_tv}
\end{align}
as long as $n>\frac{\Tr\left[\postfisherinformation(\map)^{-1}\right]}{\left.\deltabar\right.^2}$, where the last inequality follows from Lemma \ref{lemma1} and the fact that, for $U\sim\mathcal{N}(0,\postfisherinformation(\map)^{-1})$, $\mathbbm{E}\|U\|^4\leq 3 \Tr\left[\postfisherinformation(\map)^{-1}\right]^2$. To see this last fact, use the spectral theorem and write $\postfisherinformation(\map)^{-1}=P^T\Lambda P$, for an orthogonal matrix $P$ and diagonal matrix $\Lambda$. Then, it follows that, for $N\sim\mathcal{N}(0,I_{d\times d}),$ $U^TU = N^TP^T\Lambda PN = (PN)^T\Lambda (PN)$. Since $(PN)\sim\mathcal{N}(0,I_{d\times d})$, note that $\mathbbm{E}(U^TU)^2=\mathbbm{E}\left[\left(\sum_{i=1}^dN_{\lambda_i}^2\right)^2\right]\leq 3 \left(\sum_{i=1}^d\lambda_i\right)^2$, where $\lambda_i$'s are the diagonal elements of $\Lambda$ and $N_{\lambda_i}$'s are independent such that $N_{\lambda_i}\sim\mathcal{N}(0,\lambda_i)$ for $i=1,\dots,d$.

\subsubsection{Conclusion}The result now follows from adding together bounds in \cref{i2_map_tv,i1_map_tv}.

\subsection{Proof of \Cref{thm:wasserstein_map}}
\subsubsection{Controlling term $I_2^{MAP}$} Now we wish to control $I_2^{MAP}$ uniformly over all functions $g$ which are $1$-Lipschitz.
Let us fix a function $g$ that is $1$-Lipschitz and WLOG set $g(0)=0$. In that case $|g(u)|\leq \|u\|$ and, using the notation of \Cref{introductory_arguments} and \cref{i21_map},
\begin{align*}
	I_{2,1}^{MAP}\leq & \int_{\|u\|>\deltabar\sqrt{n}}\|u\|\frac{\sqrt{\left|\text{det}\,\postfisherinformation(\map)\right|}e^{-u^T\postfisherinformation(\map)u/2}}{(2\pi)^{d/2}}du\notag\\
	&+\frac{ n^{d/2+1/2}e^{-n\kappabar}\hat{M}_1\left|\text{det}\left(\jhatplus\right)\right|^{1/2}\int_{\|v-\hat{\theta}_n\|>\deltabar-\|\hat{\theta}_n-\bar{\theta}_n\|}\| v-\bar{\theta}_n\|\pi(v)du}{\left(2\pi \right)^{d/2}\hspace{-1mm}\left( 1-\decayhatplus\right)}.
\end{align*}
Now, for $U\sim\mathcal{N}(0,\postfisherinformation(\map)^{-1})$, and assuming that  $n>\frac{\Tr\left(\postfisherinformation(\map)^{-1}\right)}{\left.\deltabar\right.^2}$,
\begin{align}
	&\int_{\|u\|>\deltabar\sqrt{n}}\|u\|\frac{\sqrt{\left|\text{det}\,\postfisherinformation(\map)\right|}e^{-u^T\postfisherinformation(\map)u/2}}{(2\pi)^{d/2}}du\notag\\
	=&\int_0^{\infty}\mathbbm{P}\left[\|U\|\mathbbm{1}_{[\|U\|>\deltabar\sqrt{n}]}>t\right]dt\notag\\
	\leq &\int_0^{\infty}\mathbbm{P}\left[\|U\|>\max(t,\deltabar\sqrt{n})\right]dt\notag\\
	\stackrel{\Cref{lemma1}}\leq &\int_{\deltabar\sqrt{n}}^{\infty}\exp\left[-\frac{1}{2}\left(t-\sqrt{\Tr\left(\postfisherinformation(\map)^{-1}\right)}\right)^2\minevaluemap\right]dt\notag\\
	&+\deltabar \sqrt{n}\,\exp\left[-\frac{1}{2}\left(\deltabar\sqrt{n}-\sqrt{\Tr\left(\postfisherinformation(\map)^{-1}\right)}\right)^2\minevaluemap\right] \notag\\
	\leq &\left(\deltabar\sqrt{n} + \sqrt{\frac{2\pi}{\minevaluemap}}\right) \exp\left[-\frac{1}{2}\left(\deltabar\sqrt{n}-\sqrt{\Tr\left(\postfisherinformation(\map)^{-1}\right)}\right)^2\minevaluemap\right].\label{normal_integral_w1}
\end{align}
This means that
\begin{align}
	I_{2,1}^{MAP}\leq &\left(\deltabar\sqrt{n} + \sqrt{\frac{2\pi}{\minevaluemap}}\right)\decaybar\notag\\
		&+\frac{ n^{d/2+1/2}e^{-n\kappabar}\hat{M}_1\left|\text{det}\left(\jhatplus\right)\right|^{1/2}\int_{\|v-\hat{\theta}_n\|>\deltabar-\|\mle-\bar{\theta}_n\|}\|v-\map\|\pi(v)du}{\left(2\pi \right)^{d/2}\hspace{-1mm}\left( 1-\decayhatplus\right)}.\label{i21_w1_map}
\end{align}
Now, using \cref{i22_map}, we have
\begin{align}
	I_{2,2}^{MAP}\leq &\frac{ \left|\text{det}\left(\jbarplus\right)\right|^{1/2}\left|\text{det}\left(\jbarminus\right)\right|^{-1/2}\sqrt{\Tr\left[\jbarminus^{-1}\right]}}{ 1-\decaybarplus}\notag\\
	&\cdot \Bigg\{\decaybar+\frac{ n^{d/2}e^{-n\kappabar}\hat{M}_1\left|\text{det}\left(\jhatplus\right)\right|^{1/2}}{\left(2\pi \right)^{d/2}\left( 1-\decayhatplus\right)}\Bigg\}.\label{i22_w1_map}
\end{align}
Adding together bounds in \Cref{i21_w1_map,i22_w1_map} now yields a bound on $I_2^{MAP}$.

\subsubsection{Controlling term $I_1^{MAP}$ using the log-Sobolev inequality and the transportation-entropy inequality}
As in \Cref{subs:controlling_i1}, we shall use the log-Sobolev inequality  LSI$ \left(\minevaluemap- \deltabar\overline{M}_2\right)$ for the measure $\left[\mathcal{L}\left(\sqrt{n}\left(\postparam-\map\right)\right)\right]_{B_{0}(\deltabar\sqrt{n})}$, implied by Proposition \ref{prop_lsi}. A consequence of the log-Sobolev inequality is that we can apply the \textbf{transportation-entropy inequality} for $\left[\mathcal{L}\left(\sqrt{n}\left(\postparam-\map\right)\right)\right]_{B_{0}(\deltabar\sqrt{n})}$, given by Proposition \ref{prop_talagrand}. It lets us upper bound the 1- and 2-Wasserstein distances (see \cref{def:w2}) by a constant times the KL divergence. The KL divergence is in turn bounded by a constant times the Fisher divergence by the log-Sobolev inequality. Let $W_2(\cdot,\cdot)$ denote the $2$-Wasserstein distance and $W_1(\cdot,\cdot)$ denote the 1-Wasserstein distance. We have that for all $1$-Lipschitz functions $g$,
\begin{align}
	I_1^{MAP}\leq &W_1\left(\left[\mathcal{L}\left(\sqrt{n}\left(\postparam-\map\right)\right)\right]_{B_{0}(\deltabar\sqrt{n})},\left[\mathcal{N}(0,\postfisherinformation(\map)^{-1})\right]_{B_{0}(\deltabar\sqrt{n})}\right)\notag\\
	\stackrel{\cref{wasserstein_ineq}}\leq& W_2\left(\left[\mathcal{L}\left(\sqrt{n}\left(\postparam-\map\right)\right)\right]_{B_{0}(\deltabar\sqrt{n})},\left[\mathcal{N}(0,\postfisherinformation(\map)^{-1})\right]_{B_{0}(\deltabar\sqrt{n})}\right)\notag\\
	\underset{\text{inequality}}{\stackrel{\text{transportation-entropy}}\leq} &\frac{\sqrt{2\,\text{KL}\left(\left.\left[\mathcal{N}(0,\postfisherinformation(\map)^{-1})\right]_{B_{0}(\deltabar\sqrt{n})}\,\right\|\,\left[\mathcal{L}\left(\sqrt{n}\left(\postparam-\map\right)\right)\right]_{B_{0}(\deltabar\sqrt{n})}\right)}}{\sqrt{\minevaluemap-\deltabar\, \overline{M}_2}}\notag\\
		\underset{\text{inequality}}{\stackrel{\text{log-Sobolev}}\leq} &\frac{\sqrt{\text{Fisher}\left(\left.\left[\mathcal{N}(0,\postfisherinformation(\map)^{-1})\right]_{B_{0}(\deltabar\sqrt{n})}\,\right\|\,\left[\mathcal{L}\left(\sqrt{n}\left(\postparam-\map\right)\right)\right]_{B_{0}(\deltabar\sqrt{n})}\right)}}{\minevaluemap-\deltabar\, \overline{M}_2}\notag\\
	\leq &\frac{1}{\minevaluemap-\deltabar\, \overline{M}_2}\notag\\
	&\cdot\sqrt{\int_{\|u\|\leq\deltabar\sqrt{n}}\frac{\sqrt{\det \jbar}e^{-u^T\postfisherinformation(\map)u/2}}{F_n^{MAP}(2\pi)^{d/2}}\left\|\postfisherinformation(\map)u+\frac{\logposterior'(n^{-1/2}u+\map)}{\sqrt{n}}\right\|^2du}\notag\\
	\underset{\text{theorem}}{\stackrel{\text{Taylor's}}\leq} &\frac{\overline{M}_2}{2\sqrt{n}\left(\minevaluemap-\deltabar\, \overline{M}_2\right)}\sqrt{\int_{\|u\|\leq\deltabar\sqrt{n}}\frac{\det \jbar^{1/2}e^{-u^T\postfisherinformation(\map)u/2}}{F_n^{MAP}(2\pi)^{d/2}}\left\|u\right\|^4du}\notag\\
	\leq &\frac{\sqrt{3}\,\Tr\left[\postfisherinformation(\map)^{-1}\right] \overline{M}_2}{2\left(\minevaluemap-\deltabar\overline{M}_2\right)\sqrt{n\left( 1- \decaybar\right)}},\label{i1_map_w1}
\end{align}
	as long as $n>\frac{\Tr\left[\postfisherinformation(\map)^{-1}\right]}{\left.\deltabar\right.^2}$, where the last iequality follows from Lemma \ref{lemma1}.

\subsubsection{Conclusion} The result now follows from adding together bounds in \Cref{i21_w1_map,i22_w1_map,i1_map_w1}.
\subsection{Proof of \Cref{w2_map}}
\subsubsection{Controlling term $I_2^{MAP}$} Now we wish to control $I_2^{MAP}$ uniformly over all functions $g$ which are of the form $g(u)=\left<v,u\right>^2$ for some $v\in\mathbbm{R}^d$ with $\|v\|=1$. Such functions satisfy the following property: $|g(u)|\leq \|u\|^2$. Using the notation of \Cref{introductory_arguments} and \cref{i21_map}, we therefore have that:

\begin{align*}
	I_{2,1}^{MAP}\leq & \int_{\|u\|>\deltabar\sqrt{n}}\|u\|^2\frac{\sqrt{\left|\text{det}\,\postfisherinformation(\map)\right|}e^{-u^T\postfisherinformation(\map)u/2}}{(2\pi)^{d/2}}du\notag\\
	&+\frac{ n^{d/2+1}e^{-n\kappabar}\hat{M}_1\left|\text{det}\left(\jhatplus\right)\right|^{1/2}\int_{\|v-\mle\|>\deltabar-\|\mle-\bar{\theta}_n\|}\| v-\bar{\theta}_n\|^2\pi(v)du}{\left(2\pi \right)^{d/2}\hspace{-1mm}\left( \hspace{-1mm}1-\decayhatplus\right)}.
\end{align*}

Now, for $U\sim\mathcal{N}(0,\postfisherinformation(\map)^{-1})$, and assuming that  $n>\frac{\Tr\left(\postfisherinformation(\map)^{-1}\right)}{\left.\deltabar\right.^2}$,
\begin{align}
	&\int_{\|u\|>\deltabar\sqrt{n}}\|u\|^2\frac{\sqrt{\left|\text{det}\,\postfisherinformation(\map)\right|}e^{-u^T\postfisherinformation(\map)u/2}}{(2\pi)^{d/2}}du\notag\\
	=&\int_0^{\infty}\mathbbm{P}\left[\|U\|^2\mathbbm{1}_{[\|U\|>\deltabar\sqrt{n}]}>t\right]dt\\
	\leq &\int_0^{\infty}\mathbbm{P}\left[\|U\|>\max(\sqrt{t},\deltabar\sqrt{n})\right]dt\notag\\
	\leq &\int_{\deltabar^2n}^{\infty}\exp\left[-\frac{1}{2}\left(\sqrt{t}-\sqrt{\Tr\left(\postfisherinformation(\map)^{-1}\right)}\right)^2\minevaluemap\right]dt\notag\\
	&+\deltabar^2n\,\exp\left[-\frac{1}{2}\left(\deltabar\sqrt{n}-\sqrt{\Tr\left(\postfisherinformation(\map)^{-1}\right)}\right)^2\minevaluemap\right] \notag\\
	\leq &\left(\deltabar^2n+ \sqrt{\frac{2\pi}{\minevaluemap}}\right) \exp\left[-\frac{1}{2}\left(\deltabar\sqrt{n}-\sqrt{\Tr\left(\postfisherinformation(\map)^{-1}\right)}\right)^2\minevaluemap\right],\label{normal_integral_w2}
\end{align}
where we have used Lemma \ref{lemma1}.
This means that
\begin{align}
	I_{2,1}^{MAP}\leq &\left(\deltabar^2n+ \sqrt{\frac{2\pi}{\minevaluemap}}\right)\decaybar\notag\\
	&+\frac{ n^{d/2+1}e^{-n\kappabar}\hat{M}_1\left|\text{det}\left(\jhatplus\right)\right|^{1/2}\int_{\|v-\mle\|>\deltabar-\|\mle-\bar{\theta}_n\|}\|v-\map\|^2\pi(v)du}{\left(2\pi \right)^{d/2}\hspace{-1mm}\left( 1-\decayhatplus\right)}.\label{i21_map_w2}
\end{align}
Now, using \cref{i22_map}, we have
\begin{align}
	I_{2,2}^{MAP}\leq &\frac{ \left|\text{det}\left(\jbarplus\right)\right|^{1/2}\left|\text{det}\left(\jbarminus\right)\right|^{-1/2}\Tr\left[\jbarminus^{-1}\right]}{ 1-\decaybarplus}\notag\\
	&\cdot \Bigg\{\decaybar+\frac{ n^{d/2}e^{-n\kappabar}\hat{M}_1\left|\text{det}\left(\jhatplus\right)\right|^{1/2}}{\left(2\pi \right)^{d/2}\hspace{-1mm}\left( 1-\decayhatplus\right)}\Bigg\}.\label{i22_map_w2}
\end{align}
A bound on $I_2^{MAP}$ now follows from adding up the bounds in \cref{i21_map_w2,i22_map_w2}.

\subsubsection{Controlling term $I_1^{MAP}$ using the log-Sobolev inequality and the transportation-information inequality}\label{sec:control_i1map}
Note that calculation from \cref{i1_map_w1} yields that
\begin{align}
	&W_2\left(\left[\mathcal{L}\left(\sqrt{n}\left(\postparam-\map\right)\right)\right]_{B_{0}(\deltabar\sqrt{n})},\left[\mathcal{N}(0,\postfisherinformation(\map)^{-1})\right]_{B_{0}(\deltabar\sqrt{n})}\right)\notag\\
	\leq & \frac{\sqrt{3}\,\Tr\left[\postfisherinformation(\map)^{-1}\right] \overline{M}_2}{2\left(\minevaluemap-\deltabar\overline{M}_2\right)\sqrt{n\left( 1-\decaybar\right)}}.\label{bound_w2_map}
\end{align}
Now, let us fix two random vectors: $X\sim \left[\mathcal{L}\left(\sqrt{n}\left(\postparam-\map\right)\right)\right]_{B_{0}(\deltabar\sqrt{n})}$ and \\ $Y\sim \left[\mathcal{N}(0,\postfisherinformation(\map)^{-1})\right]_{B_{0}(\deltabar\sqrt{n})}$ and a vector $v$ such that $\|v\|=1$. Let $\gamma$ denote the set of all couplings between $\left[\mathcal{L}\left(\sqrt{n}\left(\postparam-\map\right)\right)\right]_{B_{0}(\deltabar\sqrt{n})}$ and $ \left[\mathcal{N}(0,\postfisherinformation(\map)^{-1})\right]_{B_{0}(\deltabar\sqrt{n})}$. Let $(\tilde{X},\tilde{Y})\in \gamma$ be such that
\begin{align*}
	\inf_{(Z_1,Z_2)\in\gamma}\mathbbm{E}\left[\|Z_1-Z_2\|^2\right]=\mathbbm{E}\left[\|\tilde{X}-\tilde{Y}\|^2\right].
\end{align*}
It follows that:
\begin{align*}
	\mathbbm{E}\left[\left<v,X\right>^2\right]-\mathbbm{E}\left[\left<v,Y\right>^2\right]=&\mathbbm{E}\left[\left<v,\tilde{X}-\tilde{Y}\right>\left<v,\tilde{X}+\tilde{Y}\right>\right]\\
	=&\mathbbm{E}\left[\left<v,\tilde{X}-\tilde{Y}\right>^2\right]+2\mathbbm{E}\left[\left<v,\tilde{X}-\tilde{Y}\right>\left<v,\tilde{Y}\right>\right]\\
	\leq&\mathbbm{E}\left[\left\|\tilde{X}-\tilde{Y}\right\|^2\right]+2\sqrt{\mathbbm{E}\left[\left\|\tilde{X}-\tilde{Y}\right\|^2\right]}\sqrt{\mathbbm{E}\left[\|Y\|^2\right]}.
\end{align*}
Therefore, for all functions $g$ which are of the form $g(u)=\left<v,u\right>^2$ for some $v\in\mathbbm{R}^d$ with $\|v\|=1$, we have that
\begin{align}
	I_1^{MAP}\leq &W_2\left(\left[\mathcal{L}\left(\sqrt{n}\left(\postparam-\map\right)\right)\right]_{B_{0}(\deltabar\sqrt{n})},\left[\mathcal{N}(0,\postfisherinformation(\map)^{-1})\right]_{B_{0}(\deltabar\sqrt{n})}\right)^2\notag\\
	&+2W_2\left(\left[\mathcal{L}\left(\sqrt{n}\left(\postparam-\map\right)\right)\right]_{B_{0}(\deltabar\sqrt{n})},\left[\mathcal{N}(0,\postfisherinformation(\map)^{-1})\right]_{B_{0}(\deltabar\sqrt{n})}\right)\notag\\
	&\cdot \frac{\sqrt{\Tr\left[\postfisherinformation(\map)^{-1}\right]}}{\sqrt{1-\decaybar}},\label{i1_map_w2}
\end{align}
if $n>\frac{\Tr\left[\jbar^{-1}\right]}{\deltabar^2}$. The final bound on $I_1^{MAP}$ may now be obtained by using \cref{bound_w2_map}.

\subsubsection{Conclusion} The result now follows by combining \cref{i1_map_w2} with \cref{bound_w2_map} and then summing together with \cref{i21_map_w2,i22_map_w2}.

\section{Proofs of Theorems \ref{theorem_main}, \ref{theorem_1wasserstein} and \ref{theorem_2wasserstein}} \label{appendix_c}
		Throughout this section we adopt the notation of \Cref{introductory_arguments}. For any probability measure $\mu$, we let $[\mu]_{B_0(\deltamle\sqrt{n})}$ denote its restriction (truncation) to the ball of radius $\deltamle\sqrt{n}$ around $0$. In all the proofs below, we wish to control the quantity $D_g^{MLE}$ of \cref{dgmle} for all functions $g$ that satisfy certain prescribed criteria. In the proof of \Cref{main_map}, we look at functions $g$ that are indicators of measurable sets, in the proof of \Cref{thm:wasserstein_map} we look at 1-Lipschitz functions $g$ and in the proof of \Cref{w2_map} at those which are of the form $g(x)=\left<v,x\right>^2$ for some $v\in\mathbbm{R}^d$ with $\|v\|=1$. In order to prove \Cref{theorem_main,theorem_1wasserstein,theorem_2wasserstein}, we will bound terms $I_2^{MLE}$ and $I_1^{MAP}$ of \cref{i1i2mle} separately.

\subsection{Proof of \Cref{theorem_main}}
\subsubsection{Controlling term $I_2^{MLE}$}
We wish to obtain a uniform bound on $I_2^{MLE}$ for all functions $g$ which are indicators of measurable sets. Every indicator function is upper-bounded by one, so we can use Remark \ref{remark_proofs_a} to obtain:

	\begin{align*}
	I_{2,2}^{MLE}\leq& \decayhat+\frac{ n^{d/2}e^{-n\kappamle}\hat{M}_1\left|\text{det}\left(\jhatplus\right)\right|^{1/2}}{\left(2\pi \right)^{d/2}\left( 1-\decayhatplus\right)}.
\end{align*}
Similarly, since $|g|\leq 1$, we can use \cref{i21} and Lemma \ref{lemma1} to obtain
\begin{align*}
	I_{2,1}^{MLE}\leq & \decayhat+\frac{ n^{d/2}e^{-n\kappamle}\hat{M}_1\left|\text{det}\left(\jhatplus\right)\right|^{1/2}}{\left(2\pi \right)^{d/2}\hspace{-1mm}\left( 1-\decayhatplus\right)},
\end{align*}
		which implies that
		\begin{align}
			I_2^{MLE}\leq& 2\,\decayhat+\frac{2 n^{d/2}e^{-n\kappamle}\hat{M}_1\left|\text{det}\left(\jhatplus\right)\right|^{1/2}}{\left(2\pi \right)^{d/2}\hspace{-1mm}\left( 1-\decayhatplus\right)}.\label{i2_mle_tv}
		\end{align}
\subsubsection{Controlling term $I_1^{MLE}$ using the log-Sobolev inequality}\label{subs:controlling_i1_mle}
We shall proceed as we did in \Cref{subs:controlling_i1}.   Let $\text{KL}\left(\cdot\|\cdot\right)$ denote the Kullback-Leibler divergence (\cref{def:kl_divergence}) and $\text{Fisher}\left(\cdot\|\cdot\right)$ denote the Fisher divergence (\cref{def:fisher_divergence}).
Let $F_n^{MLE}$ be given by \cref{fn_mle}.

Note that, by \Cref{assump1}, for $t$ such that $\|t\|<\sqrt{n}\deltamle$, we have
\begin{align*}
	-n^{-1}\loglikelihood''(\mle+n^{-1/2}t)\succeq\fisherinformation(\mle)-\deltamle M_2I_{d\times d}\succeq \left(\minevaluemle- \deltamle M_2\right)I_{d\times d}.
\end{align*}
This means that, inside the convex set $\{t\in\mathbbm{R}^d\,:\,\|t\|<\sqrt{n}\deltamle\}$, the measure whose density, up to a normalizing constant, is given by $t\mapsto e^{L_n(\mle+n^{-1/2}t)}\mathbbm{1}_{\{\|t\|<\sqrt{n}\deltamle \}}(t)$, is $ \left(\minevaluemle- \deltamle M_2\right)$-strongly log-concave (see e.g. \citealt{saumard}). Also, note that $$\frac{\sup_{\|t\|\leq\deltamle\sqrt{n}}\prior(n^{-1/2}t+\mle)}{\inf_{\|t\|\leq\deltamle\sqrt{n}}\prior(n^{-1/2}t+\mle)}\leq \widetilde{M}_1\hat{M}_1.$$

Using the \textbf{Bakry-\'Emery criterion}, given by Proposition \cref{prop_lsi}, and the \textbf{Holley-Stroock perturbation principle}, given by Proposition \ref{holley-stroock}, we therefore have that $\left[\sqrt{n}\left(\postparam-\mle\right)\right]_{B_{0}(\deltamle\sqrt{n})}$ satisfies the \textbf{log-Sobolev inequality} LSI$\left(\frac{\minevaluemle- \deltamle M_2}{\widetilde{M}_1\hat{M}_1 }\right)$  (see Definition \ref{def:log-sobolev}). By combining the log-Sobolev inequality with \textbf{Pinsker's inequality} (Proposition \ref{pinsker}) we obtain that, for all functions $g$, which are indicators of measurable sets,
\begin{align}
	I_1^{MLE}\leq& \text{TV}\left(\left[\mathcal{L}\left(\sqrt{n}\left(\postparam-\mle\right)\right)\right]_{B_{0}(\deltamle\sqrt{n})},\left[\mathcal{N}(0,\fisherinformation(\mle)^{-1})\right]_{B_{0}(\deltamle\sqrt{n})}\right)\notag\\
	\underset{\text{inequality}}{\stackrel{\text{Pinsker's}}\leq}&\sqrt{\frac{1}{2}\text{KL}\left(\left.\left[\mathcal{N}(0,\fisherinformation(\mle)^{-1})\right]_{B_{0}(\deltamle\sqrt{n})}\,\right\|\,\left[\mathcal{L}\left(\sqrt{n}\left(\postparam-\mle\right)\right)\right]_{B_{0}(\deltamle\sqrt{n})}\right)}\notag\\
	\underset{\text{inequality}}{\stackrel{\text{log-Sobolev}}\leq} &\frac{\sqrt{\widetilde{M}_1\hat{M}_1}}{2\sqrt{\minevaluemle-\deltamle M_2}} \sqrt{\text{Fisher}\left(\left.\left[\mathcal{N}(0,\fisherinformation(\mle)^{-1})\right]_{B_{0}(\deltamle\sqrt{n})}\,\right\|\,\left[\mathcal{L}\left(\sqrt{n}\left(\postparam-\mle\right)\right)\right]_{B_{0}(\deltamle\sqrt{n})}\right)} \notag\\
	\leq &\frac{\sqrt{\widetilde{M}_1\hat{M}_1}}{2\sqrt{\minevaluemle-\deltamle M_2}}\notag\\
	&\cdot\sqrt{\int_{\|u\|\leq\deltamle\sqrt{n}}\frac{\sqrt{\det \jhat}e^{-u^T\fisherinformation(\mle)u/2}}{F_n^{MLE}(2\pi)^{d/2}}\left\|\fisherinformation(\mle)u+\frac{\loglikelihood'(n^{-1/2}u+\mle)}{\sqrt{n}}\right\|^2du}\notag\\
	&+\frac{\sqrt{\widetilde{M}_1\hat{M}_1}}{2\sqrt{n\left(\minevaluemle-\deltamle M_2\right)}}\notag\\
	&\hspace{1cm}\cdot\sqrt{\int_{\|u\|\leq\deltamle\sqrt{n}}\frac{\sqrt{\det \jhat}e^{-u^T\fisherinformation(\mle)u/2}}{F_n^{MLE}(2\pi)^{d/2}}\left\|\frac{\prior'(n^{-1/2}u+\mle)}{\prior(n^{-1/2}u+\mle)}\right\|^2du}\notag\\
	\underset{\text{theorem}}{\stackrel{\text{Taylor's}}\leq} &\frac{\sqrt{\widetilde{M}_1\hat{M}_1}M_2}{4\sqrt{n\left(\minevaluemle-\deltamle M_2\right)}}\sqrt{\int_{\|u\|\leq\deltamle\sqrt{n}}\frac{\sqrt{\det \jhat}e^{-u^T\fisherinformation(\mle)u/2}}{F_n^{MLE}(2\pi)^{d/2}}\left\|u\right\|^4du}\notag\\
	&+\frac{M_1\sqrt{\widetilde{M}_1\hat{M}_1}}{2\sqrt{n\left(\minevaluemle-\deltamle M_2\right)}}\notag\\
	\leq &\frac{\sqrt{3}\,\Tr\left[\fisherinformation(\mle)^{-1}\right]\sqrt{\widetilde{M}_1\hat{M}_1}\,M_2}{4\sqrt{n\left(\minevaluemle-\deltamle M_2\right)\left( 1- \decayhat\right)}}+\frac{M_1\sqrt{\widetilde{M}_1\hat{M}_1}}{2\sqrt{n\left(\minevaluemle-\deltamle M_2\right)}},\label{i1_mle_tv}
\end{align}
as long as $n>\frac{\Tr\left[\fisherinformation(\mle)^{-1}\right]}{\deltamle^2}$, where we have used Lemma \ref{lemma1}.

\subsubsection{Conclusion}	The result now follows by adding together the bounds in \cref{i2_mle_tv,i1_mle_tv}.

\subsection{Proof of \Cref{theorem_1wasserstein}}
\subsubsection{Controlling term $I_2^{MLE}$}
Now we wish to control $I_2^{MLE}$ uniformly over all functions $g$ which are 1-Lipschitz and WLOG set $g(0)=0$. It follows that $|g(u)|\leq \|u\|$ and, using the notation of \Cref{introductory_arguments} and \cref{i21},
	\begin{align*}
	I_{2,1}^{MLE}\leq & \int_{\|u\|>\deltamle\sqrt{n}}\|u\|\frac{\sqrt{\left|\text{det}\,\fisherinformation(\mle)\right|}e^{-u^T\fisherinformation(\mle)u/2}}{(2\pi)^{d/2}}du\notag\\
	&+\frac{ n^{d/2+1/2}e^{-n\kappamle}\hat{M}_1\left|\text{det}\left(\jhatplus\right)\right|^{1/2}\int_{\|u\|>\deltamle}\|u\|\pi(u+\mle)du}{\left(2\pi \right)^{d/2}\hspace{-1mm}\left( 1-\decayhatplus\right)}.
\end{align*}
A calculation similar to \cref{normal_integral_w1} reveals that
\begin{align*}
&\int_{\|u\|>\deltamle\sqrt{n}}\|u\|\frac{\sqrt{\left|\text{det}\,\fisherinformation(\mle)\right|}e^{-u^T\fisherinformation(\mle)u/2}}{(2\pi)^{d/2}}du 
\leq \left(\deltamle\sqrt{n}+\sqrt{\frac{2\pi}{\minevaluemle}}\right)\decayhat
\end{align*}
	and so
	\begin{align}
		I_{2,1}^{MLE}\leq& \left(\deltamle\sqrt{n}+\sqrt{\frac{2\pi}{\minevaluemle}}\right)\decayhat\notag\\
		&+\frac{ n^{d/2+1/2}e^{-n\kappamle}\hat{M}_1\left|\text{det}\left(\jhatplus\right)\right|^{1/2}\int_{\|u\|>\deltamle}\|u\|\pi(u+\mle)du}{\left(2\pi \right)^{d/2}\hspace{-1mm}\left( 1-\decayhatplus\right)}.\label{i21_mle_w1}
	\end{align}

	Now, using \cref{i221}, we obtain
	\begin{align}
		I_{2,2}^{MLE}\leq &\frac{ \hat{M}_1\widetilde{M}_1\left|\text{det}\left(\jhatplus\right)\right|^{1/2}\left|\text{det}\left(\jhatminus\right)\right|^{-1/2}\sqrt{\Tr\left[\jhatminus^{-1}\right]}}{ 1-\decayhatplus}\notag\\
		&\cdot\left\{\decayhat+\frac{ n^{d/2}e^{-n\kappamle}\hat{M}_1\left|\text{det}\left(\jhatplus\right)\right|^{1/2}}{\left(2\pi \right)^{d/2}\hspace{-1mm}\left( 1-\decayhatplus\right)}\right\}.\label{i22_mle_w1}
	\end{align}

	Adding together bounds from \cref{i21_mle_w1,i22_mle_w1} yields a bound on $I_2^{MLE}$.

\subsubsection{Controlling term $I_1^{MLE}$ using the log-Sobolev inequality and the transportation-entropy inequality}
As in \Cref{subs:controlling_i1_mle}, we shall use the log-Sobolev inequality for the measure $\left[\mathcal{L}\left(\sqrt{n}\left(\postparam-\mle\right)\right)\right]_{B_{0}(\deltamle\sqrt{n})}$. A consequence of the log-Sobolev inequality is that we can apply the transportation-entropy inequality for $\left[\mathcal{L}\left(\sqrt{n}\left(\postparam-\mle\right)\right)\right]_{B_{0}(\deltamle\sqrt{n})}$ (Proposition \ref{prop_talagrand}), which lets us upper bound the 1- and 2-Wasserstein distances by a constant times the KL divergence. The log-Sobolev inequality then upper-bounds the KL divergence in terms of the Fisher divergence. Let $W_2(\cdot,\cdot)$ denote the $2$-Wasserstein distance and $W_1(\cdot,\cdot)$ denote the 1-Wasserstein distance. We have that, for all $1$-Lipschitz test functions $g$,
\begin{align}
	I_1^{MLE}\leq &W_1\left(\left[\sqrt{n}\left(\postparam-\mle\right)\right]_{B_{0}(\deltamle\sqrt{n})},\left[\mathcal{N}(0,\fisherinformation(\mle)^{-1})\right]_{B_{0}(\deltamle\sqrt{n})}\right)\notag\\
	\stackrel{\cref{wasserstein_ineq}}\leq&W_2\left(\left[\sqrt{n}\left(\postparam-\mle\right)\right]_{B_{0}(\deltamle\sqrt{n})},\left[\mathcal{N}(0,\fisherinformation(\mle)^{-1})\right]_{B_{0}(\deltamle\sqrt{n})}\right)\notag\\
	\underset{\text{inequality}}{\stackrel{\text{transportation-}}{\stackrel{\text{entropy}}\leq}}& \sqrt{\frac{2\widetilde{M}_1\hat{M}_1\,\,\text{KL}\left(\left.\left[\mathcal{N}(0,\fisherinformation(\mle)^{-1})\right]_{B_{0}(\deltamle\sqrt{n})}\,\right\|\,\left[\mathcal{L}\left(\sqrt{n}\left(\postparam-\mle\right)\right)\right]_{B_{0}(\deltamle\sqrt{n})}\right)}{\minevaluemle-\deltamle M_2} }\notag\\
	\underset{\text{inequality}}{\stackrel{\text{log-Sobolev}}\leq}&\frac{\widetilde{M}_1\hat{M}_1\sqrt{ \text{Fisher}\left(\left.\left[\mathcal{N}(0,\fisherinformation(\mle)^{-1})\right]_{B_{0}(\deltamle\sqrt{n})}\,\right\|\,\left[\mathcal{L}\left(\sqrt{n}\left(\postparam-\mle\right)\right)\right]_{B_{0}(\deltamle\sqrt{n})}\right)}}{\minevaluemle-\deltamle M_2} \notag\\
	\leq &\frac{\widetilde{M}_1\hat{M}_1}{\minevaluemle-\deltamle M_2}\notag\\
	&\cdot\sqrt{\int_{\|u\|\leq\deltamle\sqrt{n}}\frac{\sqrt{\det \jhat}e^{-u^T\fisherinformation(\mle)u/2}}{F_n^{MLE}(2\pi)^{d/2}}\left\|\fisherinformation(\mle)u+\frac{\loglikelihood'(n^{-1/2}u+\mle)}{\sqrt{n}}\right\|^2du}\notag\\
	&+\frac{\widetilde{M}_1\hat{M}_1}{\sqrt{n}\left(\minevaluemle-\deltamle M_2\right)}\notag\\
	&\hspace{1cm}\cdot\sqrt{\int_{\|u\|\leq\deltamle\sqrt{n}}\frac{\sqrt{\det \jhat}e^{-u^T\fisherinformation(\mle)u/2}}{F_n^{MLE}(2\pi)^{d/2}}\left\|\frac{\prior'(n^{-1/2}u+\mle)}{\prior(n^{-1/2}u+\mle)}\right\|^2du}\notag\\
	\underset{\text{theorem}}{\stackrel{\text{Taylor's}}\leq} &\frac{\widetilde{M}_1\hat{M}_1M_2}{2\sqrt{n}\left(\minevaluemle-\deltamle M_2\right)}\sqrt{\int_{\|u\|\leq\deltamle\sqrt{n}}\frac{\sqrt{\det \jhat}e^{-u^T\fisherinformation(\mle)u/2}}{F_n^{MLE}(2\pi)^{d/2}}\left\|u\right\|^4du}\notag\\
	&+\frac{M_1\widetilde{M}_1\hat{M}_1}{\sqrt{n}\left(\minevaluemle-\deltamle M_2\right)}\notag\\
	\leq &\frac{\sqrt{3}\,\Tr\left[\fisherinformation(\mle)^{-1}\right]\widetilde{M}_1\hat{M}_1\,M_2}{2\left(\minevaluemle-\deltamle M_2\right)\sqrt{n\left( 1- \decayhat\right)}}+\frac{M_1\widetilde{M}_1\hat{M}_1}{\sqrt{n}\left(\minevaluemle-\deltamle M_2\right)}.\label{i1_mle_w1}
\end{align}

\subsubsection{Conclusion} The result now follows from adding together the bounds in \cref{i21_mle_w1,i22_mle_w1,i1_mle_w1}.

\subsection{Proof of \Cref{theorem_2wasserstein}}
\subsubsection{Controlling term $I_2^{MLE}$}
Now we want to control $I_2^{MLE}$ uniformly over all functions $g$ which are of the form $g(u)=\left<v,u\right>$, for some $v\in\mathbbm{R}^d$ with $\|v\|=1$. For such functions we have that $|g(u)|\leq \|u\|^2$. Using the notation of \Cref{introductory_arguments} and \cref{i21}, we have that
\begin{align*}
	I_{2,1}^{MLE}\leq & \int_{\|u\|>\deltamle\sqrt{n}}\|u\|^2\frac{\sqrt{\left|\text{det}\,\fisherinformation(\mle)\right|}e^{-u^T\fisherinformation(\mle)u/2}}{(2\pi)^{d/2}}du\notag\\
	&+\frac{ n^{d/2+1}e^{-n\kappamle}\hat{M}_1\left|\text{det}\left(\jhatplus\right)\right|^{1/2}\int_{\|u\|>\deltamle}\|u\|^2\pi(u+\mle)du}{\left(2\pi \right)^{d/2}\hspace{-1mm}\left( 1- \decayhatplus\hspace{-1mm}\right)}.
\end{align*}
A calculation similar to \cref{normal_integral_w2} reveals that
\begin{align*}
	&\int_{\|u\|>\deltamle\sqrt{n}}\|u\|^2\frac{\sqrt{\left|\text{det}\,\fisherinformation(\mle)\right|}e^{-u^T\fisherinformation(\mle)u/2}}{(2\pi)^{d/2}}du \leq\left(\deltamle^2n+\sqrt{\frac{2\pi}{\minevaluemle}}\right)\decayhat
\end{align*}
and so
\begin{align}
	I_{2,1}^{MLE}\leq& \left(\deltamle^2n+\sqrt{\frac{2\pi}{\minevaluemle}}\right)\decayhat\notag\\
	&+\frac{ n^{d/2+1}e^{-n\kappamle}\hat{M}_1\left|\text{det}\left(\jhatplus\right)\right|^{1/2}\int_{\|u\|>\deltamle}\|u\|^2\pi(u+\mle)du}{\left(2\pi \right)^{d/2}\hspace{-1mm}\left( 1-\decayhatplus\right)}.\label{i21_mle_w2}
\end{align}

	Now, using \cref{i221}, we obtain
\begin{align}
	I_{2,2}^{MLE}\leq &\frac{ \hat{M}_1\widetilde{M}_1\left|\text{det}\left(\jhatplus\right)\right|^{1/2}\left|\text{det}\left(\jhatminus\right)\right|^{-1/2}\Tr\left[\jhatminus^{-1}\right]}{ 1-\decayhatplus}\notag\\
	&\cdot\Bigg\{\decayhat+\frac{ n^{d/2}e^{-n\kappamle}\hat{M}_1\left|\text{det}\left(\jhatplus\right)\right|^{1/2}}{\left(2\pi \right)^{d/2}\hspace{-1mm}\left( 1-\decayhatplus\right)}\Bigg\}.\label{i22_mle_w2}
\end{align}
Adding together bounds from \cref{i21_mle_w2,i22_mle_w2} yields a bound on $I_2^{MLE}$.

\subsubsection{Controlling term $I_1^{MLE}$ using the log-Sobolev inequality and the transportation-entropy inequality}
Note that the calculation in \cref{i1_mle_w1} yields that
\begin{align}
	&W_2\left(\left[\sqrt{n}\left(\postparam-\mle\right)\right]_{B_{0}(\deltamle\sqrt{n})},\left[\mathcal{N}(0,\fisherinformation(\mle)^{-1})\right]_{B_{0}(\deltamle\sqrt{n})}\right)\notag\\
	\leq &\frac{\sqrt{3}\,\Tr\left[\fisherinformation(\mle)^{-1}\right]\widetilde{M}_1\hat{M}_1\,M_2}{2\left(\minevaluemle-\deltamle M_2\right)\sqrt{n\left( 1- \decayhat\right)}}+\frac{M_1\widetilde{M}_1\hat{M}_1}{\sqrt{n}\left(\minevaluemle-\deltamle M_2\right)}.\label{w2_mle_w2}
\end{align}
Therefore, for all functions $g$ which are of the form $g(u)=\left<v,u\right>^2$ for some $v\in\mathbbm{R}^d$ with $\|v\|=1$, an argument similar to the one that led to \cref{i1_map_w2} yields:
\begin{align}
	I_1^{MLE}\leq &W_2\left(\left[\mathcal{L}\left(\sqrt{n}\left(\postparam-\mle\right)\right)\right]_{B_{0}(\deltamle\sqrt{n})},\left[\mathcal{N}(0,\fisherinformation(\mle)^{-1})\right]_{B_{0}(\deltamle\sqrt{n})}\right)^2\notag\\
	&+ \frac{2\sqrt{\Tr\left[\fisherinformation(\mle)^{-1}\right]}}{\sqrt{ 1- \decayhat}}W_2\left(\left[\mathcal{L}\left(\sqrt{n}\left(\postparam-\mle\right)\right)\right]_{B_{0}(\deltamle\sqrt{n})},\left[\mathcal{N}(0,\fisherinformation(\mle)^{-1})\right]_{B_{0}(\deltamle\sqrt{n})}\right),\label{i1_mle_w2}
\end{align}
and the final bound on $I_1^{MLE}$ follows from \cref{w2_mle_w2}.
\subsubsection{Conclusion}
The result now follows from combining \cref{w2_mle_w2,i1_mle_w2} and adding together with \cref{i21_mle_w2,i22_mle_w2}.

	\section{Proof of Theorem \ref{theorem_univariate}}\label{appendix_d}
	In this section, we concentrate on the univariate context (i.e. on $d=1$). We shall apply Stein's method, in the framework described in \citet[Section 2.1]{ernst_reinert_swan}. Before we do that, however, let us recall that we want to upper-bound the quantity $D_g^{MLE}$ given by \cref{dgmle} for all functions $g$ for which the two expectations in \cref{dgmle} exist.
	Recall the definition of  $C_n^{MLE}$ from \cref{c_n_mle} and let:
	\begin{align}\label{h_def}
		h(t)=h_g^{MLE}(t)=g(t)-\frac{n^{-1/2}}{C_n}\int_{-\deltamle\sqrt{n}}^{\deltamle\sqrt{n}}g(u)\posterior(n^{-1/2}u+\mle)du.
	\end{align}
	We can repeat the calculation leading to \cref{i1i2mle}, without dividing the first term after the first inequality by $F_n^{MLE}$. We then obtain:
	\begin{align*}
		D_g^{MLE}\leq &\left|\int_{-\deltamle\sqrt{n}}^{\deltamle\sqrt{n}}h(u)\frac{e^{-t^2/(2\sigma_n^2)}}{\sqrt{2\pi\sigma_n^2}}\right|+\left|\int_{-\deltamle\sqrt{n}}^{\deltamle\sqrt{n}} h(u)\left[\frac{e^{-u^2/(2\sigma_n^2)}}{\sqrt{2\pi\sigma_n^2}}-n^{-1/2}\posterior(n^{-1/2}u+\mle)\right]    \right|\\
		&\hspace{11cm}=:\tilde{I}_1+\tilde{I}_2.
	\end{align*}
	We will bound $\tilde{I}_1$ and $\tilde{I}_2$ separately.
	\subsection{Controlling term $\tilde{I}_2$}
	Note that $\tilde{I}_2$ is the same as $I_2^{MLE}$ defined by \cref{i1i2mle}, for $d=1$. We will use the calculations leading to \cref{i21,i22}. Instead of using Lemma \ref{lemma1}, we will, however, apply the standard one-dimensional Gaussian concentration inequality, which says that, for $Z_n\sim\mathcal{N}\left(0,\sigma_n^2\right)$,
	\begin{align*}
		\mathbbm{P}\left[   |Z_n|>\deltamle\sqrt{n} \right]\leq 2e^{-\deltamle^2n/(2\sigma_n^2)}.
	\end{align*}
We obtain
	 \begin{align}
		\tilde{I}_2\leq& \left|\int_{|u|>\deltamle\sqrt{n}}g(u)\frac{e^{-u^2/(2\sigma_n^2)}}{\sqrt{2\pi\sigma_n^2}}du\right|\notag\\
		&+\frac{ n^{1/2}e^{-n\kappamle}\hat{M}_1\left(\frac{1}{\sigma_n^2}+\frac{\deltamle M_2}{3}\right)^{1/2}\int_{
				|u|>\deltamle}|g(u\sqrt{n})|\pi(u+\mle)du}{\sqrt{2\pi} \left\{1-2\exp\left[-\frac{1}{2}\left(\frac{1}{\sigma_n^2}+\frac{M_2\deltamle}{3}\right)\deltamle^2n\right]\right\}}\notag\\
		&+\frac{ \hat{M}_1\widetilde{M}_1\left(\frac{1}{\sigma_n^2}+\frac{\deltamle M_2}{3}\right)^{1/2}\int_{|t|\leq\deltamle\sqrt{n}}|g(t)|e^{-\frac{1}{2}\left(\frac{1}{\sigma_n^2}-\frac{M_2\deltamle}{3}\right)t^2}dt}{\sqrt{2\pi}\, \hspace{-1mm}\left\{ 1-2\exp\left[-\frac{1}{2}\left(\frac{1}{\sigma_n^2}+\frac{M_2\deltamle}{3}\right)\deltamle^2n\right]\right\}}\notag\\
		&\hspace{3cm}\cdot\Bigg\{2e^{-\deltamle^2n/(2\sigma_n^2)}+\frac{ n^{1/2}e^{-n\kappamle}\hat{M}_1\left(\frac{1}{\sigma_n^2}+\frac{\deltamle M_2}{3}\right)^{1/2}}{\sqrt{2\pi}\left\{1-2\exp\left[-\frac{1}{2}\left(\frac{1}{\sigma_n^2}+\frac{M_2\deltamle}{3}\right)\deltamle^2n\right]\right\}}\Bigg\}.\label{i2_univariate_final}
	\end{align}

	\subsection{Controlling term $\tilde{I}_1$ using Stein's method}
	In this section we will use Stein's method in the framework of \citet[Section 2.1]{ernst_reinert_swan}. Note that, by integration by parts, for all continuous functions $f:\left[-\deltamle\sqrt{n},\deltamle\sqrt{n}\right]\to\mathbbm{R}$ which are differentiable on $\left(-\deltamle\sqrt{n},\deltamle\sqrt{n}\right)$ and satisfy $f(-\deltamle\sqrt{n})=f(\deltamle\sqrt{n})$, we have
	\begin{align}
		&\int_{-\deltamle\sqrt{n}}^{\deltamle\sqrt{n}}f'(t)\frac{e^{-t^2/(2\sigma_n^2)}}{\sqrt{2\pi\sigma_n^2}} dt=\int_{-\deltamle\sqrt{n}}^{\deltamle\sqrt{n}}tf(t)\frac{e^{-t^2/(2\sigma_n^2)}}{\sigma_n^2\sqrt{2\pi\sigma_n^2}}dt\label{stein_identity}
	\end{align}

	Now, for our function $h$, given by \cref{h_def}, let
	\begin{align}\label{f_definition}
	f(t):=\begin{cases}
		\frac{1}{\posterior(n^{-1/2} t + \mle)}\int_{-\deltamle\sqrt{n}}^{t} h(u)\posterior(n^{-1/2} u+\mle)du,& \text{if }t\in\left(-\deltamle\sqrt{n},\deltamle\sqrt{n}\right)\\
		0,&\text{otherwise}.
	\end{cases}
	\end{align}
	Note that $f$ is continuous on $[-\deltamle\sqrt{n},\deltamle\sqrt{n}]$, differentiable on $(-\deltamle\sqrt{n},\deltamle\sqrt{n})$ and $f(-\deltamle\sqrt{n})=f(\deltamle\sqrt{n})=0$. Moreover, on $(-\deltamle\sqrt{n},\deltamle\sqrt{n})$, $f$ solves the Stein equation associated to the distribution of $\sqrt{n}\left(\postparam-\mle\right)$ for test function $h$, as described in \citet[Section 2.1]{ernst_reinert_swan}. In other words,
	\begin{align}
		h(t)=f'(t)+f(t)\left(\dv t \log\posterior(n^{-1/2}t+\mle)\right),\qquad t\in\left(-\deltamle\sqrt{n},\deltamle\sqrt{n}\right).\label{sol_stein_eq}
	\end{align}
Now, by Taylor's theorem, we obtain that, for some $c\in(0,1)$,
	\begin{align*}
		&\tilde{I}_1\\
		\stackrel{\cref{sol_stein_eq}}=&\left|\int_{-\deltamle/\sqrt{n}}^{\deltamle\sqrt{n}}\left[f'(t)+f(t)\left(\dv t \loglikelihood(n^{-1/2} t+\mle) + \dv t \log\prior(n^{-1/2} t+\mle)\right)\right]\frac{e^{-t^2/(2\sigma_n^2)}}{\sqrt{2\pi\sigma_n^2}}dt\right|\\
		\leq&\left|\int_{-\deltamle\sqrt{n}}^{\deltamle\sqrt{n}}\left[f'(t)+f(t)\left(\frac{\loglikelihood'(\mle)}{\sqrt{n}}+\frac{t \loglikelihood''(\mle)}{n} +\frac{t^2 }{2n^{3/2}}\loglikelihood'''(\mle+cn^{-1/2}t)\right)\right]\frac{e^{-t^2/(2\sigma_n^2)}}{\sqrt{2\pi\sigma_n^2}}dt\right|\\
		&+\frac{1}{\sqrt{n}}\left|\int_{-\deltamle\sqrt{n}}^{\deltamle\sqrt{n}}f(t)\frac{\prior'(n^{-1/2}t+\mle)}{\prior(n^{-1/2}t+\mle)}\cdot\frac{e^{-t^2/(2\sigma_n^2)}}{\sqrt{2\pi\sigma_n^2}}dt\right|\\
		=&\left|\int_{-\deltamle\sqrt{n}}^{\deltamle\sqrt{n}}\left[f'(t)+f(t)\left(-\frac{t}{\sigma_n^2} +\frac{t^2 }{2n^{3/2}}\loglikelihood'''(\mle+cn^{-1/2}t)\right)\right]\frac{e^{-t^2/(2\sigma_n^2)}}{\sqrt{2\pi\sigma_n^2}}dt\right|\\
		&+\frac{1}{\sqrt{n}}\left|\int_{-\deltamle\sqrt{n}}^{\deltamle\sqrt{n}}f(t)\frac{\prior'(n^{-1/2}t+\mle)}{\prior(n^{-1/2}t+\mle)}\cdot\frac{e^{-t^2/(2\sigma_n^2)}}{\sqrt{2\pi\sigma_n^2}}dt\right|\\
		\leq &\left|\int_{-\deltamle\sqrt{n}}^{\deltamle\sqrt{n}}\left[f'(t)-\frac{tf(t)}{\sigma_n^2}\right]\frac{e^{-t^2/(2\sigma_n^2)}}{\sqrt{2\pi\sigma_n^2}}dt\right|+\frac{M_2}{2\sqrt{n}}\int_{-\deltamle\sqrt{n}}^{\deltamle\sqrt{n}}|t^2f(t)|\,\frac{e^{-t^2/(2\sigma_n^2)}}{\sqrt{2\pi\sigma_n^2}}dt\\
		&+\frac{M_1}{\sqrt{n}}\int_{-\deltamle\sqrt{n}}^{\deltamle\sqrt{n}}|f(t)|\frac{e^{-t^2/(2\sigma_n^2)}}{\sqrt{2\pi\sigma_n^2}}dt\\
		\stackrel{\cref{stein_identity}}=&\frac{M_2}{2\sqrt{n}}\int_{-\deltamle\sqrt{n}}^{\deltamle\sqrt{n}}|t^2f(t)|\,\frac{e^{-t^2/(2\sigma_n^2)}}{\sqrt{2\pi\sigma_n^2}}dt+\frac{M_1}{\sqrt{n}}\int_{-\deltamle\sqrt{n}}^{\deltamle\sqrt{n}}|f(t)|\frac{e^{-t^2/(2\sigma_n^2)}}{\sqrt{2\pi\sigma_n^2}}dt\\
		\stackrel{\cref{f_definition}}=&\frac{M_2}{2\sqrt{2\pi\sigma_n^2\,n}}\int_{-\deltamle\sqrt{n}}^{\deltamle\sqrt{n}}\frac{ t^2e^{-t^2/(2\sigma_n^2)}}{\posterior(n^{-1/2} t+\mle)}\left|\int_{-\deltamle\sqrt{n}}^t h(u)\posterior(n^{-1/2} u+\mle)du\right|dt\\
		&+\frac{M_1}{\sqrt{2\pi\sigma_n^2\,n}}\int_{-\deltamle\sqrt{n}}^{\deltamle\sqrt{n}}\frac{ e^{-t^2/(2\sigma_n^2)}}{\posterior(n^{-1/2} t+\mle)}\left|\int_{-\deltamle\sqrt{n}}^t h(u)\posterior(n^{-1/2} u+\mle)du\right|dt\\
		= &\frac{M_2}{2\sqrt{2\pi\sigma_n^2\,n}}\int_{-\deltamle\sqrt{n}}^{0}\frac{ t^2e^{-t^2/(2\sigma_n^2)}}{\posterior(n^{-1/2} t+\mle)}\left|\int_{-\deltamle\sqrt{n}}^t h(u)\posterior(n^{-1/2} u+\mle)du\right|dt\\
		&+\frac{M_2}{2\sqrt{2\pi\sigma_n^2\,n}}\int_0^{\deltamle\sqrt{n}}\frac{ t^2e^{-t^2/(2\sigma_n^2)}}{\posterior(n^{-1/2} t+\mle)}\left|\int_t^{\deltamle\sqrt{n}} h(u)\posterior(n^{-1/2} u+\mle)du\right|dt\\
		&+\frac{M_1}{\sqrt{2\pi\sigma_n^2\,n}}\int_{-\deltamle\sqrt{n}}^{0}\frac{ e^{-t^2/(2\sigma_n^2)}}{\posterior(n^{-1/2} t+\mle)}\left|\int_{-\deltamle\sqrt{n}}^t h(u)\posterior(n^{-1/2} u+\mle)du\right|dt\\
		&+\frac{M_1}{\sqrt{2\pi\sigma_n^2\,n}}\int_0^{\deltamle\sqrt{n}}\frac{ e^{-t^2/(2\sigma_n^2)}}{\posterior(n^{-1/2} t+\mle)}\left|\int_t^{\deltamle\sqrt{n}} h(u)\posterior(n^{-1/2} u+\mle)du\right|dt\\
		=:&\tilde{I}_{1,1}+\tilde{I}_{1,2}+\tilde{I}_{1,3}+\tilde{I}_{1,4}.
	\end{align*}
	 Now, note that, for some $c_1,c_2\in (0,1)$,
	\begin{align*}
		\tilde{I}_{1,1}\leq& \frac{M_2}{2\sqrt{2\pi\sigma_n^2\,n}}\int_{-\deltamle\sqrt{n}}^{0}\frac{ t^2e^{-t^2/(2\sigma_n^2)}}{\posterior(n^{-1/2} t+\mle)}\int_{-\deltamle\sqrt{n}}^t |h(u)|\posterior(n^{-1/2} u+\mle)\,du\,dt\\
		= & \frac{M_2}{2\sqrt{2\pi\sigma_n^2\,n}}\int_{-\deltamle\sqrt{n}}^{0}|h(u)|\int_{u}^0t^2e^{-t^2/(2\sigma_n^2)}\frac{\posterior(n^{-1/2} u+\mle)}{\posterior(n^{-1/2} t+\mle)}\,dt\, du\\
		\leq&\frac{\widetilde{M}_1\hat{M}_1M_2}{2\sqrt{2\pi\sigma_n^2\,n}}\int_{-\deltamle\sqrt{n}}^{0}|h(u)|\int_{u}^0t^2e^{-t^2/(2\sigma_n^2)}\exp\bigg[L_n(\mle)+\frac{u}{\sqrt{n}} L_n'(\mle) + \frac{u^2}{2n} L_n''(\mle) \\
		&\hspace{8cm}+ \frac{u^3}{6n^{3/2}}L_n'''(\mle+c_1n^{-1/2}u)\bigg]\\
		&\cdot\exp\bigg[-L_n(\mle)-\frac{t}{\sqrt{n}} L_n'(\mle) - \frac{t^2}{2n} L_n''(\mle)- \frac{t^3}{6n^{3/2}}L_n'''(\mle+c_2n^{-1/2}t)\bigg]dt\,du\\
		\leq &\frac{\widetilde{M}_1\hat{M}_1M_2}{2\sqrt{2\pi\sigma_n^2n}}\int_{-\deltamle\sqrt{n}}^{0}|h(u)|e^{-u^2/(2\sigma_n^2)}e^{\deltamle M_2 u^2/6}\int_{u}^0t^2e^{\deltamle M_2 t^2/6}\,dt\,du\\
		\leq &\frac{3\widetilde{M}_1\hat{M}_1}{\deltamle\sqrt{2\pi\sigma_n^2\,n}}\int_{-\deltamle\sqrt{n}}^{0}|uh(u)|\left(e^{-\left(\frac{1}{2\sigma_n^2}-\frac{\deltamle M_2}{3}\right)u^2}-e^{-\left(\frac{1}{2\sigma_n^2}-\frac{\deltamle M_2}{6}\right)u^2}\right)du.
	\end{align*}
	By a similar argument,
	\begin{align*}
		&\tilde{I}_{1,2}\leq \frac{3\widetilde{M}_1\hat{M}_1}{\deltamle\sqrt{2\pi\sigma_n^2\,n}}\int_0^{\deltamle\sqrt{n}}|uh(u)|\left(e^{-\left(\frac{1}{2\sigma_n^2}-\frac{\deltamle M_2}{3}\right)u^2}-e^{-\left(\frac{1}{2\sigma_n^2}-\frac{\deltamle M_2}{6}\right)u^2}\right)du.\\
		&\tilde{I}_{1,3}\leq \frac{\widetilde{M}_1\hat{M}_1M_1}{\sqrt{2\pi\sigma_n^2\,n}}\int_{-\deltamle\sqrt{n}}^{0}|uh(u)|e^{-\left(\frac{1}{2\sigma_n^2}-\frac{\deltamle M_2}{3}\right)u^2}du.\\
		&\tilde{I}_{1,4}\leq \frac{\widetilde{M}_1\hat{M}_1M_1}{\sqrt{2\pi\sigma_n^2\,n}}\int_0^{\deltamle\sqrt{n}}|uh(u)|e^{-\left(\frac{1}{2\sigma_n^2}-\frac{\deltamle M_2}{3}\right)u^2}du.
	\end{align*}
	Therefore,
	\begin{align}
		\tilde{I}_1\leq& \frac{2\widetilde{M}_1\hat{M}_1\left(M_1+\frac{3}{\deltamle}\right)}{\sqrt{2\pi\sigma_n^2\,n}}\int_{-\deltamle\sqrt{n}}^{\deltamle\sqrt{n}}|uh(u)|e^{-\left(\frac{1}{2\sigma_n^2}-\frac{\deltamle M_2}{3}\right)u^2}du\notag\\
		&-\frac{6\widetilde{M}_1\hat{M}_1}{\deltamle\sqrt{2\pi\sigma_n^2\,n}}\int_{-\deltamle\sqrt{n}}^{\deltamle\sqrt{n}}|uh(u)|e^{-\left(\frac{1}{2\sigma_n^2}-\frac{\deltamle M_2}{6}\right)u^2}du.\label{i1_univariate}
	\end{align}
	Now, by Taylor's expansion:
	\begin{align}
		\frac{n^{-1/2}}{C_n^{MLE}}\left|\int_{-\deltamle\sqrt{n}}^{\deltamle\sqrt{n}}g(t)\posterior(n^{-1/2} t+\mle)dt\right|\leq&\frac{\int_{-\deltamle\sqrt{n}}^{\deltamle\sqrt{n}}|g(t)|\posterior(n^{-1/2} t+\mle)dt}{\int_{-\deltamle\sqrt{n}}^{\deltamle\sqrt{n}}\posterior(n^{-1/2} u+\mle)du}\notag\\
		\leq& \widetilde{M}_1\hat{M}_1\frac{\int_{-\deltamle\sqrt{n}}^{\deltamle\sqrt{n}}|g(t)|e^{-(1/(2\sigma_n^2)-\deltamle M_2/6)t^2}dt}{\int_{-\deltamle\sqrt{n}}^{\deltamle\sqrt{n}}e^{-(1/(2\sigma_n^2)+\deltamle M_2/6)u^2}du}\notag\\
	&\hspace{-1cm}	\leq\widetilde{M}_1\hat{M}_1\frac{\sqrt{\frac{1}{2\sigma_n^2}+\frac{\deltamle M_2}{6}}\int_{-\deltamle\sqrt{n}}^{\deltamle\sqrt{n}}|g(t)|e^{-(1/(2\sigma_n^2)-\deltamle M_2/6)t^2}dt}{\sqrt{2\pi}\left(1-2e^{-\deltamle^2n(1/(2\sigma_n^2)+\deltamle M_2/6)}\right)}.\label{i11_univariate}
	\end{align}
\Cref{i1_univariate,i11_univariate}, together with a standard expression for the normal first absolute moment now yield that
	\begin{align}
		\tilde{I}_1\leq& \frac{2\widetilde{M}_1\hat{M}_1}{\sqrt{2\pi\sigma_n^2\,n}}\int_{-\deltamle\sqrt{n}}^{\deltamle\sqrt{n}}|ug(u)|\left[\left(M_1+\frac{3}{\deltamle}\right)e^{-\left(\frac{1}{2\sigma_n^2}-\frac{\deltamle M_2}{3}\right)u^2}-\frac{3}{\deltamle}e^{-\left(\frac{1}{2\sigma_n^2}-\frac{\deltamle M_2}{6}\right)u^2}\right]du\notag\\
		&+ \frac{2\sqrt{\frac{1}{2\sigma_n^2}+\frac{\deltamle M_2}{6}}\left(\widetilde{M}_1\hat{M}_1\right)^2\left(M_1+\frac{3}{\deltamle}\right)\int_{-\deltamle\sqrt{n}}^{\deltamle\sqrt{n}}|g(u)|e^{-(1/(2\sigma_n^2)-\deltamle M_2/6)u^2}du}{\left(\frac{1}{2\sigma_n^2}-\frac{\deltamle M_2}{3}\right)\pi\sqrt{\sigma_n^2}\left(1-2e^{-\deltamle^2n(1/(2\sigma_n^2)+\deltamle M_2/6)}\right)\sqrt{n}}\notag\\
		&\hspace{7cm}\cdot\left(\frac{M_1+\frac{3}{\deltamle}}{\frac{1}{2\sigma_n^2}-\frac{\deltamle M_2}{3}}-\frac{3}{\deltamle\left(\frac{1}{2\sigma_n^2}-\frac{\deltamle M_2}{6}\right)}\right).\label{i1_univariate_final}
	\end{align}

	\subsection{Conclusion} The final bound now follows from \cref{i2_univariate_final,i1_univariate_final}.

\section{More detail on the examples}\label{appendix_examples}

\subsection{More detail on \Cref{example_1}}\label{app_logistic_example}
Note that the function $\loglikelihood$ is concave, so its global maximum (i.e. the MLE) $\mle$ exists as long as the data are not linearly separable. The log-posterior is given by
\begin{align*}
	\logposterior(\param)=-\sum_{i=1}^n\rho\left(Y_iX_i^T\param\right)-\frac{\|\param\|^2}{2}-\frac{d}{2}\log(2\pi).
\end{align*}
Note that $\logposterior$ is also concave so its global maximum (i.e. the MAP) $\map$ exists as long as the data are not linearly separable.
\sloppy Note also that $\loglikelihood''(\param)=-\sum_{i=1}^n\rho''(Y_iX_i^T\param)X_iX_i^T$ and $\left(\log \pi\right)''(\param)=-I_{d\times d}$.


\subsubsection{Deriving the order of our total variation distance bound}\label{example_bound}
Consider the following two results:
\begin{lemma}[cf. {\citealt[Lemma 4]{sur}}]\label{lemma_bound_eigenvalue}
	Suppose that $d/n<1$. Let $H(\epsilon)=-\epsilon\log \epsilon - (1-\epsilon)\log(1-\epsilon).$ Then there exists a constant $\epsilon_0$ such that for all $0\leq \epsilon\leq \epsilon_0$, with probability at least $1-2\exp(-nH(\epsilon))-2\exp(-n/2)$, the following matrix inequality
	\begin{align*}
		-\frac{\loglikelihood''(\param)}{n}\succeq \frac{\exp(3\|\param\|/\sqrt{\epsilon})}{\left(1+\exp(3\|\param\|/\sqrt{\epsilon})\right)^2}\left(\sqrt{1-\epsilon} - \frac{d}{n}-2\sqrt{\frac{H(\epsilon)}{1-\epsilon}}\right)I_{d\times d}
	\end{align*}
	holds simultaneously for all $\param\in\mathbb{R}^d$.
\end{lemma}
\begin{proof}
	This follows directly from \citet[Lemma 4]{sur}, using the fact that
	\begin{align*}
		\inf_{z:|z|\leq \frac{3\|\param\|}{\sqrt{\epsilon}}} \rho''(z) = \frac{\exp(3\|\param\|/\sqrt{\epsilon})}{\left(1+\exp(3\|\param\|/\sqrt{\epsilon})\right)^2}.
	\end{align*}
\end{proof}
\begin{lemma}[cf. {\citealt[Corollary 6.4]{katsevich_bvm}}]\label{lemma_bounded_mle}
	Fix any small constant $\epsilon>0$. Then, there exist universal constants $c_1,c_2,C_2>0$, independent of $n$ or $d$ such that if $d^{3/2}\leq n$ then, for large enough $n$, the MLE $\mle$ associated to $\loglikelihood$ satisfies
	\begin{align*}
		\|\mle\|<c_1,
	\end{align*}
	with probability at least $1-\exp(-c_2(nd)^{1/8})-e^{-C_2n^{1/6}}$.
\end{lemma}
\begin{proof}
    This follows directly from \citet[Corollary 6.4]{katsevich_bvm} upon choosing $s= c_3n^{1/6}$ in their setup, for a sufficiently small absolute constant $c_3>0$ not depending on $n$ or $d$.
\end{proof}
Note that functions $\frac{\loglikelihood}{n}$ and $\frac{\logposterior}{n}$ are continuously differentiable and strictly concave and converge almost surely to the same (concave) limit as $n\to\infty$. Therefore, if $\|\mle\|$ is uniformly  bounded, then so is $\|\map\|$. Therefore, by Lemma \ref{lemma_bounded_mle}, if $d^{3/2}\leq n$ then $\|\map\|$ is upper-bounded, uniformly in $n$ and $d$, with high probability. As $\logposterior''(\param)=\loglikelihood''(\param)-I_{d\times d}$, we can now apply Lemma \ref{lemma_bound_eigenvalue} to conclude that, when $d^{3/2}\leq n$, then $\minevaluemap$ is lower bounded uniformly in $d$ and $n$, with high probability.

Now, consider the following result, which is a direct consequence of \citet[Lemma 3.2]{katsevich}:
\begin{lemma}[cf. {\citealt[Lemma 3.2]{katsevich}}]\label{lemma_bounded_m2}
    Suppose that $d\leq n\leq e^{\sqrt{d}}$. Then, there exist absolute constants $B_1,B_2,C>0$, such that
    \begin{align*}
        \sup_{\param\in\mathbb{R}^d}\left\|\frac{\logposterior'''(\param)}{n}\right\|\leq C\left(1+\frac{d^{3/2}}{n}\right);\qquad \sup_{\param\in\mathbb{R}^d}\left\|\frac{\loglikelihood'''(\param)}{n}\right\|\leq C\left(1+\frac{d^{3/2}}{n}\right)
    \end{align*}
    with probability at least $1-B_1\exp(-B_2\sqrt{nd}/\log(2n/d))$.
\end{lemma}

As discussed in \Cref{dimension_dependence} and in \Cref{app_logistic_example_assumptions} below, as long as $\frac{d\log n}{n}\xrightarrow{n\to\infty}0$, the leading term in in our bound on the total variation distance from \Cref{main_map} is of order $C_d\sqrt{d^2/n}$, where $C_d\leq \frac{\overline{M}_2}{\minevaluemap\sqrt{\minevaluemap-\deltabar\overline{M}_2}}$. As described above, we have established via Lemmas \ref{lemma_bound_eigenvalue} and \ref{lemma_bounded_mle} that $\minevaluemap$ is lower-bounded with high probability if $n\geq d^{3/2}$. We have also established via Lemma \ref{lemma_bounded_m2} that $\overline{M}_2$ is upper-bounded with high probability. Moreover, we can choose $\deltabar>0$ to be independent of $n$ and $d$ and arbitrarily small. Therefore $C_d$ is with high probability upper-bounded by a finite constant not depending on $d$ or $n$.
It follows that the leading order term in our bound on the total variation distance from \Cref{main_map} is with high probability upper bounded by a universal constant multiplied by $\sqrt{d^2/n}$. 

\subsubsection{Lower bound on the effective dimension}\label{lower_bound_eff_dim}
Now, we shall show that, in the setup of \Cref{example_1},
\begin{align*}
d_{eff}:=&\Tr{\left(\jbar+\frac{(\log\pi)''(\map)}{n}\right)\jbar^{-1}}\geq d\left(1- \frac{1}{n \,\minevaluemap}\right),
    \end{align*}
where $d_{eff}$ denotes the \textit{effective dimension} introduced in \citet{spokoiny_laplace}. Note that
\begin{align*}
    d_{eff}=&\Tr{\left(\jbar-\frac{1}{n}I_{d\times d}\right)\jbar^{-1}}\\
    =&\Tr{I_{d\times d}-\frac{1}{n}\jbar^{-1}}\\
    =& d-\frac{1}{n}\Tr{\jbar^{-1}}\\
    \geq & d\left(1- \frac{1}{n \,\minevaluemap}\right).
\end{align*}
The fact that $\minevaluemap$ is lower-bounded by a positive number not depending on $n$ of $d$ with high probability follows from Lemmas \ref{lemma_bound_eigenvalue} and \ref{lemma_bounded_mle} presented above and is discussed directly below Lemma \ref{lemma_bounded_mle} in \Cref{example_bound} above.

\subsubsection{Checking that the model satisfies the assumptions of \Cref{main_map}}\label{app_logistic_example_assumptions}
We retain the assumption that $2d<n < e^{\sqrt{d}}$ and assume that $\frac{d \log n}{n}\xrightarrow{n\to\infty}0$. As discussed above, the MLE $\mle$ and MAP $\map$ exist and are bounded with high probability by Lemma \ref{lemma_bounded_mle}. \Cref{assump1} is satisfied with high probability, (for instance with $\delta=1$), with $M_2$ not depending on $n$ or $d$, by Lemma \ref{lemma_bounded_m2}. The Gaussian prior trivially satisfies \Cref{assump_prior1} for any $\deltamle\leq 1$. Moreover, $\hat{M}_1$ is with high probability bounded by a constant multiplied by $(2\pi)^{d/2}$, as $\|\mle\|$ is upper bounded by a constant independent of $n$ and $d$ with high probability. By Lemma \ref{lemma_bounded_m2}, with high probability \Cref{assump2} is also satisfied for any fixed $\deltabar$ not depending on $n$ or $d$ and with $\overline{M}_2$ not depending on $n$ or $d$.

As discussed above, $\frac{\logposterior}{n}$ and $\frac{\loglikelihood}{n}$ are continuously differentiable and concave and converge almost surely to the same (concave) limit as $n\to\infty$. By Lemma \ref{lemma_bounded_mle}, $\mle$ and $\map$ are bounded with high probability. Therefore, with high probability, $\|\map-\mle\|<\deltabar$, for any fixed $\deltabar>0$ and for large enough $n$. Moreover, with high probability, $\minevaluemap$ and $\minevaluemle$ are lower bounded by a positive number not depending on $n$ or $d$, by Lemmas \ref{lemma_bound_eigenvalue} and \ref{lemma_bounded_mle}. Moreover $\frac{\Tr\left[\postfisherinformation(\map)^{-1}\right]}{n}\leq \frac{d}{n\,\minevaluemap}$ and $\frac{\Tr\left[\fisherinformation(\mle)^{-1}\right]}{n}\leq \frac{d}{n\,\minevaluemle}$. Therefore, with high probability, \Cref{assump_size_of_delta_bar} is satisfied, for large enough $n$ and for any fixed $\deltamle$ and $\deltabar$, not depending on $n$ or $d$.

 Recall that $\minevaluemap$ is lower bounded by a positive number not depending on $n$ or $d$ with high probability, if $d<\frac{n}{2}$. Therefore, \Cref{assump7} is satisfied for small enough $\deltabar$.  Assume $n$ is large enough so that $\deltabar-\|\mle-\map\|>0$. \Cref{assump_kappa1} is satisfied with high probability by the following reasoning, which uses the strict concavity of $\loglikelihood$:
 	\begin{align*}
 	&\sup_{\param:\|\param-\mle\|>\deltabar-\|\mle-\map\|}\frac{L_n(\param)-L_n(\mle)}{n}\leq\sup_{\param:\|\param-\mle\|=\deltabar-\|\mle-\map\|}\frac{L_n(\param)-L_n(\mle)}{n}\\
 		\leq
 	&\sup_{\param:\|\param-\mle\|=\deltabar-\|\mle-\map\|}\left\{-\frac{1}{2}\left(\param-\mle\right)^T\fisherinformation(\mle)\left(\param-\mle\right)\right\}+\frac{M_2\left(\deltabar-\|\mle-\map\|\right)^3}{6}\\
 	\leq& -\frac{1}{2}\minevaluemle\left(\deltabar-\|\mle-\map\|\right)^2+\frac{M_2\left(\deltabar-\|\mle-\map\|\right)^3}{6}=:-\kappabar.
 \end{align*}
With high probability, $\kappabar$ is lower bounded by a positive number not depending on $n$ or $d$, for any fixed $\deltabar$ independent of $n$ and $d$, as long as $n$ is large enough.

Now, we have shown that we can make choices of $\deltabar$ and $\kappabar$ independent of $n$ and $d$, such that with high probability the assumptions of \Cref{main_map} hold for large enough $n$. For such a choice of $\deltabar$ we have that $\deltabar\gg \frac{\sqrt{\log n}}{\sqrt{n\minevaluemap}}$. By our assumption $\frac{d \log n}{n}\xrightarrow{n\to\infty}0$, we also have that $\kappabar\gg \frac{\log n}{n}\cdot \frac{d+1}{2}$. Moreover, we have established that $\hat{M}_1$ is with high probability bounded by a universal constant multiplied by $(2\pi)^{d/2}$. Also, note that
\begin{align*}
	\det{\jhatplus}=\det{\frac{\sum_{i=1}^n\rho''(Y_iX_i^T\mle)X_iX_i^T}{n}+(\delta M_2/3)I_{d\times d}}
\end{align*}
Note that
\begin{align*}
	\left\|\frac{\sum_{i=1}^n\rho''(Y_iX_i^T\mle)X_iX_i^T}{n}\right\|_{op}
	\leq& \frac{1}{4n}\sum_{i=1}^n\left\|X_iX_i^T\right\|_{op}
	\leq\frac{1}{4n}\sum_{i=1}^n\left\|X_i\right\|^2.
\end{align*}
Therefore,
\begin{align*}
	\log \left[\det{\jhatplus}^{1/2}\right]\leq \frac{d}{2}\log\left[\frac{1}{4n}\sum_{i=1}^n\left\|X_i\right\|^2+\delta M_2/3\right].
\end{align*}
It follows that $\kappabar\gg \frac{1}{n}\log\left(\hat{M}_1\det{\jhatplus}^{1/2}\right)$ with high probability.
Hence, by the discussion of \Cref{dimension_dependence}, with high probability, the bound in \Cref{main_map} is of the same order as that of the first summand $A_1n^{-1/2}$.

	\subsection{Calculations for \Cref{poisson_example}}\label{calculations_poisson}
	\subsubsection{The MLE-centric approach}
Let $X_1,\dots,X_n\geq 0$ be our data and assume that their sum is positive. We have
\[
\loglikelihood(\param)=-n\param + (\log\param)\left(\sum_{i=1}^n X_i\right)-\sum_{i=1}^n\log\left(X_i!\right)
\]
The MLE is given by $\mle=\overline{X}_n:=\frac{1}{n}\sum_{i=1}^n X_i$.  Also,
\[
\loglikelihood'(\param)=-n+\frac{n\overline{X}_n}{\param},\quad \loglikelihood''(\param)=-\frac{n\overline{X}_n}{\param^2},\quad \loglikelihood'''(\param)=\frac{2n\overline{X}_n}{\param^3}.
\]
We have $\sigma_n^2:=\fisherinformation(\mle)^{-1}=\mle^2/\left.\overline{X}_n\right.=\left|\overline{X}_n\right|$.

Now, for $c\in(0,1)$ and $\deltamle=c\mle$, we have that for $\param\in(\mle-\deltamle,\mle+\deltamle)=\left(\overline{X}_n-c\overline{X}_n,\overline{X}_n+c\overline{X}_n\right)$,
\[
\frac{\left|L_n'''(\param)\right|}{n}=\frac{2\overline{X}_n}{\left|\param\right|^3}\leq \frac{2}{(1-c)^3\left(\overline{X}_n\right)^2}=:M_2
\]
Moreover, for $\theta$, such that $\left|\param-\mle\right|>\deltamle=c\overline{X}_n$, i.e. for $\param>\overline{X}_n+c\overline{X}_n$ or $\param<\overline{X}_n-c\overline{X}_n$,
\begin{align*}
	\frac{\loglikelihood(\param)-\loglikelihood(\mle)}{n}\leq &\max\left\{ \frac{\loglikelihood((1+c)\overline{X}_n)-\loglikelihood(\overline{X}_n)}{n},   \frac{\loglikelihood((1-c)\overline{X}_n)-\loglikelihood(\overline{X}_n)}{n}  \right\}\\
	\leq & \overline{X}_n\cdot\max\left\{ \log(1+c)-c, \log(1-c)+c \right\}=\left[\log(1+c)-c\right]\overline{X}_n=:-\kappamle.
\end{align*}
We also need to make sure that
$\fisherinformation(\mle)>\deltamle M_2$, i.e. that
$\frac{1}{\overline{X}_n}>\frac{2 c}{(1-c)^3\overline{X}_n}$, which is true for all $0<c\leq 0.229$.

Now, the gamma prior with shape $\alpha$ and rate $\beta$ satisfies
$$\prior'(\param)=\frac{\beta^{\alpha}}{\Gamma(\alpha)}\left[(\alpha-1)\param^{\alpha-2}e^{-\beta\param}-\beta\param^{\alpha-1}e^{-\beta\param}\right].$$ Suppose that $\alpha<1$, then
\begin{align*}
	&\sup_{\param\in((1-c)\overline{X}_n, (1+c)\overline{X}_n)}|\prior'(\param)|\leq \frac{\beta^{\alpha}}{\Gamma(\alpha)}e^{-\beta(1-c)\overline{X}_n}\overline{X}_n^{\alpha-2}(1-c)^{\alpha-2}\left[ (1-\alpha)+\beta(1-c)\overline{X}_n\right]\\
	&\sup_{\theta\in((1-c)\overline{X}_n, (1+c)\overline{X}_n)}|\pi(\theta)|\leq \frac{\beta^{\alpha}}{\Gamma(\alpha)}(1-c)^{\alpha-1}\overline{X}_n^{\alpha-1}e^{-\beta(1-c)\overline{X}_n}\\
	&\sup_{\param\in((1-c)\overline{X}_n, (1+c)\overline{X}_n)}\frac{1}{|\prior(\theta)|}\leq \frac{\Gamma(\alpha)}{\beta^{\alpha}}(1+c)^{1-\alpha}\overline{X}_n^{1-\alpha}e^{\beta(1+c)\overline{X}_n}.
\end{align*}
Finally, the bounds in the MLE-centric approach are computed under the assumption $\sqrt{\frac{\Tr\left[\fisherinformation(\mle)^{-1}\right]}{n}}<\deltamle$. In our case, it says
    $\sqrt{\frac{\overline{X}_n}{n}}<c\overline{X}_n$, i.e. that $c>\frac{1}{\sqrt{n\overline{X}_n}}$.
    Therefore, assuming $\frac{1}{\sqrt{n\overline{X}_n}}<0.299$, letting $c\in\left(\frac{1}{\sqrt{n\overline{X}_n}},0.229\right]$, and assuming the shape $\alpha$ of the gamma prior is smaller than one, we can set
\begin{enumerate}[a)]
	\item $\deltamle=c\overline{X}_n$
	\item $M_1=(1-c)^{\alpha-2}(1+c)^{1-\alpha}\overline{X}_n^{-1}e^{2\beta c\overline{X}_n}\left[(1-\alpha)+\beta(1-c)\overline{X}_n\right]$
	\item $\widetilde{M}_1=\frac{\beta^{\alpha}}{\Gamma(\alpha)}(1-c)^{\alpha-1}\overline{X}_n^{\alpha-1}e^{-\beta(1-c)\overline{X}_n}$
	\item $\hat{M}_1= \frac{\Gamma(\alpha)}{\beta^{\alpha}}(1+c)^{1-\alpha}\overline{X}_n^{1-\alpha}e^{\beta(1+c)\overline{X}_n}$
	\item $M_2=\frac{2}{(1-c)^3(\overline{X}_n)^2}$
	\item $\fisherinformation(\mle)=\overline{X}_n^{-1}$
	\item $\kappamle=\left[c-\log(1+c)\right]\overline{X}_n$.
\end{enumerate}

The concrete choice of $c$ may be optimized numerically.

\subsubsection{The MAP-centric approach}\label{map_poisson}
Let us still assume $\alpha<1$. We note that
\begin{align*}
	&\logposterior(\param)=-n\param+n(\log\param)\overline{X}_n-\sum_{i=1}^n\log(X_i!)+\alpha\log(\beta)-\log(\Gamma(\alpha))+(\alpha-1)\log\param-\beta\param;\\
	&\logposterior'(\theta)=-n+\frac{n\overline{X}_n}{\theta}+\frac{(\alpha-1)}{\theta}-\beta=0\quad\text{iff}\quad \theta=\map:=\frac{n\overline{X}_n+(\alpha-1)}{(n+\beta)};\\
	&\logposterior''(\theta)=-\frac{n\overline{X}_n+\alpha-1}{\theta^2}\quad\text{and so}\quad \postfisherinformation(\map)=\frac{(n+\beta)^2}{n(n\overline{X}_n+\alpha-1)};\\
&\logposterior'''(\param)=2\frac{n\overline{X}_n+\alpha-1}{\theta^3}
\end{align*}
and so for $\bar{c}\in(0,1), \deltabar=\bar{c}\map
\text{ and }\param\in(\map-\deltabar,\map+\deltabar)=\left((1-\bar{c})\map,(1+\bar{c})\map\right)$, we have that
$$\frac{1}{n}|\logposterior'''(\theta)|\leq \frac{2(n+\beta)^3}{n(n\overline{X}_n+\alpha-1)^2(1-\bar{c})^3}=:\overline{M}_2.$$

Now, we require that $\postfisherinformation(\map)>\deltabar\,\overline{M}_2$, i.e. that
\begin{align*}
	\frac{(n+\beta)^2}{n(n\overline{X}_n+\alpha-1)}>\frac{2\bar{c}(n+\beta)^2}{n(n\overline{X}_n+\alpha-1)(1-\bar{c})^3},\quad\text{which holds for }\bar{c}\in(0,0.229],\text{ if }n\overline{X}_n>1-\alpha.
\end{align*}
We will also want to make sure that $\deltabar=\bar{c}\frac{n\overline{X}_n+(\alpha-1)}{n+\beta}\geq \|\map-\mle\|=\frac{\beta\overline{X}_n+1-\alpha}{n+\beta}$. Assuming that $n\overline{X}_n>1-\alpha$, this translates to $\bar{c}\geq \frac{1}{n}\cdot\frac{\beta\overline{X}_n+1-\alpha}{\overline{X}_n+(\alpha-1)/n}$. Moreover, we require that $\deltabar=\bar{c}\frac{n\overline{X}_n+(\alpha-1)}{n+\beta}> \sqrt{\frac{1}{n\postfisherinformation(\map)}}=\sqrt{\frac{n\overline{X}_n+\alpha-1}{(n+\beta)^2}}$, which is equivalent to saying that $\bar{c}>\sqrt{\frac{1}{n\overline{X}_n+\alpha-1}}$.

Finally, the value of $\kappabar$ may be obtained in the following way:
\begin{align*}
    \kappabar:=&-\max\bigg\{-\deltabar+\|\map-\mle\|+\mle\log\left(\frac{\mle+\deltabar-\|\map-\mle\|}{\mle}\right),\\
    &\hspace{3cm}\deltabar-\|\map-\mle\|+\mle\log\left(\frac{\mle-\deltabar+\|\map-\mle\|}{\mle}\right) \bigg\}\\
    =&-\overline{X}_n\bigg\{\frac{\beta-\bar{c}n+(1-\alpha)(1+\bar{c})/\overline{X}_n}{n+\beta}+\log\left(1+\frac{\bar{c}n-\beta-(1-\alpha)(1+\bar{c})/\overline{X}_n}{n+\beta}\right)\bigg\}.
\end{align*}

Therefore, in addition to the values we listed at the end of \Cref{map_poisson}, we have the following. We assume that $\alpha<1$, $n\overline{X}_n>1-\alpha$ and  $\max\left\{\sqrt{\frac{1}{n\overline{X}_n+\alpha-1}},\frac{1}{n}\cdot\frac{\beta\overline{X}_n+1-\alpha}{\overline{X}_n+(\alpha-1)/n}\right\}<0.229$. We let $\bar{c}\in\left(\frac{1}{n}\cdot\frac{\beta\overline{X}_n+1-\alpha}{\overline{X}_n+(\alpha-1)/n},0.229\right]\cap \left(\sqrt{\frac{1}{n\overline{X}_n+\alpha-1}},0.229\right]$. Then

\begin{enumerate}[i)]

	\item $\deltabar=\bar{c}\frac{n\overline{X}_n+(\alpha-1)}{(n+\beta)}$
	\item $\postfisherinformation(\map)=\frac{(n+\beta)^2}{n(n\overline{X}_n+\alpha-1)}$
	\item $\overline{M}_2=\frac{2\cdot (n+\beta)^3}{(1-\bar{c})^3n(n\overline{X}_n+\alpha-1)^2}$
	\item $\kappabar=-\overline{X}_n\bigg\{\frac{\beta-\bar{c}n+(1-\alpha)(1+\bar{c})/\overline{X}_n}{n+\beta}+\log\left(1+\frac{\bar{c}n-\beta-(1-\alpha)(1+\bar{c})/\overline{X}_n}{n+\beta}\right)\bigg\}$.
\end{enumerate}
	\subsection{Calculations for \Cref{weibull_example}: the MAP-centric approach}\label{calculations_weibull}
	Let $k$ be the shape of the Weibull and let $X_1,\dots,X_n\geq 0$ be our data.
Our log-likelihood is given by:
\[
\loglikelihood(\theta)=n\left[\log(k)-\log(\theta)\right]+(k-1)\sum_{i=1}^n\log(X_i)-\frac{\sum_{i=1}^n X_i^k}{\theta}
\]
\subsubsection{Calculating $\mle$ and $\fisherinformation(\mle)$}
Now
\begin{align*}
&\loglikelihood'(\param)=-\frac{n}{\theta}+\frac{\sum_{i=1}^nX_i^k}{\theta^{2}},\quad \loglikelihood''(\param)=\frac{n}{\theta^2}-\frac{2\sum_{i=1}^n X_i^k}{\theta^{3}}
\end{align*}
The MLE is $\mle=\frac{\sum_{i=1}^n X_i^k}{n}=:\overline{X^k}(n)$.
We have that
\[
\fisherinformation(\mle)=-\frac{1}{\mle^2}+\frac{2}{\mle^{2}}=\frac{1}{\mle^{2}}.
\]
\subsubsection{Calculating $M_2$}
Now, note that
\begin{align*}
& \loglikelihood'''(\theta)=\frac{6\sum_{i=1}^n X_i^k-2n\theta}{\theta^{4}},\quad		\loglikelihood^{(4)}(\theta) = \frac{6n\theta - 24 \sum_{i=1}^n X_i^k}{\theta^5},\quad \loglikelihood^{(5)}(\theta)=\frac{120\sum_{i=1}^n X_i^k - 24n\theta}{\theta^6}.
\end{align*}
Therefore,  for $0<\theta<2\mle$, $L_n'''$ is decreasing and positive.  This means that, if we let $0<\deltamle<\hat{\theta}_n$, then
\begin{align*}
	\sup_{\theta\in(\hat{\theta}_n-\deltamle,\hat{\theta}_n+\deltamle)}\frac{|L_n'''(\theta)|}{n}\leq \frac{L_n'''(\hat{\theta}_n-\deltamle)}{n}=:M_2.
\end{align*}

\subsubsection{Calculating $\hat{M}_1$}

Now, for a given shape $\alpha>0$ and scale $\beta>0$,
\[
\prior(\theta)=\frac{\beta^{\alpha}}{\Gamma(\alpha)}\left(\frac{1}{\theta}\right)^{\alpha+1}\exp\left(-\beta/\theta\right),\quad \pi'(\theta)=\frac{\beta^{\alpha}\exp\left(-\beta/\theta\right)}{\Gamma(\alpha)\theta^{\alpha+2}}\left(-\alpha-1+\frac{\beta}{\theta}\right)
\]
and it follows that, for $\theta>\frac{\beta}{\alpha+1}$, $\frac{1}{\pi(\theta)}$ is increasing and it is decreasing otherwise. Therefore:

\[
\sup_{\theta\in(\mle-\hat{\deltamle},\mle+\deltamle)}\left|\frac{1}{\pi(\theta)}\right|\leq
\max\left\{ \frac{1}{\pi(\hat{\theta}_n-\deltamle)},\frac{1}{\pi(\hat{\theta}_n+\deltamle)}\right\}=:\hat{M}_1.
\]

\subsubsection{Calculating $\map$ and $\postfisherinformation(\map)$}
Now
\begin{align*}
	\overline{L}_n(\param)=n\left[\log(k)-\log(\theta)\right]+(k-1)\sum_{i=1}^n\log(X_i)-\frac{\sum_{i=1}^n X_i^k}{\theta}-(\alpha+1)\log(\theta)-\frac{\beta}{\theta}.
\end{align*}
Therefore,
\begin{align*}
	&\overline{L}_n(\param)'=-\frac{n}{\theta}+\frac{\sum_{i=1}^n X_i^k}{\theta^2}-\frac{\alpha+1}{\theta}+\frac{\beta}{\theta^2}\\
	&\overline{L}_n(\param)''=\frac{n}{\theta^2}-\frac{2\sum_{i=1}^nX_i^k}{\theta^{3}}+\frac{\alpha+1}{\theta^2}-\frac{2\beta}{\theta^3}
\end{align*}
and the MAP $\bar{\theta}_n$ is given by
\begin{align*}
	(n+\alpha+1)\bar{\theta}_n=\beta+\sum_{i=1}^n X_i^k\quad\Leftrightarrow\quad \bar{\theta}_n=\frac{\beta+\sum_{i=1}^n X_i^k}{n+\alpha+1}.
\end{align*}
Moreover,
\begin{align*}
    \postfisherinformation(\map)=-\frac{1}{\map^2}+\frac{2\sum_{i=1}^nX_i^k}{n\map^3}-\frac{\alpha+1}{n\map^2}+\frac{2\beta}{n\map^3}.
\end{align*}
\subsubsection{Calculating $\bar{M}_2$}
Now, note that
\begin{align*}
\overline{L}_n'''(\param)=-\frac{2n}{\theta^3}+\frac{6\sum_{i=1}^n X_i^k}{\theta^{4}}-\frac{2(\alpha+1)}{\theta^3}+\frac{6\beta}{\theta^4},\quad \overline{L}_n^{(4)}(\param)=\frac{6n}{\theta^4}-\frac{24\sum_{i=1}^n X_i^k}{\theta^{5}}+\frac{6(\alpha+1)}{\theta^4}-\frac{24\beta}{\theta^5}.
\end{align*}
Therefore $\overline{L}'''_n$ is increasing if and only if
\begin{align*}
	6(n+\alpha+1)\theta>24\left(\beta+\sum_{i=1}^n X_i^k\right)\quad\Leftrightarrow \theta>4\map
\end{align*}
This means that, for $\deltabar\in\left(0,\map\right)$ and $\theta\in(\bar{\theta}_n-\deltabar,\bar{\theta}_n+\deltabar)$, $\overline{L}'''_n$ is decreasing and
\begin{align*}
	\sup_{|\theta-\map|<\deltabar}\frac{|\overline{L}_n'''(\theta)|}{n}\leq \frac{|\overline{L}_n'''(\map-\deltabar)|}{n}=:\overline{M}_2.
\end{align*}
\subsubsection{Calculating $\kappabar$}\label{kappa_bar_weibull}
Now, note that, for $\theta<\mle$, $L_n$, is increasing and otherwise it's decreasing. This means that,
\begin{align*}
	&\sup_{|\theta-\hat{\theta}_n|>\deltabar-|\hat{\theta}_n-\bar{\theta}_n|}\frac{L_n(\theta)-L_n(\mle)}{n}\\
	\leq& \max\left\{ \frac{L_n(\hat{\theta}_n-\deltabar+|\hat{\theta}_n-\bar{\theta}_n|)-L_n(\hat{\theta}_n)}{n},\frac{L_n(\hat{\theta}_n+\deltabar-|\hat{\theta}_n-\bar{\theta}_n|)-L_n(\hat{\theta}_n)}{n}\right\}=:-\kappabar.
\end{align*}
\subsubsection{Constraints on $\deltabar$ and $\deltamle$}
Finally, we need the derive the constraints on $\deltabar$. We first assume that $0<\deltabar<\bar{\theta}_n$. We also suppose that $\deltabar> \max\left(\|\bar{\theta}_n-\hat{\theta}_n\|,\frac{1}{\sqrt{n\postfisherinformation(\map)}}\right)$ and conditions on how large $n$ needs to be for this to hold are easy to obtain numerically. We also require $\deltabar<\frac{\lambda(\bar{\theta}_n)}{\overline{M}_2}$. Looking closer at this last condition, we require that:
\begin{align*}
	\deltabar<\minevaluemap\left(\bar{\theta}_n-\deltabar\right)^4\left[\left(-2-\frac{2\alpha+2}{n}	\right)\left(\bar{\theta}_n-\deltabar\right)+6\left(\frac{\beta}{n}+\hat{\theta}_n\right)\right]^{-1}
\end{align*}

Letting $0<\tilde{\delta}:=\map-\deltabar<\bar{\theta}_n$, we therefore require
\begin{align*}
	\tilde{\delta}+\minevaluemap\tilde{\delta}^4\left[\left(-2-\frac{2\alpha+2}{n}	\right)\tilde{\delta}+6\left(\frac{\beta}{n}+\hat{\theta}_n\right)\right]^{-1}>\bar{\theta}_n
\end{align*}
which is equivalent to:
\begin{align*}
	\tilde{\delta}+\tilde{\delta}^4\minevaluemap\left(2+\frac{2\alpha+2}{n}	\right)^{-1}\left[3\bar{\theta}_n-\tilde{\delta}\right]^{-1}>\bar{\theta}_n.
\end{align*}
This condition will be satisfied if
\begin{align}\label{ineq}
	\tilde{\delta}+\tilde{\delta}^4\minevaluemap\left(2+\frac{2\alpha+2}{n}	\right)^{-1}\left[3\bar{\theta}_n\right]^{-1}>\bar{\theta}_n.
\end{align}
The left-hand side of \cref{ineq} is increasing in $\tilde{\delta}$ and is clearly strictly greater than $\map$ for $\tilde{\delta}=\map$. This means that there exists a choice of $\deltabar$ that yields $\deltabar<\frac{\minevaluemap}{\overline{M}_1}$ and the set of such choices can be obtained by solving \cref{ineq} numerically. Finally, in order to make sure that the condition on $\deltamle$ is satisfied, we just need to check numerically how large $n$ needs to be so that $\deltamle>\frac{1}{\sqrt{n\fisherinformation(\mle)}}$.


\subsection{Calculations for \Cref{logistic_example}}\label{example_detail_logistic}
\subsubsection{Calculating $\fisherinformation(\mle)$ and $\postfisherinformation(\map)$}
Note that:
\begin{align*}
\fisherinformation(\mle)=&\frac{1}{n}\sum_{k=1}^n\frac{e^{X_k^T\mle Y_k}}{(1+e^{X_k^T\mle Y_k})^2}X_k\left(X_k\right)^T;\\
\postfisherinformation(\map)
=&\frac{1}{n}\sum_{k=1}^n\frac{e^{X_k^T\map Y_k}}{(1+e^{X_k^T\map Y_k})^2}X_k\left(X_k\right)^T\\
&+\frac{\nu+d}{2\nu n}\left[1+\frac{1}{\nu}(\map-\mu)^T\Sigma^{-1}(\map-\mu)\right]^{-1}\left(\Sigma^{-1}+\text{diag}\left(\Sigma^{-1}\right)\right)\\
& -\frac{\nu+d}{2\nu^2 n}\left[1+\frac{1}{\nu}(\map-\mu)^T\Sigma^{-1}(\map-\mu)\right]^{-2}\\
&\hspace{2cm}\cdot\left[\left(\Sigma^{-1}+\text{diag}\left(\Sigma^{-1}\right)\right)(\map-\mu)\right]\left[\left(\Sigma^{-1}+\text{diag}\left(\Sigma^{-1}\right)\right)(\map-\mu)\right]^T.
\end{align*}
\subsubsection{Calculating $M_1$, $\widetilde{M}_1$ and $\hat{M}_1$}
Note that
\begin{align}
 \prior'(\param)=&\frac{\Gamma((\nu+d)/2)}{\Gamma(\nu/2)\nu^{d/2}\pi^{d/2}|\Sigma|^{1/2}}\left(-\frac{\nu+d}{2}\right)\left[1+\frac{1}{\nu}(\param-\mu)^T\Sigma^{-1}(\param-\mu)\right]^{-(\nu+d)/2-1}\notag\\
	&\cdot\left[\frac{1}{\nu}\left(\Sigma^{-1}+\text{diag}\left(\Sigma^{-1}_{1,1},\dots,\Sigma^{-1}_{d,d}\right)\right)\right](\param-\mu).\label{pi_logistic_der}
\end{align}
It follows from the expressions in \cref{t_prior,pi_logistic_der} that
\begin{align*}
	\sup_{\theta:\|\theta-\hat{\theta}\|<\deltamle}|\pi(\theta)|\leq &\frac{\Gamma((\nu+d)/2)}{\Gamma(\nu/2)\nu^{d/2}\pi^{d/2}|\Sigma|^{1/2}}=:\widetilde{M_1};\\
	\sup_{\theta:\|\theta-\hat{\theta}\|<\deltamle}\frac{1}{|\pi(\theta)|}\leq& \sup_{\theta:\|\theta-\hat{\theta}\|<\deltamle}\frac{\Gamma(\nu/2)\nu^{d/2}\pi^{d/2}|\Sigma|^{1/2}}{\Gamma((\nu+d)/2)}  \left[1+\frac{1}{\nu\lambda_{\min}(\Sigma)}\|\theta-\mu\|^2\right]^{(\nu+d)/2} \\
	\leq &\frac{\Gamma(\nu/2)\nu^{d/2}\pi^{d/2}|\Sigma|^{1/2}}{\Gamma((\nu+d)/2)}  \left[1+\frac{2\deltamle^2+2\|\hat{\theta}-\mu\|^2}{\nu\lambda_{\min}(\Sigma)}\right]^{(\nu+d)/2} =:\hat{M}_1;\\
	\sup_{\theta:\|\theta-\hat{\theta}\|<\deltamle}\frac{\| \pi'(\theta)\|}{|\pi(\theta)|}=&\sup_{\theta:\|\theta-\hat{\theta}\|<\deltamle}\left(\frac{\nu+d}{2}\right)\frac{\left\|\frac{1}{\nu}\left(\Sigma^{-1}+\text{diag}\left(\Sigma^{-1}_{1,1},\dots,\Sigma^{-1}_{d,d}\right)\right)(\theta-\mu)\right\|}{\left[1+\frac{1}{\nu}(\theta-\mu)^T\Sigma^{-1}(\theta-\mu)\right]}\\
	\leq & \left(\frac{\nu+d}{2}\right)\frac{\left(\deltamle+\|\hat{\theta}-\mu\|\right)\left\|\Sigma^{-1}+\text{diag}\left(\Sigma^{-1}_{1,1},\dots,\Sigma^{-1}_{d,d}\right)\right\|}{\nu}\\
	\leq & \frac{\left(\nu+d\right)\left(\deltamle+\|\hat{\theta}-\mu\|\right)}{\nu\lambda_{\min}(\Sigma)}=:M_1.
\end{align*}
\subsubsection{Calculating $M_2$ and $\overline{M}_2$}
Note that, for all $\param\in\mathbbm{R}^d$,
\begin{align*}
&\loglikelihood'''\left(\param\right)[u_1,u_2,u_3]
=\sum_{i=1}^nY_i^3\sum_{j,k,l=1}^d \frac{e^{X_i^T\param Y_i}\left(e^{X_i^T\param Y_i}-1\right)}{\left(1+e^{X_i^T\param Y_i}\right)^3}X_i^{(j)}X_i^{(k)}X_i^{(l)}u_1^{(j)}u_1^{(k)}u_1^{(l)}
\end{align*}
and therefore, for all $\param\in\mathbbm{R}^d$,
\begin{align}
	\frac{1}{n}\|\loglikelihood'''\left(\param\right)\|\leq&\frac{1}{n}\sum_{k=1}^n\|X_k\|^3\frac{e^{X_k^T\param Y_k}\left|e^{X_k^T\param Y_k}-1\right|}{\left(1+e^{X_k^T\param Y_k}\right)^3}
	\leq\frac{1}{6\sqrt{3}\,n}\sum_{k=1}^n\|X_k\|^3=:M_2.\label{m2_logistic}
\end{align}
Now, a straightforward calculation reveals that, for $\|u\|\leq 1, \|v\|\leq 1, \|w\|\leq 1$ and any $\param\in\mathbbm{R}^d$, and $\deltabar\leq 1$,
\begin{align*}
	&\sup_{\|\param-\map\|\leq\deltabar}\left|\sum_{i,j,k=1}^d\left(\frac{\partial^3}{\partial \theta_j\partial\theta_i\partial\theta_k}\log\pi(\theta)\right)u_iv_jw_k\right|\\
	\leq &\sup_{\|\param-\map\|\leq\deltabar}\Bigg\{\frac{3(\nu+d)}{\left[\nu+(\theta-\mu)^T\Sigma^{-1}(\theta-\mu)\right]^2}\left\|\Sigma^{-1}+\text{diag}(\Sigma^{-1})\right\|_{op}^2\|\theta-\mu\|\\
	&+\frac{2(\nu+d)}{\left[\nu+(\theta-\mu)^T\Sigma^{-1}(\theta-\mu)\right]^3}\left\|\Sigma^{-1}+\text{diag}(\Sigma^{-1})\right\|_{op}^3\|\theta-\mu\|^3\Bigg\}\\
	\leq &\frac{3(\nu+d)}{\nu^2}\left\|\Sigma^{-1}+\text{diag}(\Sigma^{-1})\right\|_{op}^2\left(1+\|\map-\mu\|\right)\\
	&+\frac{2(\nu+d)}{\nu^2}\left\|\Sigma^{-1}+\text{diag}(\Sigma^{-1})\right\|_{op}^3\left(1+\|\map-\mu\|\right)^3.
\end{align*}
Combining this with \cref{m2_logistic}, we obtain
\begin{align}
	\frac{1}{n}\left\|\logposterior'''(\map)\right\|\leq &\frac{1}{6\sqrt{3}\,n}\sum_{k=1}^n\|X_k\|^3 + \frac{3(\nu+d)}{\nu^2\,n}\left\|\Sigma^{-1}+\text{diag}(\Sigma^{-1})\right\|_{op}^2\left(1+\|\map-\mu\|\right)\notag\\
	&+\frac{2(\nu+d)}{\nu^3\,n}\left\|\Sigma^{-1}+\text{diag}(\Sigma^{-1})\right\|_{op}^3\left(1+\|\map-\mu\|\right)^3=:\overline{M}_2.\label{m2bar_logistic}
\end{align}
\subsubsection{Calculating $\kappamle$ and $\kappabar$} \label{calculating_kappa}
Note that $\loglikelihood$ is strictly concave. Therefore,
	\begin{align*}
		\sup_{\param:\|\param-\mle\|>\deltamle}\frac{L_n(\param)-L_n(\hat{\theta})}{n}\leq& \sup_{\param:\|\param-\mle\|=\deltamle}\frac{L_n(\theta)-L_n(\hat{\theta})}{n}\\
		\leq
		&\sup_{\param:\|\param-\mle\|=\deltamle}\left\{-\frac{1}{2}\left(\param-\mle\right)^T\fisherinformation(\mle)\left(\param-\mle\right)\right\}+\frac{M_2\deltamle^3}{2}\\
		\leq
		&-\frac{1}{2}\minevaluemle\deltamle^2+\frac{M_2\deltamle^3}{2}=:-\kappamle.
	\end{align*}
Since $M_2$ in \cref{m2_logistic} provides a uniform bound on $	\frac{1}{n}\|\loglikelihood'''\left(\param\right)\|$ over $\param\in\mathbbm{R}^d$, a similar calculation shows:
	\begin{align*}
	&\sup_{\param:\|\param-\mle\|>\deltabar-\|\mle-\map\|}\frac{L_n(\param)-L_n(\mle)}{n}\\
	\leq
	&-\frac{1}{2}\minevaluemle\left(\deltabar-\|\mle-\map\|\right)^2+\frac{M_2\left(\deltabar-\|\mle-\map\|\right)^3}{2}=:-\kappabar.
\end{align*}

\subsubsection{Finding appropriate values of $\deltamle$ and $\deltabar$}
	In order to apply our results in the MLE-centric approach, we need
	\begin{align}\label{delta_bound}
\sqrt{\frac{\Tr\left[\fisherinformation(\mle)^{-1}\right]}{n}}<\deltamle<\frac{\minevaluemle}{M_2}.
\end{align}
 This assumption also ensures that $\kappamle$ from \Cref{calculating_kappa} is positive. In the MAP-centric approach we also require
 \begin{align*} \max\left\{\|\hat{\theta}_n-\bar{\theta}_n\|,\sqrt{\frac{\Tr\left(\postfisherinformation(\map)^{-1}\right)}{n}}\right\}< \deltabar<\frac{\minevaluemap}{\overline{M}_2},
\end{align*}
 which also ensures that $\kappabar$ from \Cref{calculating_kappa} is positive.
Such choices of $\deltabar$ and $\deltamle$ will be available for sufficiently large $n$. In order to choose the appropriate concrete values of $\deltabar$ and $\deltamle$ one can run a numerical optimization scheme.

\end{appendix}

\vskip 0.2in
\bibliography{Bibliography}

\end{document}